\documentclass{amsart}
\usepackage{etex}
\usepackage{mathtools,tensor}
\usepackage{latexsym,amssymb,amsthm,amsmath,enumerate,epsfig,setspace,bbding,color,graphicx,caption}
\usepackage{stmaryrd}
\usepackage[all]{xy}
\usepackage[hidelinks]{hyperref}
\usepackage{subfiles}
\usepackage[margin=10pt,font=small,labelfont=bf]{caption}
\usepackage{tikz-cd}
\usetikzlibrary{arrows}
\usepackage{verbatim}
\usepackage{amsmath}
\usepackage{amssymb}
\usepackage{amsfonts}
\usepackage{enumitem}
\usepackage{todonotes}
\usepackage{mathtools}
\usepackage{tikz}
\usepackage[utf8]{inputenc}
\usepackage{textcomp}

\numberwithin{equation}{section}

\newcommand{\extg}{\mathfrak{a}} % the extended isotropy Lie algebra

\newcommand{\BB}{\mathbf{X}}
\newcommand{\iM}{\mathbb{M}}
\newcommand{\iN}{\mathbb{N}}
\newcommand{\iE}{\mathbb{E}}
\newcommand{\iP}{\mathbb{P}}
\newcommand{\iC}{\mathcal{C}}

\newcommand{\TT}{\mathfrak{T}}

\newcommand{\tto}{\rightrightarrows}    % Arrows of a groupoid
\newcommand{\Ger}{\mathcal{G}erm}
\newcommand{\CC}{\mathcal{C}} % for the (Cartan) connections
\newcommand{\ilim}{\varinjlim} 
\newcommand{\plim}{\varprojlim}

\newtheorem{thm}{Theorem}[section]
\newtheorem{lem}[thm]{Lemma}
\newtheorem{prp}[thm]{Proposition}
\newtheorem{cor}[thm]{Corollary}

\theoremstyle{definition}
\newtheorem{defn}[thm]{Definition}
\newtheorem{exm}[thm]{Example}

\newtheorem{rmk}[thm]{Remark}

\pgfmathsetmacro{\shift}{0.65ex}

\newcommand{\G}{\mathcal{G}}

\newcommand{\F}{\mathcal{F}}

\newcommand{\R}{\mathbb{R}}

\newcommand{\D}{\mathcal{D}}
\renewcommand{\P}{\mathcal{P}}

\renewcommand{\dim}{\text{dim}}
\renewcommand{\ker}{\text{Ker}}

\newcommand{\dom}{\text{Dom}}

\renewcommand\phi{\varphi}

\begin{document}
%%%%%%%%%%%%%%%%%%%%%%%%%%%%%%%%%%%%%%%%%%
\title{Haefliger's differentiable cohomology}

\author{Luca Accornero}
\address{Croeselaan 1, 3521 BJ, Utrecht,
The Netherlands}
\email{luca.accornero.nl@gmail.com}

% author one information
\author{Marius Crainic}
\address{Depart. of Math., Utrecht University, 3508 TA Utrecht,
The Netherlands}
\email{m.crainic@uu.nl}

\thanks{This research is supported by the NWO Vici Grant no. 639.033.312 and by 
an NWO PhD grant with the Utrecht Geometry Center}

\begin{abstract}
We review Haefliger's differentiable cohomology for the pseudogroup of diffeomorphisms of $\mathbb{R}^q$~\cite{HAEFLIGER}. 
We unravel the structure that governs such cohomologies, which, remarkably, is related to the so called Cartan distribution underlying the geometric study of PDEs. 
Hence, we extend Haefliger's differentiable cohomology to the general framwork of flat Cartan groupoids, we investigate its infinitesimal counterpart and relate the two by a van Est-like map. Finally, we define a characteristic map for geometric structures on manifolds $M$ associated to flat Cartan groupoids. The outcome generalizes the existing approaches to characteristic classes for foliations~\cite{HAEFLIGER, BOTTHAEFLIGER, BERNSTEINROSENFELD, BERNSTEINROSENFELD2}. The motivation for this work is two-fold. 
On the one hand, it is motivated by the recent approach to geometric structures via multiplicative (Cartan) distributions~\cite{MARIA,ORI,FRANCESCO}; from that perspective, we are constructing characteristic classes for such structures. On the other hand, it is motivated by our (ongoing) attempt to turn classical symmetries (pseudogroups) into non-commutative, Hopf-algebraic, ones; such attempt is inspired by existing work in non-commutative geometry~\cite{Mosc1, MoscR1, MoscR2} and aims at a unified approach which allows also for non-transitive pseudogroups.

%On the one hand, it is motivated by the recent approach to geometric structures via multiplicative (Cartan) distributions; from that perspective, we are constructing characteristic classes for such structures. On the other hand, it is motivated by our (ongoing) attempt to turn classical symmetries (pseudogroups) into non-commutative, Hopf-algebraic, ones, via a more unified approach which allows also for non-transitive pseudogroups. 
%
%
%
%We review Haefliger's differentiable cohomology for the pseudogroup of diffeomorphisms of $\mathbb{R}^q$~\cite{HAEFLIGER}. We investigate the structure needed to define such a cohomology, which, remarkably, is related to the so called Cartan distribution underlying the geometric study of PDE. We define an analogue of Haefliger differentiable cohomology for flat Cartan groupoids, investigate its infinitesimal counterpart and relate the two by a van Est-like map. Finally, we define a characteristic map for geometric structures on manifolds $M$ associated to flat Cartan groupoids. The outcome generalizes the existing approaches to characteristic classes for foliations~\cite{HAEFLIGER, BOTTHAEFLIGER, BERNSTEINROSENFELD, BERNSTEINROSENFELD2}.
\end{abstract}
\dedicatory{Dedicated to André Haefliger on his 90th birthday.}
\maketitle

\tableofcontents

\section{Introduction}
%%%%%%%%%%%%%%%%%%%%%%%%%%%%
%%%%%%%%%%%%%%%%%%%%%%%%%%%%
%%%%%%%%%%%%%%%%%%%%%%%%%%%%
%%%%%%%%%%%%%%%%%%%%%%%%%%%%
%%%%%%%%%%%%%%%%%%%%%%%%%%%%
%%%%%%%%%%%%%%%%%%%%%%%%%%%%
%%%%%%%%%%%%%%%%%%%%%%%%%%%%
%%%%%%%%%%%%%%%%%%%%%%%%%%%%
%%%%%%%%%%%%%%%%%%%%%%%%%%%%
%%%%%%%%%%%%%%%%%%%%%%%%%%%%

Various geometric structures on manifolds $M$ come with invariants that live in the cohomology of the ambient space $M$, and which are organised into a map from a certain ``universal space" (associated to the type of structure one is looking at) to the cohomology of $M$; such maps are called generically ``characteristic maps". For instance, for vector bundles one has Pontryagin/Chern classes and the Chern-Weil homomorphisms. To help the reader, but also to put things into perspective, in the first section of this paper we overview some of the standard characteristic maps (for principal bundles, for flat ones, for foliations and then for more general $\Gamma$-structures). 

Haefliger's differentiable cohomology arose in the development of characteristic classes for foliations $\F$ in analogy with the one for flat principal bundles. The theory had produced:
\begin{itemize}
\item explicit/geometric characteristic maps $\kappa^{\F}$, in the spirit of Chern-Weil theory, defined on certain Lie algebra cohomologies % $H^*(\mathfrak{a}_q, O_q)$ 
that can be computed explicitly (subsection \ref{Geometric characteristic classes for foliations}), 
\item abstract characteristic maps $\kappa^{\F}_{\textrm{abs}}$ built via classifying spaces and maps, defined on a rather huge and complicated cohomology $H^{*}(B\Gamma^q)$  (subsection \ref{Abstract characteristic classes for foliations}),
\item a universal characteristic map $\kappa^{\textrm{univ}}$ (subsection \ref{ssec:Motivation and definition})
\end{itemize} 
fitting together in a commutative diagram
\begin{equation}
\begin{tikzcd}
H^*(\mathfrak{a}_q, O_q)\arrow[rr]{}{\kappa^{\F}} \arrow[d, swap]{}{\kappa^{\textrm{univ}}}& & H^*(M).\\
H^*(B\Gamma^q)\arrow[rru, swap]{}{\kappa^{\F}_{abs}}& & 
\end{tikzcd}
\end{equation}
As the name indicates, $\kappa^{\F}_{\textrm{abs}}$ was constructed rather abstractly and even its domain $H^{*}(B\Gamma^q)$ is rather huge and unmanageable. There have been several more explicit descriptions of this cohomology, and of the abstract characteristic map, most notably:
\begin{itemize}
\item as a ``DeRham-like" cohomology based on the Bott-Shulman complex (subsection \ref{More on cohomology of classifying spaces and the abstract characteristic map}),
\item via sheaf cohomology and then, based on explicit bar-type resolutions, via group-like cochains on $\Gamma^q$, the groupoid of germs of diffeomorphisms of $\mathbb{R}^q$. This is in complete analogy with cohomology of discrete groups which, themselves, play a similar role in characteristic maps for flat principal $G$-bundles. 
\end{itemize}
While everything was very similar to the case of flat principal bundles, there was still a question left: 
whether one could make sense, inside the complex computing the cohomology $H^*(B\Gamma^q)$,  of a ``differentiable complex of $\Gamma^q$" giving rise to a ``differentiable cohomology"
\[ H^{*}_{\textrm{diff}}(\Gamma^q)\]
and whether one could prove a ``van Est isomorphism" between this cohomology and $H^*(\mathfrak{a}_q, O_q)$, so that the previous diagram becomes 
\begin{equation}
\begin{tikzcd}
H^*_{\textrm{diff}}(\Gamma^q)\simeq H^*(\mathfrak{a}_q, O_q)\arrow[rr]{}{\kappa^{\F}} \arrow[d, swap]{}& & H^\bullet(M).\\
H^*(\Gamma^q)\simeq H^*(B\Gamma^q)\arrow[rru, swap]{}{\kappa^{\F}_{abs}}& & 
\end{tikzcd}
\end{equation}
Of course, the vertical map would be simply the one induced by the inclusion of differentiable cocycles into the general ones. This is precisely the task undertaken by Haefliger in \cite{HAEFLIGER}. 
%
%Given the fact that he was concentrating on the special case of $\Gamma^q$, and perhaps also because of the apparently ad-hoc nature of the construction, his construction of $H^*_{\textrm{diff}}(\Gamma^q)$ (and of his van Est-like isomorphism) did not receive much attention and seems to be forgotten. One of the aims of this paper is to explain and provide more insight into Haefliger's differentiable cohomology. 
While Gelfand-Fuchs cohomology $H^*(\mathfrak{a}_q, O_q)$ continued to play a central role in the discussion of characteristic classes for foliations, Haefliger's differentiable cohomology remained a rather ad-hoc construction, which seems to work primarily for infinite jet groupoids. On the other hand,
while it received little attention within the classical approach to characteristic classes, it was very nicely used in Noncommutative Geometry in order to relate Hopf-cyclic cohomology with Gelfand-Fuchs cohomology \cite{Mosc2}. One of the aims of this paper is to explain and provide more insight into Haefliger's 
differentiable cohomology, by providing a conceptual and more general framework, with an eye on questions in Noncommutative Geometry (see below).

As it was/is interesting to generalise such constructions from foliations to general $\Gamma$-structures (also called $\Gamma$-foliations), we will proceed in this generality. But it is fair to say that everything was in place in this generality already during Haefliger's work mentioned above. The only little problems were that 
\begin{itemize}
\item one usually assumed that $\Gamma$ was transitive,
\item while in principle Haefliger's differentiable cohomology $H^*_{\rm diff}(\Gamma^q)$ could just be copied for any pseudogroup $\Gamma$, the situation is actually a bit more subtle (and, as we shall point out, it is for {\it Lie} pseudogroups $\Gamma$ that everything works). 
\end{itemize}
Furthermore, looking at Haefliger's definition, it is clear right away that the outcome does not depend on $\Gamma$, but rather on the associated infinite jet groupoid $J^\infty \Gamma$ (and certain structure on it). 
One of the main driving questions behind this paper was to find the structure on $J^{\infty}\Gamma$ that makes Haefliger's differentiable cohomology work. The outcome is that of ``flat Cartan groupoid $(\Sigma, \CC)$", whose resulting cohomology we called Haefliger cohomology
\[ H^{*}_{\textrm{Haef}}(\Sigma, \CC).\]
For $\Sigma= J^\infty\Gamma$ we find out that the extra-structure that we were after is precisely the well-known Cartan distribution $\CC$ on infinite jet spaces plus the (less known) understanding of its compatibility with the groupoid structure; and, of course,
\[ H^{*}_{\textrm{Haef}}(J^\infty\Gamma, \CC)= H^{*}_{\textrm{diff}}(\Gamma). \]

% 
%which gives back Haefliger's cohomology $\Sigma= J^\infty \Gamma$ endowed with the structure that really matter  $\CC$ is the well-known Cartan distribution on infinite jet spaces). 
%
%
%$H^{\bullet}_{\textrm{diff}}(\Gamma)$ when $\Sigma= J^\infty \Gamma$ (... and $\CC$ is the well-known Cartan distribution on infinite jet spaces). 

We must add here that the notion of Cartan groupoid (and variations) came already to our attention from our study 
of Lie pseudogroups from the point of view of PDEs, and of the corresponding geometric structures~\cite{MARIA, ORI, FRANCESCO}. The fact that such different aspects point out to the same structure may seem quite surprising/remarkable at first;
at a second thought however, they just show that it is always the same structure that make the jet spaces interesting and useful over and over again. This rather philosophical insight is actually used several times inside the paper.
For instance, the theory of almost geometric structures reveals the notion of flat ``principal Cartan groupoid-bundle'' or, in the terminology of Section~\ref{Generalization: Haefliger cohomology}, of flat $(\Sigma, \CC)$-structure. We use that notion to point out that some of the characteristic maps associated to geometric structures depend only on the underlying almost structures. In turn, this may have interesting applications to the study of almost structures themselves: as pointed out in Remark \ref{int}, the resulting Haefliger characteristic maps may be used to detect almost $\Gamma$-structures that are not integrable.

At this point we would like to mention a second motivation for the present work: strengthen the bridge between classical symmetries (Lie pseudogroups) and non-commutative symmetries (encoded in Hopf-like structures). The aim would be to extend the previous work in this direction \cite{Mosc1, MoscR1, MoscR2} to general (possibly non-transitive) Lie pseudogroups or even to flat Cartan groupoids (in order to relate it to the unifying approach to geometric structures from \cite{MARIA, FRANCESCO}). One step (that we plan to  undertake in future work) is to construct the correct Hopf-like objects, extending the constructions from \cite{Mosc1, MoscR1}.
The second step would be to investigate the resulting Hopf-cyclic cohomologies, while a third one would concern characteristic maps; this paper should provide the 
necessary ``classical aspects'' for carrying out the second and third steps. This is one of the motivations for investigating the van Est map in this generality (notice that, when one moves away from the transitive case, there is no direct analogue of $H^*(\mathfrak{a}_q, O_q)$). It is also fair to say that our approach to characteristic maps associated to geometric structures is inspired by already existing non-commutative theory \cite{CONMOSC, Mosc1, CraiH}.

Here is a brief outline of the paper. In Section~\ref{Overview of varous characteristic classes}, we briefly review the connection-curvature construction and the classifying space approach to characteristic classes for principal bundles and flat principal bundles, and the similar machinery for characteristic classes for foliations. We introduce pseudogroups and recall the equivalence between the category of pseudogroups and the category of effective étale groupoids. We review the Bott-Shulman model and the sheaf theoretical approach to the cohomology of the classifying space of an étale groupoid. In Section~\ref{Haefliger's differentiable cohomology}, we adapt Haefliger's definition of $H^*_{\rm diff}(\Gamma^q)$ to a general Lie pseudogroup $\Gamma$. We proceed to investigate the minimal structure that allows to construct an analogous cohomology on a general Lie groupoid $\Sigma \tto \BB$, discovering flat Cartan connections; we then briefly review the canonical flat Cartan connection on $J^\infty\Gamma$. In Section~\ref{Generalization: Haefliger cohomology} we depart from Lie pseudogroup to work with Lie groupoids equipped with a flat Cartan connection, introducing what we called ``Haefliger cohomology'' and discussing some examples. We also study flat principal bundles in this context, which include ``almost geometric structures'' as a particular example. In Section~\ref{The infinitesimal version of Haefliger cohomology} we present the infinitesimal picture of the theory discussed in Sections~\ref{Haefliger's differentiable cohomology} and~\ref{Generalization: Haefliger cohomology}. Finally, in Section~\ref{Van Est maps}, we discuss a van Est theorem for groupoids equipped with a flat Cartan connection - generalizing a similar result proved by Haefliger for ${\rm Diff}_{\rm loc}(\mathbb{R}^q)$. We conclude presenting a characteristic map for ``flat principal actions'' of groupoids with a flat Cartan connection. Specializing to $J^\infty\Gamma^q$ with its canonical Cartan connection, we recover Haefliger's approach to characteristic classes of foliations. Throughout the paper, we need to work with pro-finite (dimensional) manifolds and pro-finite Lie groupoids. Since some of the properties that we make use of are not explicitly spelled out in the literature, we provide a detailed appendix.

\medskip

{\bf Acknowledgments.} 
The main results of this paper were first presented at the Conference in honor of Prof. André Haefliger's 90th birthday
held in Geneva in May 2019. We would like to thank A. Alekseev and A. Haefliger for that opportunity. Special thanks are due to Francesco Cattafi and Luca Vitagliano for their valuable feedback. L.A. also wishes to thank Francesco Cattafi, Maarten Mol, Lauran Toussaint and Aldo Witte for all the insightful discussions that helped shaping this paper. Furthermore, 
M.C. is grateful to André Haefliger for the many inspiring interactions and support throughout the years; and for a rare original copy 
of \cite{HAEFLIGER}.

\newpage

%%%%%%%%%%%%%%%%%%%%%%%%%%%%
%%%%%%%%%%%%%%%%%%%%%%%%%%%%
%%%%%%%%%%%%%%%%%%%%%%%%%%%%
%%%%%%%%%%%%%%%%%%%%%%%%%%%%
%%%%%%%%%%%%%%%%%%%%%%%%%%%%
%%%%%%%%%%%%%%%%%%%%%%%%%%%%
%%%%%%%%%%%%%%%%%%%%%%%%%%%%
%%%%%%%%%%%%%%%%%%%%%%%%%%%%
%%%%%%%%%%%%%%%%%%%%%%%%%%%%
%%%%%%%%%%%%%%%%%%%%%%%%%%%%
\section{Overview of various characteristic classes}\label{Overview of varous characteristic classes}
%%%%%%%%%%%%%%%%%%%%%%%%%%%%
%%%%%%%%%%%%%%%%%%%%%%%%%%%%
%%%%%%%%%%%%%%%%%%%%%%%%%%%%
%%%%%%%%%%%%%%%%%%%%%%%%%%%%
%%%%%%%%%%%%%%%%%%%%%%%%%%%%
%%%%%%%%%%%%%%%%%%%%%%%%%%%%
%%%%%%%%%%%%%%%%%%%%%%%%%%%%
%%%%%%%%%%%%%%%%%%%%%%%%%%%%
%%%%%%%%%%%%%%%%%%%%%%%%%%%%
%%%%%%%%%%%%%%%%%%%%%%%%%%%%

%%%%%%%%%%%%%%%%%%%%%%%%%%%%
%%%%%%%%%%%%%%%%%%%%%%%%%%%%
%%%%%%%%%%%%%%%%%%%%%%%%%%%%
%%%%%%%%%%%%%%%%%%%%%%%%%%%%
%%%%%%%%%%%%%%%%%%%%%%%%%%%%
\subsection{Geometric/abstract characteristic classes for principal bundles}
\label{Geometric/abstract characteristic classes for principal bundles}
%%%%%%%%%%%%%%%%%%%%%%%%%%%%
%%%%%%%%%%%%%%%%%%%%%%%%%%%%
%%%%%%%%%%%%%%%%%%%%%%%%%%%%
%%%%%%%%%%%%%%%%%%%%%%%%%%%%
%%%%%%%%%%%%%%%%%%%%%%%%%%%%
Characteristic classes for principal $G$-bundles $P\to M$ ($G$ being a Lie group and $P$, $M$ smooth manifolds) are usually introduced either:
\begin{itemize}
\item[(geom):] more explicitly/geometrically: one uses connections to produce explicit cohomology classes on the base $M$ of the principal $G$-bundle. Known as the Chern-Weil theory, it produces the so-called Chern-Weil map
\[ \kappa^P: \textrm{Inv}(\mathfrak{g}) \to H^*(M)\]
defined on the space of $G$-invariant polynomials $\textrm{Inv}(\mathfrak{g})= S(\mathfrak{g}^*)^{G}$ on the Lie algebra $\mathfrak{g}$ of $G$. 
\item[(abs):] more abstractly/topologically: this is based on the existence of a certain principal $G$-bundle $EG\to BG$ which is universal in the sense that any principal $G$-bundle $P\to M$ arises as the pull-back of $EG$ via a classifying map $M\to BG$. This map is unique up to homotopy, hence one has a uniquely defined characteristic map in cohomology: 
\begin{equation}\label{eq:abstract_char_map_pb}
 \kappa_{\textrm{abs}}^{P}: H^*(BG)\to H^*(M).
\end{equation}
\end{itemize}

In the more abstract approach one concentrates on the classifying space $BG$ and its cohomology and one would like to have more explicit models. While $BG$ is unique up to homotopy, there are various known models for it; however, in general, $BG$ still remains rather abstract. Of course, its cohomology may be simpler and one hopes for more explicit models for realising $H^*(BG)$ and $\kappa_{\textrm{abs}}^{P}$. Of course, the Chern-Weil map $\kappa^{P}$ fits this scheme. First of all one has a natural map
\[ \kappa^{\textrm{univ}}: \textrm{Inv}(\mathfrak{g})\to H^*(BG)\]
that can be thought of as the Chern-Weil map associated to the universal bundle $EG$, and all the maps fit in a commutative diagram: 
%\[
%\begin{tikzcd}
%H^\bullet(BG)\arrow[rr]{}{\kappa_{P}^{abs}}& & H^\bullet (M) \\
%&\textrm{Inv}(\mathfrak{g}) \arrow[lu]{}{\kappa_{\textrm{univ}}}\arrow[ru, swap]{}{\kappa_{P}}&  
%\end{tikzcd}
%\]
%\[
%\begin{tikzcd}
%\textrm{Inv}(\mathfrak{g})\arrow[d, swap]{}{\kappa_{\textrm{univ}}} \arrow[rrd]{}{\kappa_{P}}& & \\
%H^\bullet(BG)\arrow[rr, swap]{}{\kappa_{P}^{abs}}& & H^\bullet (M) 
%\end{tikzcd}
%\]
\begin{equation}\label{char-diag-G-bundles}
\begin{tikzcd}
\textrm{Inv}(\mathfrak{g})\arrow[rr]{}{\kappa^{P}} \arrow[d, swap]{}{\kappa^{\textrm{univ}}}& & H^*(M)\\
H^*(BG)\arrow[rru, swap]{}{\kappa^{P}_{\rm abs}}& & 
\end{tikzcd}
\end{equation}

Furthermore, it is known (see e.g. \cite{DUPONT}) that, for compact Lie groups $G$, $\kappa^{\textrm{univ}}$ is an isomorphism and, therefore, the two characteristic maps $\kappa^P$ and $\kappa_{\textrm{abs}}^{P}$ become identified:
%. Furthermore, in this case by restricting to a maximal torus $T\subset G$ (with Lie algebra $\mathfrak{t}\subset \mathfrak{g}$), $\textrm{Inv}(\mathfrak{g})$ becomes isomorphic with the algebra of polynomials on $\mathfrak{t}$ that are invariant under the action of the corresponding Weyl group $W$. Hence
\[ H^*(BG)\cong \textrm{Inv}(\mathfrak{g}).\]
%\cong S(\mathfrak{t}^*)^W.
When $G$ is compact and connected, by restricting to a maximal torus $T\subset G$ (with Lie algebra $\mathfrak{t}\subset \mathfrak{g}$), $\textrm{Inv}(\mathfrak{g})$ becomes isomorphic with the algebra of polynomials on $\mathfrak{t}$ that are invariant under the action of the corresponding Weyl group $W$, i.e. $H^*(BG)\cong S(\mathfrak{t}^*)^W$. More in general, when $G$ is compact with connected component of the identity denoted by $G_0$, 
%one has ${\rm Inv}(\g)\cong {\rm Inv}_{G/G_0}( S(\mathfrak{t}^*)^W)$ 
one has an action of $G/G_0$ on $S(\mathfrak{t}^*)^W$ and
\[
H^*(BG)\cong  {\rm Inv}_{G/G_0}( S(\mathfrak{t}^*)^W).
\]
When $G$ is non-compact and has finitely many connected components, the cohomology of $BG$ is canonically isomorphic to the cohomology of $BK$, the classifying space of the maximal compact subgroup $K$ of $G$. Therefore, for Lie groups, the rather abstract $H^*(BG)$ can often be made quite explicit. 

\begin{exm} When looking at complex vector bundles $E\to M$ interpreted as principal $GL_n(\mathbb{C})$-bundles $Fr(E)\to M$, $K= U(n)$ with maximal torus $T^n$ (diagonal matrices), and the Weyl group $W= S_n$ permuting the entries, 
we have
\[ H^*(BGL_n(\mathbb{C}))\cong H^*(BU(n))= \mathbb{R}[x_1, \ldots, x_n]^{S_n}= \mathbb{R}[c_1, \ldots, c_n],\]
%= \mathbb{R}[x_1, \ldots, x_n]^{S_n}
where the $x_i$'s are degree one variables and the $c_i$'s are the fundamental symmetric polynomials on  the $x_i$'s. Of course, for a rank $n$ complex vector bundle $E\to M$, 
\[ c_i(E)= \kappa^{Fr(E)}_{\textrm{abs}}(c_i)\in H^{2i}(M)\]
are the usual Chern classes of $E$.
\end{exm}

\begin{exm}
A similar discussion applies to the Pontryagin classes of real vector bundles $F\to M$,
\[ p_j(F)\in H^{4j}(M), \quad (0\leq j\leq k:= \left[\frac{n}{2}\right], \ n= \textrm{rank}(F)),\]
which are usually defined as the even Chern classes $c_{2j}(F_{\mathbb{C}})$ of the complexification 
$F_{\mathbb{C}}$ of $F$ (with the footnote that the odd classes vanish). The relevant group is $G= GL_n(\mathbb{R})$ and 
\[ H^*(BGL_n(\mathbb{R}))= 
H^*(BO(n))= \mathbb{R}[y_1, \ldots, y_k]^{S_n}= \mathbb{R}[p_1, \ldots, p_k],\]
%= \mathbb{R}[y_1, \ldots, y_k]^{S_n}
where $k= \left[ \frac{n}{2}\right]$ and $y_j$ has degree $2j$. 
%y_j, 4j?
Of course, with the mind at the definition $p_j(F)= c_{2j}(F_{\mathbb{C}})$, one may think that $y_j= x_{j}^{2}$. 
\end{exm}

%\vspace*{.3in}

%%%%%%%%%%%%%%%%%%%%%%%%%%%%
%%%%%%%%%%%%%%%%%%%%%%%%%%%%
%%%%%%%%%%%%%%%%%%%%%%%%%%%%
%%%%%%%%%%%%%%%%%%%%%%%%%%%%
%%%%%%%%%%%%%%%%%%%%%%%%%%%%
\subsection{Geometric/abstract characteristic classes for flat bundles} 
\label{Geometric/abstract characteristic classes for flat bundles} 
%%%%%%%%%%%%%%%%%%%%%%%%%%%%
%%%%%%%%%%%%%%%%%%%%%%%%%%%%
%%%%%%%%%%%%%%%%%%%%%%%%%%%%
%%%%%%%%%%%%%%%%%%%%%%%%%%%%
%%%%%%%%%%%%%%%%%%%%%%%%%%%%

The previous discussion serves as guideline in various other contexts. For instance, a related (and actually simpler) discussion is that of characteristic classes for flat $G$-bundles $P\to M$:
\begin{itemize}
\item[(geom):] more explicitly: the flatness is encoded in a flat connection $1$-form $\omega\in \Omega^1(P, \mathfrak{g})$. This is immediately turned into a chain map from the Chevalley-Eilenberg complex of $\mathfrak{g}$ to the DeRham complex of $P$ and then to a map in cohomology, 
\begin{equation}\label{kappa-omega}
\kappa^{\omega}: H^*(\mathfrak{g}, K) \to H^*(M),
\end{equation}
defined on the Lie algebra cohomology of $\mathfrak{g}$ relative to a maximal compact subgroup $K$ of $G$. 
\item[(abs):] abstractly/topologically: the flatness of $P$ ensures that it arises from a principal $G^{\delta}$-bundle ($G$ endowed with the discrete topology) hence, by the discussion above, the relevant characteristic map is
\[ \kappa_{\textrm{abs}}^{\omega}: H^*(BG^{\delta})\to H^*(M).\] 
%This time, one has a simpler explicit model computing the left hand side, namely the group cohomology $H^\bullet(G^\delta)$; recall that the relevant complex is given by group cochains $C^p(G^\delta)=\{c: G^p\to \mathbb{R}\}$, while the differential is the group differential given by
%\begin{align*}
%\delta(c)(g_1, \dots, g_{p+1})=c(g_2, \dots, g_{p+1})+\sum\limits^p_{i=1}(-1)^ic(g_1, \dots, g_i\cdot g_{i+1}, \dots, \dots g_{p+1})+(-1)^{p+1}c(g_1, \dots, g_p)
%\end{align*}
\end{itemize}

\medskip

Again, the two theories are related to each other via a ``universal characteristic map''
\[ \kappa^{\textrm{univ}}_{\delta}: H^*(\mathfrak{g}, K)\to H^*(BG^{\delta}).\]
Actually, everything can be made even more explicit since:
\begin{itemize}
\item there is a simple explicit model computing $H^*(BG^{\delta})$ namely the group cohomology $H_\delta^*(G)$ of $G$ viewed as a discrete group; recall that the relevant complex is given by group cochains $C^p_\delta(G)= C^p(G^\delta)= \{c: G^p\to \mathbb{R}\}$ (which ignore the topology on $G$), while the differential is the group differential given by
\begin{align}\label{delta-group}
\delta(c)(g_1, \dots, g_{p+1})= & c(g_2, \dots, g_{p+1})+ \sum\limits^p_{i=1}(-1)^ic(g_1, \dots, g_i\cdot g_{i+1}, \dots, \dots g_{p+1})\\
& +(-1)^{p+1}c(g_1, \dots, g_p)\nonumber 
\end{align}
\item of course, inside $C^{*}_\delta(G)$ one has the subcomplex of differentiable cochains
\[ C^{*}_{\textrm{diff}}(G)\subset C^{*}_\delta(G)\]
and then the so-called differentiable cohomology $H^{*}_{\textrm{diff}}(G)$ mapping naturally into $H^*_\delta(G)= H^*(BG^{\delta})$.
\item and one has the so-called van Est isomorphism, 
% \[ VE: H^*(\mathfrak{g}, K)\to H^{*}_{\textrm{diff}}(G).\]
\[
VE: H^*_{\rm diff}(G)\overset{\simeq}{\longrightarrow}H^*(\mathfrak{g}, K)
\]
\end{itemize}

All together, everything fits in one commutative diagram analogous to (\ref{char-diag-G-bundles}): 
%\[
%\begin{tikzcd}
%H^\bullet(BG^\delta)\arrow[rr]{}{\kappa_{\omega}^{abs}}& & H^\bullet_{dR}(X) \\
%&H^\bullet_{diff}(G)\simeq H^\bullet(\mathfrak{g}, K)\arrow[lu]{}{k_{\textrm{univ}}^{\delta}}\arrow[ru, swap]{}{\kappa_\omega}&  
%\end{tikzcd}
%\]
%\begin{equation}\label{char-diag-G-bundles-discrete}
%\begin{tikzcd}
%H^\bullet_{diff}(G)\simeq H^\bullet(\mathfrak{g}, K)\arrow[rr]{}{\kappa^{\omega}} \arrow[d, swap]{}{\kappa_{\delta}^{\textrm{univ}}}& & H^*(M)\\
%H^\bullet_\delta(G)\arrow[rru, swap]{}{\kappa^{\omega}_{abs}}& & 
%\end{tikzcd}
%\end{equation}

\begin{equation}\label{char-diag-G-bundles-discrete}
\begin{tikzcd}
H^*_{\rm diff}(G)\simeq H^*(\mathfrak{g}, K)\arrow[rr]{}{\kappa^{\omega}} \arrow[d, swap]{}{\kappa_{\delta}^{\textrm{univ}}}& & H^*(M)\\
H^*_\delta(G)= H^*(BG^\delta)\arrow[rru, swap]{}{\kappa^{\omega}_{\rm abs}}& & 
\end{tikzcd}
\end{equation}

\begin{exm}
For flat complex vector bundles one deals with $G= GL_n(\mathbb{C})$,  
\[ H^{*}_{\textrm{diff}}(GL_n(\mathbb{C}))\cong H^*(\mathfrak{gl}_n(\mathbb{C}, U(n))\cong \Lambda(h_1, \ldots, h_{2n-1}),\]
the exterior algebra in generators $h_{2i-1}$ of degree $2i-1$. In some sense, the vanishing of the Chern classes (in degree $2i$) give rise to ``secondary'' classes in odd degrees. A completely analogous story holds for real vector bundles. See, e.g., \cite{KT}.

%And similarly for real vector bundles, when ...
\end{exm}

%%%%%%%%%%%%%%%%%%%%%%%%%%%%
%%%%%%%%%%%%%%%%%%%%%%%%%%%%
%%%%%%%%%%%%%%%%%%%%%%%%%%%%
%%%%%%%%%%%%%%%%%%%%%%%%%%%%
%%%%%%%%%%%%%%%%%%%%%%%%%%%%
% \subsection{Characteristic classes for foliations I (geom)} 
\subsection{Geometric characteristic classes for foliations}
\label{Geometric characteristic classes for foliations}

%%%%%%%%%%%%%%%%%%%%%%%%%%%%
%%%%%%%%%%%%%%%%%%%%%%%%%%%%
%%%%%%%%%%%%%%%%%%%%%%%%%%%%
%%%%%%%%%%%%%%%%%%%%%%%%%%%%
%%%%%%%%%%%%%%%%%%%%%%%%%%%%

The explicit/geometric approach to characteristic classes for foliations has an outcome similar to what has been described above: for a codimension $q$ foliation $\F$ on a manifold $M$ there is an explicit characteristic map 
\begin{equation}\label{eq:k-abstr} 
\kappa^{\F}: GF^*_q\to H^{*}(M) ,
\end{equation} 
defined on certain cohomology groups $GF_{q}^{*}$. Here we use the notation ``GF" to illustrate the fact that these groups are usually referred to as ``Gelfand-Fuchs cohomology groups". Themselves rather explicit/computable, they have several different descriptions, depending on the way one approaches characteristic classes for foliations.

The most explicit way of describing codimension $q$ foliations on a manifold $M$ is as codimension $q$ sub-bundles $\F\subset TM$ which are involutive (closed under taking Lie brackets). And the explicit approach to $\kappa^{\F}$ arises by looking at the associated normal bundle $\nu= TM/\F$. Actually, one of the origins of the characteristic classes for foliations is 
Bott's vanishing theorem which says that 
%, for a codimension $q$ foliation $\F$ on a manifold $M$, 
all the usual characteristic classes (polynomial expressions in the Pontryagin classes) of  
% the normal bundle $\nu= TM/T\F$ 
$\nu$ vanish in degrees $> 2q$. As for flat bundles, this gives rise to ``secondary classes" which Bott constructs (see \cite{BOTTSTASHEFF}) in the spirit of the Chern-Weil theory. The vanishing result indicates that, instead of the polynomial algebra $\mathbb{R}[c_1, \ldots, c_q]$ one should be using
\[ \mathbb{R}_q[c_1, \ldots, c_q]:= \mathbb{R}[c_1, \ldots, c_q]/\{\P: \textrm{deg}(\P)> 2q\}.\]
Combined with the fact that the odd Chern classes of $\nu_{\mathbb{C}}$ vanish, one eventually discovers the differential graded algebra 
\[ WO_q:= \Lambda(h_1, h_3, \ldots, h_{\left<q\right>})\otimes \mathbb{R}_q[c_1, \ldots, c_q]\]
where $c_i$ have degree $2i$, $h_i$ have degree $2i-1$ (and defined only when $i$ itself is odd), $\left<q\right>$ is the largest odd number with $2\left<q\right>-1\leq q$, endowed with the differential
\[ d(c_i)= 0, \quad d(h_i)= c_i .\]
With these, 
\[ GF^*_q:= H^*(WO_q)\]

Another rather explicit but more conceptual standard description of $GF^*_q$ is as the cohomology of the Lie algebra $\mathfrak{a}_q$ of formal vector fields on $\mathbb{R}^q$, i.e. expressions of type
\[
X=\sum\limits_{i}^{q}f^i(x^1, \dots x^q)\frac{\partial}{\partial x^i}
\]
%where $(x^1, \dots, x^n)$ are coordinates around $x$ and the 
where the $f^i$'s are formal power series in the variables $x^i$'s. We endow $\mathfrak{a}_q$ with the power series topology and we consider the resulting (continuous) Lie algebra cohomology $H^*(\mathfrak{a}_q)$ and the relative version $H^*(\mathfrak{a}_q, O_q)$. As explained in \cite{BOTT} there is a canonical map from $WO_q$ to the relative Chevalley-Eilenberg complex, which induces an isomorphism 
\[ H^*(WO_q) \cong H^*(\mathfrak{a}_q, O_q).\]
With this in mind, $\kappa^{\F}$ can be constructed in the spirit of (\ref{kappa-omega}). 

%etc etc? 

%%%%%%%%%%%%%%%%%%%%%%%%%%%%
%%%%%%%%%%%%%%%%%%%%%%%%%%%%
%%%%%%%%%%%%%%%%%%%%%%%%%%%%
%%%%%%%%%%%%%%%%%%%%%%%%%%%%
%%%%%%%%%%%%%%%%%%%%%%%%%%%%
% \subsection{Characteristic classes for foliations II (abs)} 
\subsection{Abstract characteristic classes for foliations. Pseudogroups and $\Gamma$-structures}
\label{Abstract characteristic classes for foliations}
%%%%%%%%%%%%%%%%%%%%%%%%%%%%
%%%%%%%%%%%%%%%%%%%%%%%%%%%%
%%%%%%%%%%%%%%%%%%%%%%%%%%%%
%%%%%%%%%%%%%%%%%%%%%%%%%%%%
%%%%%%%%%%%%%%%%%%%%%%%%%%%%

The abstract characteristic map for foliations is based on the reinterpretation of foliations as 
principal bundles where, instead of Lie groups as structure groups one has to allow for pseudogroups/étale groupoids. 

To recall this we start from the very definition of  (codimension $q$) foliations $\F$ on $M$: as partitions of  $M$ by connected immersed submanifolds (leaves) of codimension $q$ which, locally, look like
\[ \mathbb{R}^n= \R^p\times \R^q= \cup_{y\in \mathbb{R}^q} \mathbb{R}^{p} \times \{y\} \quad (p= n- q).\]
This is realised by ``foliation charts'' on opens $U_i\subset M$ (covering $M$) ,
\[ \chi_i: U_i\to \mathbb{R}^n;\]
the fact that the leaves are preserved translates into the fact that the changes of coordinates 
$\chi_{j}\circ \chi_{i}^{-1}$ are, locally, of type 
% \[ (x, y)\mapsto (c_{U, V}(x, y), g_{U, V}(y)). )\]
\[ (x, y)\mapsto (h_{ij}(x, y) , g_{ij}(y))\]
Of course, this definition is equivalent to the description as an involutive sub-bundle of $TM$ thanks to Frobenius' theorem.

A slight variation is obtained by realizing that the leaves are induced by (i.e. are just the connected components of fibers of) submersions 
\[ f_i: U_i\rightarrow \R^q,\]
and the different submersions $f_i$ are, locally, related by 
\begin{equation}
\label{def-g-U-V} 
f_{j}= g_{ij}\circ f_i.
\end{equation}
These are the same $g_{ij}$'s as before, and they are diffeomorphisms between opens inside $\R^q$- the collection of which we denote by $\textrm{Diff}_{\textrm{loc}}(\R^q)$. Notice that these $g_{ij}$'s are there only locally: even when $i$ and $j$ are fixed, as $U_i\cap U_j$ may be disconnected. One can elegantly encode this situation by passing to germs:
\[ \Gamma^q:= \{\textrm{germ}_x(\phi): \phi\in \textrm{Diff}_{\textrm{loc}}(\R^q), x\in \textrm{Domain}(\phi)\}.\]
This has the so-called germ topology, in which the basic opens are the ones consisting of germs of any given $\phi\in \textrm{Diff}_{\textrm{loc}}(\R^q)$:
\[ U(\phi):= \{ \textrm{germ}_x(\phi): x\in \textrm{Domain}(\phi)\}.\]
Equipped with such topology, $\Gamma^q$ is locally Euclidean; in fact, the natural projection 
\[
{\rm germ}_x(\varphi)\in \Gamma^q  \to x \in \mathbb{R}^q
\]
is a local homeomorphism. Consequently, $\Gamma^q$ can be equipped with a natural smooth structure, with the caveat that it is neither Hausdorff nor second countable. 
Nevertheless, with this one may say, with a slight abuse of notation, that 
\[ g_{ij}: U_i\cap U_j\rightarrow \Gamma^q\]
are smooth maps, bringing foliations closer to principal bundles. Note, in particular, that the cocycle condition 
\[ g_{jk}g_{ij}= g_{ik}\]
is satisfied, as it follows from (\ref{def-g-U-V}) and the fact that the $f_i$s are submersions. Understanding what allows one to write down such an equation  reveals the relevant structure behind this discussion. In particular, one discovers the more general setting of pseudogroups, étale groupoids and $\Gamma$-structures. We stress that this general framework not only allows us to get an abstract characteristic map for foliation in complete analogy with~\eqref{eq:abstract_char_map_pb}, but provide us with a unifying approach that:
\begin{itemize}
\item produces more general constructions, i.e. the characteristic map is defined for $\Gamma$-structures, Definition~\ref{def:Gamma_structure} below;
\item is well suited to describe a geometric map as well, subsection~\ref{Intermezzo: back to geometric characteristic classes} and section~\ref{Haefliger's differentiable cohomology};
\item paves the way for a further generalization of the various concepts involved, sections~\ref{Generalization: Haefliger cohomology},~\ref{The infinitesimal version of Haefliger cohomology} and~\ref{Van Est maps}, that makes the structure behind both maps and their interplay conceptually clear.
\end{itemize}

We start by summing up the main definitions and constructions.
\medskip

First of all, a {\bf pseudogroup} on a manifold $\BB$ is a subset $\Gamma\subset \mathrm{Diff_{loc}}(\BB)$ of the set of diffeomorphisms between opens in $\BB$ satisfying the following axioms:
\begin{itemize}
\item $\Gamma$ is closed under composition, inversion and it contains the identity $id_\BB$.
\item $\Gamma$ is closed under restriction to smaller opens.
\item Let $\mathcal{U}$ be an open cover of an open set $\tilde{U}\subset \BB$ and $\phi$ a diffeomorphism defined on $\tilde{U}$ and such that $\phi|_{U}\in \Gamma$ for all $U\in \mathcal{U}$. Then $\phi \in \Gamma$.
\end{itemize} 

\medskip 

Of course, the basic example for us is $\textrm{Diff}_{\textrm{loc}}(\R^q)$. As indicated by the discussion above, to talk about cocycles (and principal bundles) it is useful to pass to germs
% \[ \Ger(\Gamma):= \{\textrm{germ}_x(\phi): \phi\in \Gamma, x\in \textrm{Domain}(\phi)\}\]
\[ \Ger(\Gamma):= \{\textrm{germ}_x(\phi): \phi\in \Gamma, x\in \textrm{Domain}(\phi)\}\]
endowing $\Ger(\Gamma)$ with the germ topology; as for $\Gamma^q$, this choice of topology makes $\Ger(\Gamma)$ locally Euclidean. The structure present on 
$\Ger(\Gamma)$ that is relevant to our discussion is that of groupoid over $\BB$- schematically denoted as
\[ \Ger(\Gamma)\tto \BB.\]
In more detail: it comes with source and target maps
\[ s, t: \Ger(\Gamma)\rightarrow \BB\]
which send an element $\textrm{germ}_x(\phi)$ to $x$, and $\phi(x)$, respectively; and, using composition of functions, any two elements $\gamma_1, \gamma_2\in \Ger(\Gamma)$ which match (i.e. $s(\gamma_1)= t(\gamma_2)$) can be composed to give $\gamma_1\cdot \gamma_2\in \Ger(\Gamma)$. Indeed, writing $\gamma_i= \textrm{germ}_{x_i}(\phi_i)$, the matching condition is $x_1= \phi(x_2)$ and then $\phi:= \phi_1\circ \phi_2$ is well defined around $x:= x_2$ hence we can define $\gamma_1\cdot \gamma_2:= \textrm{germ}_{x}(\phi)$. 
\medskip

% There are two important properties 
There is one important property 
of the groupoid $\Ger(\Gamma)$ to be noticed:  
it is an {\bf étale groupoid}, in the sense that its source and target maps are local diffeomorphisms. The fact that étale groupoids $\G\tto \BB$ are closely related to pseudogroups on $\BB$ is seen right away by looking at bisections of $\G$, i.e. submanifolds $\mathcal{B}\subset \G$ on which both $s$ and $t$ restrict to diffeomorphism onto opens inside $\BB$, $U:= s(\Sigma)$ and $V:= t(\Sigma)$.  
%Up to precomposing with a locally defined diffeomorphism of $\BB$,
 One can re-interpret such bisections as sections of $s$
\[ b: U\to \G\]
with the property that $t\circ b$ is a diffeomorphism onto its image ($b$ is the inverse of $s|_{\mathcal{B}}$,  and, reversely, $\mathcal{B}$ is the image of $b$). 
% $U:= s(\Sigma)$ and $V:= t(\Sigma)$. 
Any such bisection gives rise to a diffeomorphism 
\begin{equation}\label{basic-action-for-étale}
\phi_{\mathcal{B}}=\phi_b:=\textrm{germ}_g(t)\circ \textrm{germ}_g(s)^{-1}: U\to V
\end{equation}
and, all together, they form {\bf the pseudogroup $\Gamma_{\G}$ associated to 
% the étale groupoid 
$\G$}. Note that:
\begin{itemize}
\item the pseudogroup associated to an étale groupoid of type $\Ger(\Gamma)$ is $\Gamma$ itself;
\item for a general étale groupoid $\G\tto \BB$, there is a canonical surjection from $\G$ into the germ groupoid associated to $\Gamma_{\G}$,
\[ \G\rightarrow \Ger(\Gamma_{\G}), \quad g\mapsto \xi_g ,\]
where $\xi_g$ is represented by a/any open neighborhood $\Sigma_g$ of $g$ on which 
$s$ and $t$ restrict to diffeomorphisms - more directly, by the germ of 
%$\phi_\mathcal{B}$. Equivalently, one can make use of 
$\phi_\mathcal{B}$.
%\begin{equation} \label{basic-action-for-étale}
%b_g:=  \textrm{germ}_g(t)\circ \textrm{germ}_g(s)^{-1}.
%\end{equation}
\end{itemize} 
It is customary to call an étale groupoid $\G\tto \BB$ {\bf effective} if, for $g\in \G$, the germ $b_g$ is the identity only if $g$ is a unit arrow. All together, one obtains
%\[ 
%\bigg\{  
%\text{pseudogroups on $\BB$} \bigg\}   
%\stackrel{1-1}{\longleftrightarrow}
%    \left\{   \begin{array}{c}
%               \text{effective étale groupoids over $\BB$}
%               \end{array} 
%    \right\} 
%\]
%
\[ 
 \left\{   \begin{array}{c}
            \text{pseudogroups}\\
            \text{$\Gamma$ on $\BB$} 
            \end{array} 
\right\} 
\stackrel{1-1}{\longleftrightarrow}
\left\{   \begin{array}{c}
           \text{effective étale groupoids}\\
           \text{$\G\tto \BB$}
           \end{array} 
\right\} 
\]
up to isomorphism or, more precisely, an equivalence of  categories. However, many notions involving pseudogroups $\Gamma$ (e.g. equivalences, cocycles, actions) are simpler and more elegantly described passing to the associated germ groupoids $\Ger(\Gamma)$.

\begin{defn}\label{Cocycle} Let $\Gamma$ be a pseudogroup on $\BB$ and let $M$ be another manifold. A {\bf $\Gamma$-cocycle} on $M$ consists of a cover $\mathcal{U}=\{U_i\}_{i\in I}$ of $M$ and smooth maps 
\[ \gamma_{ij}:U_i\cap U_j\to \Ger(\Gamma)\] 
satisfying
\begin{itemize}
\item for $i= j$, each $\gamma_{i i}(x)$ is a unit (hence can be seen as a point in $\BB$),
\item for arbitrary $i$ and $j$ and $x\in U_i\cap U_j$, $\gamma_{ij}(x)$ is an arrow
% (of an element from $\Gamma$) 
from $\gamma_{ii}(x)$ to $\gamma_{jj}(x)$,
\end{itemize}
and such that the cocycle condition 
\[
\gamma_{jk}(x)\cdot \gamma_{ij}(x)=\gamma_{ik}(x).
\]
is satisfied for all $x\in U_i\cap U_j\cap U_k$. Similarly we talk about {\bf $\G$-cocycles}
for any groupoid $\G\tto \BB$ by replacing above $\Ger(\Gamma)$ with $\G$. \end{defn}

\medskip 

Note that the dotted conditions actually follow from the cocycle one and the fact that it should make sense. Spelling them out, the first condition encodes maps 
\[ f_{i}:= \gamma_{i, i}:U_i\to \BB\]
while the second condition means that, around each $x\in U_i\cap U_j$, one can find $g_{ij}\in \Gamma$ representing $\gamma_{ij}(x)$ such that, on its domain, 
\[ f_j= g_{ij}\circ f_i.\]
When we want to emphasize the $f_i$s, we also say that $(\mathcal{U}, f_i, \gamma_{ij})_{i, j \in I}$ is a $\Gamma$-cocycle. 
If the maps $f_i$ are submersions, then the cocycle condition follows automatically, but not in general. 

One recognizes the discussion from foliations when $\BB= \R^q$, $\Gamma= \Gamma^q$ and the $f_i$'s are submersions. General $\Gamma^q$-cocycles can be thought of as a ``singular foliations" in the naive sense: the fibers of the $f_i$'s can change dimension when moving transversally. %in this generality, since the $f_i$'s are no longer required to be submersions, the cocycle condition does not follow from the rest. 

\begin{rmk}\label{rk:gpd-GU} Any open cover $\mathcal{U}= \{U_i\}_{i\in I}$ of a manifold $M$ gives rise to a Mayer-Vietoris groupoid $M_{\mathcal{U}}$ which allows to re-interpret the notion of cocycle a bit more conceptually. The groupoid $M_{\mathcal{U}}$ is defined over the disjoint union of the $U_i$ as the pull-back of the unit groupoid $M\tto M$. More explicitly, $M_{\mathcal{U}}$ is the groupoid 
\[  \bigsqcup_i U_i\cap U_j\tto \bigsqcup_i U_i\]
where to describe the groupoid structure we write the elements in 
% $M_{\mathcal{U}}$ will be the disjoint union of all the intersections $U_i\cap U_j$.
the base as pairs $(i, x)$ with $i\in I$ and $x\in U_i$, and the elements in the total space as triples $(i, x, j)$ with $i, j\in I$ and $x\in U_i\cap U_j$; in these notations, the target, source and composition are:
%while the base is
%\[ \bigsqcup_i U_i= \{(i, x): i\in I, x\in U_i\},\]
%the total space will be 
%\[ M_{\mathcal{U}}:= \bigsqcup_i U_i\cap U_j= \{(i, x, j): i, j\in I, x\in U_i\cap U_j\},\]
%with target, source and composition:
\[ t(i, x, j)= (i, x), \quad s(i, x, j)= (j, x),\quad (i, x, j)\cdot (j, x, k)= (i, x, k).\]

With this, for any étale groupoid $\G\tto \BB$, a $\G$-cocycle $\gamma= \{\gamma_{ij}\}$ on $M$ over an open cover $\mathcal{U}= \{U_i\}$ of $M$ is the same thing as a morphism of groupoids
\begin{equation}\label{eq:cocycl-Veit-morphism} 
\gamma: M_{\mathcal{U}}\to \G.
\end{equation}
\end{rmk}

%
%For any Lie groupoid $\G\tto \BB$ the notion of $\G$-cocycle can be expressed more compactly using the fact that any open cover
%$\mathcal{U}= \{U_i\}_{i\in I}$ gives rise to a groupoid $\G_{\mathcal{U}}$ over the disjoint union of the $U_i$s. This is defined as the pull-back of $\G$ along the canonical map $\sqcup_i U_i\to \BB$. Concretely, while the elements of the base can be represented as pairs $(i, x)$ with $i\in I$, $x\in U_i$,
%\[ \G_{\mathcal{U}}=\{(i, g, j): i, j\in I, g\in \G \textrm{with}\ t(g)\in U_i, s(g)\in U_j\}\]
%and the target and source of $(i, g, j)$ is $(i, t(g))$ and $(j, s(g))$, respectively. 
%
%\end{defn}

\medskip

There is also a straightforward analogue of the notion of equivalence of cocycles for Lie groups, called {\bf equivalence of $\Gamma$-cocycles}. Compactly: two $\Gamma$-cocycles indexed by $I$ and $J$ are equivalent if they are part of a larger $\Gamma$-cocycle indexed by $I\sqcup J$.

%\begin{defn}
%A {\bf $\Gamma$-structure} on a manifold $M$ is an equivalence class of $\Gamma$-cocycles. 
%A $\Gamma$-structure induced by $\Gamma$-cocycles for which all the $f_i$s are submersions is called a {\bf $\Gamma$-foliation on $M$}. 
%
%\begin{defn}
%A {\bf $\Gamma$-structure} on a manifold $M$ is an equivalence class of $\Gamma$-cocycles. 
%A {\bf $\Gamma$-foliation on $M$} is a $\Gamma$-structure induced by a $\Gamma$-cocycle for which all the $f_i$s are submersions.
%\end{defn}
%
%
%\begin{defn}
%A {\bf singular $\Gamma$-structure} on a manifold $M$ is an equivalence class of $\Gamma$-cocycles. 
%A {\bf $\Gamma$-structure} on $M$ is a singular $\Gamma$-structure induced by a $\Gamma$-cocycle for which all the $f_i$s are submersions.
%\end{defn}

\begin{defn}\label{def:Gamma_structure}
A {\bf Haefliger $\Gamma$-structure} on a manifold $M$ is an equivalence class of $\Gamma$-cocycles. 
A {\bf $\Gamma$-structure} on $M$ is a Haefliger $\Gamma$-structure induced by a $\Gamma$-cocycle for which all the $f_i$s are submersions.
\end{defn}

\begin{rmk}
	
	A remark is due concerning our choice of terminology.
	
	$\Gamma$-structures are sometimes also called $\Gamma$-foliations. This terminology emphasises the fact that a $\Gamma$-structure on $M$ induces a foliation (built out of the fibers of the submersions $f_i$'s from the cocycle). In some sense, a $\Gamma$-structure on $M$ is a foliation together with extra-structure in the transversal direction.
	
	Moreover, in some of the existing literature, the term ``$\Gamma$-structure" is reserved to the case when the dimension of $M$ is equal to the dimension of $\BB$ (on which $\Gamma$ lives). Under such assumption, one usually works with the notion of 
	{\bf $\Gamma$-atlas} on $M$, i.e. a smooth atlas modelled on opens in $\BB$ with the property that the coordinate changes are elements of $\Gamma$. Of course, this is just the notion of $\Gamma$-cocycle which becomes simpler due to the condition $\textrm{dim}(M)= \textrm{dim}(\BB)$.
\end{rmk}

\begin{exm}  
	Let $\BB= \mathbb{C}^k= \mathbb{R}^{2k}$ and $\Gamma$ be the pseudogroup of holomorphic diffeomorphisms. A $\Gamma$-structure on a $ 2k$-dimensional manifold $M$ is the same thing as a complex structure on $M$; if $\dim(M)>2k$, one is looking at transversally holomorphic foliations on $M$ of codimension $n- 2k$. 
	
	On the other hand, a Haefliger $\Gamma$-structure on $M$ induces a singular foliation on $M$, and the fact that such foliation can be ``presented'' by a $\Gamma$-cocycle can be interpreted as the existence of a complex structure in the transverse direction/on its leaf space.
\end{exm}

% and are called {\bf Haefliger cocycles}. Since the $f_i$s are no longer required to be submersions, the cocycle condition does not follow from the rest. 
\medskip

As for cocycles with values in a Lie group $G$, one can re-encode everything into more global objects: principal bundles.  In particular, this allows one not to worry anymore about the choice of coverings (and the corresponding notion of equivalence). 

This can be done (and will be needed) in the generality of {\bf Lie groupoids}. A Lie groupoid is a groupoid $\Sigma\tto \BB$ where both the arrow space and the unit space are smooth manifolds, all the defining maps are smooth and both the source and the target map are surjective submersions. We will use $\mathcal{G}$ to denote arrow spaces of étale groupoids and $\Sigma$ to denote arrow spaces of more general, non étale, Lie groupoids. 

For a Lie groupoid $\Sigma\tto \BB$, a {\bf principal $\Sigma$-bundle} over a manifold $M$, schematically described as
%\[
%\xymatrix{
%&   \mathcal{P} \ar[drr]_{\mu}\ar[dl]_{\pi}  & \ar@(dr, ur)& \G \ar@<0.25pc>[d] \ar@<-0.25pc>[d]      \\
%M&  & &  \BB},
%\]\
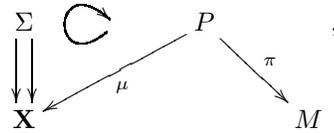
\[
\xymatrix{
\Sigma \ar@<0.25pc>[d] \ar@<-0.25pc>[d]  & \ar@(dl, ul) &  P \ar[dll]^{\mu}\ar[dr]^{\pi} &    \\
\BB&  & &  M},
\]
is a manifold $P$, together with a submersion $\pi: P\to M$ and 
a free and proper action of $\Sigma$ on $P$ from the left, along some smooth map $\mu: \mathcal{P}\to \BB$, such that $\pi$ is identified with the resulting quotient map $P\to P/\Sigma$. Recall here that an action of $\Sigma$ on $P$ means that any arrow $g\in \Sigma$ from $x$ to $y$ gives rise to ``multiplication by $g$'' between the resulting $\mu$-fibers
\[ \mu^{-1}(x) \to \mu^{-1}(y), \quad p\mapsto g\cdot p\]
% \[ \mu^{-1}(y) \to \mu^{-1}(x), \quad p\mapsto p\cdot g\]
satisfying the usual identities for actions. Of course, everything is required to be smooth- i.e. $(g, p)\mapsto g\cdot p$ is smooth as a map defined on the submanifold $\Sigma\ltimes P\subset \Sigma\times P$ consisting of pairs $(g, p)$ with $\mu(p)= s(g)$. 
Similarly one can talk about actions from the right and right principal $\Sigma$-bundles:
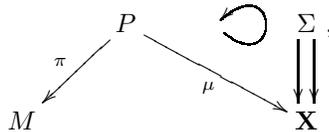
\[
\xymatrix{
&   P \ar[drr]_{\mu}\ar[dl]_{\pi}  & \ar@(dr, ur)& \Sigma \ar@<0.25pc>[d] \ar@<-0.25pc>[d]      \\
M&  & &  \BB},
\]
%
%\[
%\xymatrix{
%\G \ar@<0.25pc>[d] \ar@<-0.25pc>[d]  & \ar@(dl, ul) &  \mathcal{P} \ar[dll]^{\mu}\ar[dr]^{\pi} &    \\
%\BB&  & &  M},
%\]
%
when the multiplication $p\mapsto p\cdot g$ by an arrow $g\in \Sigma$ from $x$ to $y$ takes $\mu^{-1}(y)$ to $\mu^{-1}(x)$. The action is free if the equalities of type $g\cdot p= p$ hold only when $g$ is a unit arrow; and the action is proper if the map $\Sigma\ltimes P\to P\times P$, $(g, p)\mapsto (g\cdot p, p)$ is a proper map. These conditions can be slightly reformulated by saying that one has a smooth action of $\Sigma$ on $P$ such that the map
\[ \Sigma\ltimes P \to P\times_{M} P\]
is a diffeomorphism.

For future use, we recall here that a (say, left) action of $\Sigma$ on $P$ along $\mu:P\to \BB$ is equivalently encoded into the Lie groupoid structure on $\Sigma\ltimes P \tto P$ such that
\[
s(g,p) = p,\quad t(g,p) = g\cdot p,\quad (h,g\cdot p)\cdot (g,p) = (hg,p).
\]
The Lie groupoid $\Sigma\ltimes P\tto P$ is called the {\bf action groupoid} associated to the action of $\Sigma$ on $P$; when such action is principal, the diffeomorphism $\Sigma\ltimes P \to P\times_{M} P$ is in fact a Lie groupoid isomorphism, where the right hand side carries the groupoid structure with structure maps
\[
s(p_1,p_2) = p_1,\quad t(p_1,p_2) = p_2,\quad (p_1,p_2)\cdot (p_2,p_3) = (p_1,p_3).
\]

 Of course, the entire discussion is completely similar to that from principal bundles with Lie structural group. And so is the relationship with cocycles: given such a principal $\Sigma$-bundle $\pi: P\to M$, one chooses a cover $\{U_i\}_{i\in I}$ of $M$ by opens $U_i$ on which $P$ admits sections $\sigma_i: U_i\to P$ and then, on points $x\in U_i\cap U_j$ in the overlaps, one can write 
\[ \sigma_i(x)=  \gamma_{ij}(x)\cdot \sigma_j(x)\]
for some $\gamma_{ij}(x)\in \Sigma$. Note the the resulting maps into $\BB$ are
\[ f_i= \mu\circ \sigma_i: U_i \to \BB\]
and one ends up with a $\Sigma$-cocycle. And, still proceeding like for Lie groups to define the notion of equivalence of coycles (see \cite{JANEZ}), one obtains 
\[ 
 \left\{   \begin{array}{c}
%            \text{$\Gamma$-cocycles}\\
%            \text{(or $\G$-cocycles)}\\
 % 
 %          \text{$\G$-cocycles }\\
 %           \text{on $M$ (up to equivalence)} 
 %
            \text{$\Sigma$-cocycles on $M$}\\
            \text{(up to equivalence)}
            \end{array} 
\right\} 
\stackrel{1-1}{\longleftrightarrow}
\left\{   \begin{array}{c}
%           \text{principal $\Germ(\Gamma)$-bundles}\\
%           \text{(or principal $\G$-bundles)}\\
 %
 %          \text{principal $\G$-bundles}\\
 %          \text{over $M$ (up to isomorphism)}
 %
           \text{principal $\Sigma$-bundles on $M$}\\
           \text{(up to isomorphism)}
           \end{array} 
\right\} 
\]\

And, again as for Lie groups, for any Lie groupoid $\Sigma$ one can talk about its classifying space $B\Sigma$, for which there are various homotopically equivalent models. Principal $\Sigma$-bundles $\pi: P\to M$ come together with classifying maps $M\to B\Sigma$ uniquely defined up to homotopy. A standard approach can be found in \cite{SEGAL};  $B\Sigma$ can be defined as the thick geometric realization of the nerve of $\Sigma$ (see also subsection~\ref{More on cohomology of classifying spaces and the abstract characteristic map}).

Observe that, when the Lie groupoid $\Sigma\tto \BB$ is an étale groupoid $\G\tto \BB$, the projection $\pi$ is an étale map. To stress this, we will use the notation $\P$ for the total space of a principal $\G$-bundle. We focus now on this case.

As a consequence of the existence of a classifying space, one has abstract characteristic maps 
\begin{equation}\label{eq:abstr-k-G} 
\kappa^{\mathcal{P}}_{\textrm{abs}}: H^*(B\G)\to H^*(M).
\end{equation}
An explicit construction is described/recalled a bit later, in Example \ref{exm:back:charact;map}.  For $\G= \Ger(\Gamma)$ associated to a pseudogroup $\Gamma$, one uses the notation $B\Gamma$. 
Finally, applied to codimension $q$ foliations $\F$ on $M$ interpreted as $\Gamma^q$-cocycles, one obtains the resulting abstract characteristic maps that are of interest for us:
\begin{equation}\label{eq:abstr-k-Gammaq} 
 \kappa^{\F}_{\textrm{abs}}: H^*(B\Gamma^q)\to H^*(M).
 \end{equation}

%%%%%%%%%%%%%%%%%%%%%%%%%%%%
%%%%%%%%%%%%%%%%%%%%%%%%%%%%
%%%%%%%%%%%%%%%%%%%%%%%%%%%%
%%%%%%%%%%%%%%%%%%%%%%%%%%%%
%%%%%%%%%%%%%%%%%%%%%%%%%%%%
\subsection{Intermezzo: back to geometric characteristic classes (for $\Gamma$-structures)}
\label{Intermezzo: back to geometric characteristic classes}
%\footnote{{\color{red} this subsection was not revised after writing the appendix; some things have to be adapted now to what is in the appendix}}
%%%%%%%%%%%%%%%%%%%%%%%%%%%%
%%%%%%%%%%%%%%%%%%%%%%%%%%%%
%%%%%%%%%%%%%%%%%%%%%%%%%%%%
%%%%%%%%%%%%%%%%%%%%%%%%%%%%
%%%%%%%%%%%%%%%%%%%%%%%%%%%%

While the discussion of the abstract characteristic classes map led us to the more general context of pseudogroups $\Gamma$ and $\Gamma$-structures, 
% principal $\Ger(\Gamma)$-bundles, 
it is natural to wonder whether also the geometric characteristic map (\ref{eq:k-abstr}) can be adapted to this more general context. The answer is ``yes", at least in the case of pseudogroups 
$\Gamma$ that are {\bf Lie} (see below) and {\bf transitive} in the sense that for any two points in the base, $x, y\in \BB$, there exists $\phi\in \Gamma$ such that $\phi(x)= y$. The key is to look at the infinitesimal counterpart of diffeomorphisms, i.e. at vector fields.

\begin{defn}\label{defn:Gamma-vector-fields}
Let $\Gamma$ be a pseudogroup on $\BB$. 
\begin{itemize}
\item A {\bf (local) $\Gamma$-vector field} is any 
(local) vector field $X$ on $\BB$ with the property that its associated flow of diffeomorphisms $\phi_X^t$ all belong to $\Gamma$. Denote by $\mathfrak{X}_{\Gamma}(U)$ the space of such vector fields defined over $U\subset \BB$ open. 
\item A {\bf formal $\Gamma$-vector field at $x$}, for $x\in \BB$, is an infinite jet at $x$ of a local $\Gamma$-vector field defined around $x$. Denote by $\mathfrak{a}_x(\Gamma)$ the space of such.
\item When $\Gamma$ is a transitive pseudogroup on $\BB= \R^q$, $x= 0$,  the space $\mathfrak{a}_0(\Gamma)$ can be equipped with the Lie bracket defined by
\[
[j^\infty_0X, j^\infty_0 Y]=j^\infty_0[X, Y],
\]
where $X, Y$ are local $\Gamma$-vector fields around $0$; it will be called the {\bf algebra of $\Gamma$-vector fields on $\R^q$}, or the {\bf formal algebra of $\Gamma$}, and will be denoted $\mathfrak{a}(\Gamma)$. 
\end{itemize}
\end{defn}

%
%
%\begin{defn}
%Let $\Gamma$ be a pseudogroup on $\BB$. A {\bf (local) $\Gamma$-vector field} is any 
%(local) vector field $X$ on $\BB$ with the property that its associated flow diffeomorphisms $\phi_X^t$ all belong to $\Gamma$. 
%
%For $b\in \BB$, we define
%% {\bf the formal Lie algebra of $\Gamma$ at $b$}, 
%{\bf the Lie algebra of formal $\Gamma$-vector fields at $b$}, 
%denoted $\mathfrak{a}_b(\Gamma)$, as the Lie algebra consisting of infinite jets at $b$ of local $\Gamma$-vector fields. 
%
%If $\Gamma$ is a transitive pseudogroup on $\BB= \R^q$, $b= 0$, the resulting $\mathfrak{a}_0(\Gamma)$ will be called {\bf the Lie algebra of formal vector field on $\R^q$} and will be denoted $\mathfrak{a}(\Gamma)$. 
%\end{defn}

%
%\begin{defn}
%Assume that $\Gamma$ is a transitive pseudogroup on $\BB$ and fix $b_0\in \BB$. A {\bf (local) $\Gamma$-vector field} is any (local) vector field $X$ with the property that all the associated flow diffeomorphisms $\phi_X^t$ belong to $\Gamma$. {\bf The formal Lie algebra of $\Gamma$}, denoted $\mathfrak{a}(\Gamma)$, is defined as the Lie algebra consisting of infinite jets at $b_0$ of local $\Gamma$-vector fields. 
%\end{defn}
%

Concerning the third point, it is not difficult to see that, when $\Gamma$ is transitive, all the formal algebras $\mathfrak{a}_x(\Gamma)$ are isomorphic to each other hence, up to isomorphism, there is an unambiguously defined formal Lie algebra $\mathfrak{a}(\Gamma)$; however, to work with a concrete model, one has to fix a base point $x\in \BB$. When $\BB=\mathbb{R}^q$ and $\Gamma$ is the pseudogroup ${\rm Diff}_{\rm loc}(\mathbb{R}^q)$, the resulting Lie algebra is denoted by $\mathfrak{a}_q$. Notice that in this case, formal vector fields $X\in \mathfrak{a}_q$ are simply expressions 
\[
X=\sum\limits_{i}^{q}f^i(x^1, \dots x^q)\frac{\partial}{\partial x^i}
\]
where each $f^i$ is a formal power series in $x^1, \ldots, x^q$. 
%Incidentally, let us point out here that the obvious map from 
%$\mathfrak{X}(\R^q)$ to $\mathfrak{a}_q$ (taking the infinite jet at the origin) is know to induce isomorphisms in (continuous) cohomology, and this is  the Gelfand-Fuchs approach ~\cite{GF} (see also the next theorem)
%
%see~\cite{GF}, 
%
%While, strictly speaking, the Gelfand-Fuchs cohomology (denoted generically $GF_{q}^{*}$ above) is about the (continuous) cohomology of the Lie algebra of vector fields, a 
%result of Gelfand-Fuchs-Bott-Haelfiger\footnote{add proper ref} says that the Lie algebra morphism $j^{\infty}_{0}: \mathfrak{X}(\R^q)\to \mathfrak{a}_q$ induces an isomorphism in (continuous) cohomology. See~\cite{GF}. 
%
%

\medskip

For a general pseudogroup $\Gamma$, germs of elements from $\Gamma$ gave rise to the groupoid $\Ger(\Gamma)$. Similarly, for $l\in \mathbb{N}$, $l$-jets of elements from $\Gamma$ gives rise to the {\bf groupoid of $l$-jets $J^l\Gamma$} and one has the tower of groupoids
\begin{equation}\label{eq:the-jet-tower} 
J^{\infty}\Gamma \rightarrow \ldots \rightarrow J^l\Gamma\rightarrow J^{l-1}\Gamma \rightarrow \ldots \rightarrow J^1\Gamma\rightarrow J^0\Gamma.
\end{equation}
For instance, $J^0\Gamma\subset \BB\times \BB$ and, if $\Gamma$ is transitive, this inclusion becomes equality. For general $l$, $J^l\Gamma$ is a subspace of the manifold $J^l(\BB, \BB)$ of $l$-jets of diffeomorphisms of $\BB$. When all these are smooth submanifolds and all the projections in the tower are surjective submersions, one says that $\Gamma$ is a {\bf Lie pseudogroup}. 
Lie pseudogroups are best discussed in the framework of profinite dimensional differential geometry, which allows to make sense of $J^\infty\Gamma$ as a {\bf pf Lie groupoid}. We refer to the Appendix; in particular, Example~\ref{exam-Lie-pseudogroups} provides a definition of Lie pseudogroup that fit into this framework and explores some of the features of the tower above. Using terminology from the appendix, we point out that {\it we will always assume the tower above to be a {\bf normal pf-atlas} of $J^\infty\Gamma$}, i.e. that the natural map $J^\infty\Gamma\to \lim\limits_{\longleftarrow}J^l\Gamma$ is a bijection. Here, we are especially interested in the inverse system of the isotropy Lie groups $J^l\Gamma_x$ at any $x\in \BB$, for $l\in \mathbb{N}$. The isotropy group $J^{\infty}\Gamma_x$ is isomorphic to the limit of this system via an isomorphism of pf-Lie groups. 
%Note that in this case the isotropy groups at any $b\in \BB$, $J^l\Gamma_b$ (arrows starting and ending at $b$), are Lie groups for $l\neq \infty$ and they fit into an inverse system of groups with limit $J^{\infty}\Gamma_b$. 
%The infinitesimal counterpart of this gives rise to an inverse system of Lie algebras 
%\[ \mathfrak{a}_b(\Gamma) \rightarrow \ldots \rightarrow \mathfrak{a}^l_b(\Gamma)\rightarrow \mathfrak{a}^{l-1}_b(\Gamma) \rightarrow \ldots \rightarrow \mathfrak{a}^1_b(\Gamma)\rightarrow \mathfrak{a}^0_b(\Gamma)\]
%with the limit precisely the formal Lie algebras discussed above. 
Thanks to the Lie assumption one can find, at each $b\in \BB$, a subgroup 
$K\subset J^{\infty}\Gamma_b$ such that its projection at each finite order $l$ is a maximal compact subgroup of $J^l\Gamma$. Such a $K$ will be called a {\bf maximal compact subgroup at $b$}; they are unique up to conjugation. As before, when $\Gamma$ is transitive, one can move from $x$ to any other point in $\BB$ and carry $K$ to similar groups at other points. In particular, when $\Gamma$ is a transitive Lie pseudogroup over $\BB= \R^q$, we use $x= 0$ and we talk about {\bf maximal compact subgroups for $\Gamma$}. For instance, when $\Gamma$ is the entire pseudogroup $\textrm{Diff}_{\textrm{loc}}(\R^q)$, one can take $K= O(q)$; this can be embedded in the isotropy group of the infinite jet groupoid of ${\rm Diff}_{\textrm{loc}}(\R^q)$ by taking infinite jets of linear orthogonal maps.

\begin{thm}[\cite{BOTTHAEFLIGER}, Theorems $1$-$2$] 
For any transitive Lie pseudogroup $\Gamma$ on $\R^q$ and for any choice of a maximal compact $K$, any principal $\G= \Ger(\Gamma)$-bundle $\mathcal{P}\to M$ can be associated to a map in cohomology
\[ \kappa^{\mathcal{P}}: H^*(\mathfrak{a}(\Gamma), K) \to H^*(M)\]
such that: 
\begin{itemize}
\item when $\Gamma= {\rm Diff}_{\textrm{loc}}(\R^q)$ and $\P$ describes a codimension $q$ foliations, one has a canonical isomorphism 
\[ H^*(\mathfrak{a}_q, O(q))\cong H^*(WO_q)\]
and we recover the geometric characteristic map $\kappa^{\F}$~\eqref{eq:k-abstr} of the foliation.
\item if $\mathcal{P}$ encodes a $\Gamma$-structure, with induced foliation $\F$ on $M$, then the composition of $\kappa^{\mathcal{P}}$ with the canonical map induced by the inclusion $i: \mathfrak{a}(\Gamma)\hookrightarrow \mathfrak{a}$ is precisely the geometric characteristic map $\kappa^{\F}$~\eqref{eq:k-abstr}. 
\[ 
\xymatrix{
      H^*(\mathfrak{a}_q, O(q)) \ar[rr]^-{\kappa^{\F}}\ar[rd]_{i^*} & & H^*(M)\\
      & H^*(\mathfrak{a}(\Gamma), K) \ar[ru]_-{\kappa^{\mathcal{P}}} & 
}\]                
\end{itemize}
\end{thm}
%\begin{thm} 
%For any transitive Lie pseudogroup $\Gamma$ on $\R^q$ and the choice of a maximal compact $K$, there is a construction that associates to any principal $\Gamma$-bundle $\mathcal{P}\to M$ a map in cohomology
%\[ k^{\mathcal{P}}: H^*(\mathfrak{a}(\Gamma), K) \to H^*(M)\]
%such that: 
%\begin{itemize}
%\item when $\Gamma= \textrm{Diff}_{\textrm{loc}}(\R^q)$ and codimension $q$ foliations, one has a canonical isomorphism 
%\[ H^*(\mathfrak{a}_q, O(q))\cong H^*(WO_q)\]
%and we recover the geometric characteristic map $k^{\F}$.
%\item if $\mathcal{P}$ encodes a $\Gamma$-foliation $\F$, then the composition of $k^{\mathcal{P}}$ with the canonical map $H^*(\mathfrak{a}_q, O(q))\to H^*(\mathfrak{a}(\Gamma), K)$ induced by the inclusion $\mathfrak{a}(\Gamma)\subset \mathfrak{a}$ is precisely the geometric characteristic map $k^{\F}$ of the foliation. 
%\end{itemize}
%%There is a canonical isomorphism 
%%\[ H^*(\mathfrak{a}_q, O(q))\cong H^*(WO_q).\]
%%Moreover, for any transitive Lie pseudogroup $\Gamma$ on $\R^q$ there is a construction that associates to any principal $\Gamma$-bundle $\mathcal{P}\to M$ a map in cohomology
%%\[ k^{\mathcal{P}}: H^*(\mathfrak{a}(\Gamma), K) \to H^*(M)\]
%%such that:
%%\begin{itemize}
%%\item when $\Gamma= \textrm{Diff}_{\textrm{loc}}(\R^q)$ and codimension $q$ foliations one recovers the  one recovers the map 
%\end{thm}

Incidentally, let us point out here that the obvious map from $\mathfrak{X}(\R^q)$ to $\mathfrak{a}_q$ (taking the infinite jet at the origin) is known to induce isomorphisms in (continuous) Lie algebra cohomology (see~\cite{GF}), and this provides yet another description of the Gelfand-Fuchs cohomology. In fact, this is the original approach of Gelfand and Fuchs.

%%%%%%%%%%%%%%%%%%%%%%%%%%%%
%%%%%%%%%%%%%%%%%%%%%%%%%%%%
%%%%%%%%%%%%%%%%%%%%%%%%%%%%
%%%%%%%%%%%%%%%%%%%%%%%%%%%%
%%%%%%%%%%%%%%%%%%%%%%%%%%%%
\subsection{Cohomology of classifying spaces: the Bott-Shulman model}
\label{More on cohomology of classifying spaces and the abstract characteristic map}
%%%%%%%%%%%%%%%%%%%%%%%%%%%%
%%%%%%%%%%%%%%%%%%%%%%%%%%%%
%%%%%%%%%%%%%%%%%%%%%%%%%%%%
%%%%%%%%%%%%%%%%%%%%%%%%%%%%
%%%%%%%%%%%%%%%%%%%%%%%%%%%%

 With the construction of the abstract characteristic map in mind (see the conclusions of subsection \ref{Abstract characteristic classes for foliations}) we see that, as for Lie groups, what really matters for us is just the cohomology of the classifying spaces rather than the classifying spaces themselves. There are various more explicit models available- each one of them providing a more or less explicit description of the characteristic map. One such model is the Bott-Shulman double complex $\Omega^*(\Sigma^*)$,  defined for any Lie groupoid $\Sigma\tto \BB$. It is based on the so-called nerve of $\Sigma$, whose construction we briefly recall. One considers the spaces $\Sigma^{(p)}\subset \Sigma^p$ of $p$-strings $(g_1, \ldots, g_p)$ of composable arrows of $\Sigma$, related by the maps:
\[ d_i: \Sigma^{(p)}\rightarrow \Sigma^{(p-1)}, \quad d_i(g_1, \ldots, g_p)= \left\{
\begin{array}{cc}
(g_2, \ldots, g_p) & \textrm{if $i= 0$}\\
(g_1, \ldots, g_ig_{i-1}, \ldots, g_p) & \textrm{for $1\leq i \leq p-1$}\\
(g_1, \ldots, g_{p-1})& \textrm{if $i= p$}
\end{array}
\right.
\]
Note that these maps show up already when looking at group cohomology, the differential (\ref{delta-group}) being simply 
\[ \delta= \sum_i (-1)^i d_{i}^{*} .\]
The same formula, but using pull-backs of forms, defines a differential which, together with DeRham differential, forms a double complex
\[ (\Omega^*(\Sigma^{(*)}), \delta, d_{DR}),\]
with associated total complex 
\begin{equation}\label{Tot-BS} 
\textrm{Tot}^k\Omega^*(\Sigma^{(*)}):= \bigoplus_{p+ q= k} \Omega^q(\Sigma^{(p)}), \quad D_{\textrm{tot}}= \delta+ (-1)^p d_{DR} \ \textrm{on}\ \Omega^q(\Sigma^{(p)}).
\end{equation}
We refer to both of them as the {\bf Bott-Shulman complex} of the Lie groupoid $\Sigma\tto \BB$.

\begin{defn}\label{defn:BS-DR} The cohomology of the Bott-Shulman complex is called the {\bf DeRham cohomology} of the Lie groupoid $\Sigma\tto \BB$, denoted $H^{*}_{\textrm{dR}}(\Sigma)$. 
%\[ H^{*}_{\textrm{DeR}}(\G):= H^{*}(\Omega^{\bullet}(\G^{(\bullet)}).\]
% \[ H^{*}_{\textrm{dR}}(\G):= H^{*}(\Omega^{\bullet}(\G^{(\bullet)})).\]
%\[ H^{*}_{\textrm{dRh}}(\G):= H^{*}(\Omega^{\bullet}(\G^{(\bullet)}).\]
%\[ H^{*}_{\textrm{DR}}(\G):= H^{*}(\Omega^{\bullet}(\G^{(\bullet)}).\]
%
\end{defn}

By a folklore theorem, for Hausdorff groupoids, this is isomorphic to the cohomology of $B\Sigma$ (see below). %modulo the Hausdorfness issue mentioned below. 
Furthermore, also the product structure on cohomology can be exhibited directly on the Bott-Shulman complex: 

\begin{equation}\label{eq:cup-BS-complex}
- \cup_{\textrm{tot}} -: \Omega^q(\Sigma^{(p)})\times \Omega^{q'}(\Sigma^{(p')})\to \Omega^{q+q'}(\Sigma^{(p+p')}),\quad 
\omega\cup_{\textrm{tot}} \omega':=  (-1)^{qp'} \textrm{first}_{p}^*(\omega) \wedge \textrm{last}_{p'}^*(\omega'), 
% \quad \textrm{for $\omega\in \Omega^{p, q}, \omega'\in \Omega^{p', q'}$.}
\end{equation}
% \[ c\cup c':= \textrm{first}_{p}^*(c) \wedge \textrm{last}_{p'}^*(c'),\]
where $\textrm{first}_p: \Sigma^{(p+p')}\to \Sigma^{(p)}$ keeps the first $p$ arrows, $\textrm{last}_{p'}$ keeps the last $p'$.
%  and $\wedge$ is the usual wedge product of forms $\Omega^q\times \Omega^{q'}\to \Omega^{q+q'}$.
Of course, the signs are chosen so that $D_{\textrm{tot}}$ satisfies the Leibniz identity w.r.t. the total degree:
\begin{equation}\label{1_Leibniz-total} 
D_{\textrm{tot}}(\omega\cup_{\textrm{tot}}\omega')= D_{\textrm{tot}}(\omega)\cup_{\textrm{tot}}\omega'+ 
(-1)^{k} \omega\cup_{\textrm{tot}} D_{\textrm{tot}}(\omega')
\end{equation}
where $k= p+q$ is the total degree of $\omega\in \Omega^q(\Sigma^{(p)})$. Therefore:
\[(\textrm{Tot}\, \Omega^{*, *}(\Sigma), D_{\textrm{tot}}, \cup_{\textrm{tot}})\] 
becomes a DGA and the isomorphism with 
% identification with 
$H^*(B\Sigma)$ mentioned above
is an isomorphism of algebras.
%  is compatible with corresponding algebra structure. 

%
%Note that the last operation can be written in a form that is more compact and also makes sense when $\G$ is not étale:
%\begin{equation}\label{cup-prod-str} 
%\omega\cup \omega':= \textrm{first}_{p}^*(\omega) \wedge \textrm{last}_{p'}^*(\omega'), 
%\quad \textrm{for $\omega\in \Omega^{p, q}, \omega'\in \Omega^{p', q'}$.}
%\end{equation}
%where $\textrm{first}_p: \G^{(p+p')}\to \G^{(p)}$ keeps the first $p$ arrows and $\textrm{last}_{p'}$ keeps the last $p'$. Moreover, when passing to the total complex (\ref{Tot-BS}), there is a signed-version of this operation;
%\[ \omega\cup_{\textrm{tot}}\omega':= (-1)^{qp'} \omega\cup \omega'. \]
%% \[ \omega\cup_{\textrm{tot}}\omega':= (-1)^{qp'} \omega\cup \omega',\quad \textrm{for $\omega\in \Omega^{p, q}, \omega'\in \Omega^{p', q'}$.}\]
%Of course, the signs are chosen so that $D_{\textrm{tot}}$ satisfies the Leibniz identity w.r.t. the total degree:
%\begin{equation}\label{Leibniz-total} 
%D_{\textrm{tot}}(\omega\cup_{\textrm{tot}}\omega')= D_{\textrm{tot}}(\omega)\cup_{\textrm{tot}}\omega'+ 
%(-1)^{k} \omega\cup_{\textrm{tot}} D_{\textrm{tot}}(\omega')
%\end{equation}
%where $k= p+q$ is the total degree of $\omega\in \Omega^q(\G^{(p)})$. Therefore one obtains a DGA:
%\[(\textrm{Tot}^k \Omega^{\bullet, \bullet}, D_{\textrm{tot}}, \cup_{\textrm{tot}}).\] % $ becomes a DGA.
%

\medskip

Recall also that, next to the maps $d_i$ there are also the degeneracy maps
\[ s_i: \Sigma^{(p)}\to \Sigma^{(p+1)}, \quad s_i(g_1, \ldots, g_p)= (\ldots, g_i, 1, g_{i+1}, \ldots)\quad (0\leq i\leq p)\] 
All together, they form the {\bf nerve of $\Sigma$}, which is a simplicial manifold, in the sense that the simplicial identities
%\[ d_id_j= d_{j-1}d_i,  \quad \textrm{if $i< j$},\]
%\[ s_i s_j= s_j s_{i-1} \quad \textrm{if $i> j$}\]
\begin{equation}\label{simplicial-identities}
\left\{ 
\begin{array}{cc}
d_id_j= d_{j-1}d_i, & \textrm{if $i< j$}\\
s_i s_j= s_j s_{i-1} & \textrm{if $i> j$}
\end{array}
\right.
\quad 
d_i s_j= \left\{ \begin{array}{cc}
                              s_{j-1}d_i & \textrm{if $i<j$}\\
                              \textrm{id} & \textrm{if $i\in \{j, j+1\}$}\\
                              s_j d_{i-1} & \textrm{if $i> j+1$}
                         \end{array}
                \right.
\end{equation}
are satisfied. The main point is that to any simplicial manifold one can associate a space, called geometric realization - a construction which applied to the nerve of $\Sigma$ produces precisely $B\Sigma$. See~\cite{SEGAL}.

%
%\medskip

%\medskip
%\begin{equation}\label{eq:cup-BS-complex}
%\cdot \cup \cdot: \Omega^{q_1}(\G^{(p_1)})\times \Omega^{q_2}(\G^{(p_2)})\to \Omega^{q_1+q_2}(\G^{(p_1+p_2)})
%\end{equation}
%\[ c\cup c':= \textrm{first}_{p_1}^*(c) \wedge \textrm{last}_{p_2}^*(c'),\]
%where $\textrm{first}_{p_1}: \G^{(p_1+p_2)}\to \G^{(p_1)}$ keeps the first $p_1$ arrows, $\textrm{last}_{p_2}$ keeps the last $p_2$ and $\wedge$ is the usual wedge product of forms $\Omega^{q_1}\times \Omega^{q_2}\to \Omega^{q_1+q_2}$.
%
%

We are interested in the case of étale groupoids $\mathcal{G}\tto \BB$ such as $\Ger(\Gamma)$ - where there are a couple of remarks to be made.

\begin{rmk}
So far we have ignored one issue: some of our groupoids, especially the ones that use the germ topology such as $\Ger(\Gamma)$, have a space of arrows which is a non-Hausdorff manifold. Often one just ignores at first this aspect, carries on and returns later adapting the arguments/constructions to include the non-Hausdorff case. This affects our discussions here because the folklore theorem that the Bott-Shulman complex of $\G$ computes the cohomology of $B\G$ requires $\G$ to be Hausdorff. The reason is very simple: for non-Hausdorff manifolds $M$, differential forms still provide a resolution of the constant sheaf, but it is no longer a ``good" resolution (not even acyclic); in other words, while there still is a canonical map from DeRham cohomology to sheaf cohomology (with coefficients in $\R$, say), this may fail to be an isomorphism if $M$ is not Hausdorff. However, pin-pointing the problem also indicates the way out: just build an analogue of the Bott-Shulman complex which uses ``good resolutions". However, as indicated above, for the reader who is not comfortable with non-Hausdorff manifolds one may just assume first that we work only with pseudogroups for which $\Ger(\Gamma)$ is Hausdorff. 
\end{rmk}

The other remark to be made is that, 
% for groupoids of type $\Ger(\Gamma)\tto \BB$, or 
% when $\G\tto \BB$ is étale then 
for étale groupoids, there is more structure available that allows one to re-interpret the Bott-Shulman complex (and variations if it). Very briefly: the sheaves of differential forms on the base carry an action of the groupoid (using the germs (\ref{basic-action-for-étale}) induced by arrows). This brings us to another approach to the cohomology of $B\G$ which, in this context, was first considered by Haefliger: via $\G$-sheaves and their cohomology. 

\subsection{Cohomology of classifying spaces: sheaf theoretical approach}

\medskip

The notion of $\Gamma$-sheaf for a pseudogroup $\Gamma$ over $\BB$ should be clear: 

\begin{defn} Given a pseudogroup $\Gamma$ over $\BB$, 
a $\Gamma$-sheaf is any sheaf $\mathcal{S}$ over $\BB$ together with an action of
$\Gamma$ on $\mathcal{S}$, i.e. a collection of maps 
\[\phi^*: \mathcal{S}(V)\to \mathcal{S}(U),\quad U, V \text{ opens in } \BB\]
- one for any element $\phi: U\to V$ of $\Gamma$ - satisfying the functoriality condition 
\[ \phi^*\circ \psi^*= (\psi\circ \phi)^*\]
for any $\phi, \psi\in \Gamma$ composable.
We denote by $\mathsf{Sh}(\Gamma)$ the resulting category of $\Gamma$-sheaves (where morphisms are morphisms of sheaves that commute with all $\phi^*$'s). Similarly, we introduce the category $\mathsf{Ab}(\Gamma)$ of abelian $\Gamma$-sheaves, or the the category $\mathsf{Vect}_{\R}(\Gamma)$ of $\Gamma$-sheaves of vector spaces. 
\end{defn}

%Note that, in some sense, we are enlarging the lattice of opens $\mathcal{O}_{\BB}$ in $\BB$ by $\Gamma$. We interpret $\mathcal{O}_{\BB}$ as the category with the opens $U\subset \BB$ as objects, and inclusions $U\hookrightarrow V$ as morphisms- so that a (pre-)sheaf on $\BB$ can be interpreted as a functor $\mathcal{O}_{\BB}\to \mathsf{Set}$
%into sets. Any $\Gamma$ gives rise to a similar, but larger category 
%% $\mathcal{O}_{\BB}^{\Gamma}$, 
%$\mathcal{O}_{\BB}(\Gamma)$, 
%with the same objects, but where the morphisms are all the smooth maps $\phi: U\to V$ between opens in $\BB$ such that, as a map from $U$ to $\phi(U)$, $\phi$ belongs to $\Gamma$. With this, $\Gamma$-sheaves can be interpreted as contravariant functors
%% \[ \mathcal{S}: \mathcal{O}_{\BB}^{\Gamma}\to \mathsf{Set}\]
%\[ \mathcal{S}: \mathcal{O}_{\BB}(\Gamma)\to \mathsf{Set}\]
%satisfying the gluing condition. Of course, all classes of sheaves that are natural (with respect to diffeomorphisms) are automatically $\Gamma$-sheaves for any $\Gamma$. One example to have in mind are differential forms.  
%
%

Note that, in some sense, we are enlarging the lattice of opens $\mathcal{O}p(\BB)$ in $\BB$ by $\Gamma$. We interpret $\mathcal{O}p(\BB)$ as the category with the opens $U\subset \BB$ as objects, and inclusions $U\hookrightarrow V$ as morphisms- so that a (pre-)sheaf on $\BB$ can be interpreted as a functor $\mathcal{O}p(\BB)\to \mathsf{Set}$. Any pseudogroup $\Gamma$ gives rise to a similar, but larger category 
% $\mathcal{O}_{\BB}^{\Gamma}$, 
$\mathcal{O}p_{\Gamma}(\BB)$, 
with the same objects, but where the morphisms are all the smooth maps $\phi: U\to V$ between opens in $\BB$ such that, as a map from $U$ to $\phi(U)$, $\phi$ belongs to $\Gamma$. Intuitively, one should think of $\mathcal{O}p_{\Gamma}(\BB)$ as representing the lattice of opens in the quotient space $\BB/\Gamma$ of $\Gamma$-orbits - where the {\bf orbit} of $x\in \BB$ is defined as the set of $y\in \BB$ such that $\phi(x)=y$ for some $\phi\in \Gamma$. With this, $\Gamma$-sheaves can be interpreted as contravariant functors
% \[ \mathcal{S}: \mathcal{O}_{\BB}^{\Gamma}\to \mathsf{Set}\]
\[ \mathcal{S}: \mathcal{O}p_{\Gamma}(\BB)\to \mathsf{Set}\]
satisfying the gluing condition. Of course, the sheaves of sections of natural bundles~\cite{NIJENHUIS} are automatically $\Gamma$-sheaves for any $\Gamma$. One example to have in mind are differential forms. 

Sheaf cohomology of topological spaces $\BB$ arises taking the right derived functors of the global sections functor $\mathcal{O}p(\BB)\to \mathsf{Ab}$. For $\Gamma$-sheaves, one considers invariant global sections:

\begin{defn} For a $\Gamma$-sheaf $\mathcal{S}$, a global section $s\in \mathcal{S}(\BB)$ is called invariant if, for any $\phi: U\to V$ in $\Gamma$, 
\[ \phi^*(s|_{V})= s|_{U}.\]
We denote by $\mathcal{S}^{\textrm{inv}}(\BB)$ the space of such sections. 
\end{defn}

It is not difficult to see that the category $\mathsf{Ab}(\Gamma)$ of abelian $\Gamma$-sheaves is abelian and $\mathcal{S}\mapsto \mathcal{S}^{\textrm{inv}}(\BB)$ is a left-exact functor; one can also show that $\mathsf{Ab}(\Gamma)$ has enough injectives (see~\cite{HAEFLIGERNOTES}).

\begin{defn} Given a pseudogroup $\Gamma$ over $\BB$, the resulting right derived functors of $\mathcal{S}\mapsto \mathcal{S}^{\textrm{inv}}(\BB)$ are denoted
\[ H^{p}(\Gamma, \cdot): \mathsf{Ab}(\Gamma)\to \mathsf{Ab}, \quad
\mathcal{S}\mapsto H^{p}(\Gamma, \mathcal{S}).\]
$H^p(\Gamma, \mathcal{S})$ is called the {\bf $p$-cohomology group} of $\Gamma$ with coefficients in $\mathcal{S}$. 
\end{defn}

Hence, in order to compute the cohomology with coefficients in $\mathcal{S}$ explicitly one needs an injective resolution 
\[
0\to \mathcal{S}\to \mathcal{I}^*
\]
 by $\Gamma$-sheaves of $\mathcal{S}$ and $H^{*}(\Gamma, \mathcal{S})$ is the cohomology of the complex: 
%\[
% 0\to (\mathcal{I}^\bullet)^{\textrm{inv}}(\BB)
%\]  
\[
 0\to \mathcal{I}^{*, \textrm{inv}}(\BB)
\]  
For the constant sheaf $\mathcal{S}= \mathbb{R}$ we simplify the notation to $H^{*}(\Gamma)$. As for usual cohomology, concrete models are obtained using various injective resolutions of $\mathbb{R}\in \mathsf{Ab}(\Gamma)$. And this provides explicit models for $H^*(B\Gamma)$ since:

\begin{thm}\label{thm-Haefl-conj} For any pseudogroup $\Gamma$, $H^*(\Gamma)$ is canonically isomorphic to $H^*(B\Gamma)$. 
\end{thm}

It is interesting to rewrite the entire discussion about cohomology in terms of  
$\Ger(\Gamma)$ rather than $\Gamma$ itself. Of course, an outcome is a slight generalisation to the more general context of étale groupoids $\G\tto \BB$. 
One way to proceed is to consider the analogue of $\mathcal{O}p_{\Gamma}(\BB)$ known as the embedding category of $\G$ and denoted $\textsf{Emb}(\mathcal{G})$.  Its objects are opens $U\subset M$, while arrows from $U$ to $V$ are bisections $\sigma:U\to \mathcal{G}$ with the property that $t(\sigma(U))\subset V$. The product of $\sigma_1: U_1\to \mathcal{G}$ from $U_1$ to $U_2$ with $\sigma_2: U_2\to \mathcal{G}$ from $U_2$ to $U_3$ is given by 
\[
\sigma_2\cdot \sigma_1 (x)=\sigma_2(\sigma_1(x)) \cdot \sigma_1(x).
\]
%\begin{defn}
This gives rise to the resulting categories of $\G$-sheaves  $\mathsf{Sh}(\G)$, $\mathsf{Ab}(\G)$, etc. Also the discussion about cohomology carries on without any change, giving rise to 
\[ H^*(\G, \cdot): \mathsf{Ab}(\G)\to \mathsf{Ab}\]
and the analogue of Theorem \ref{thm-Haefl-conj} continues to hold; see \cite{IEKE}.
%
% In order to compute the cohomology with coefficients in $\mathcal{S}$ explicitly one needs an injective resolution 
%\[
%0\to \mathcal{S}\to \mathcal{I}^\bullet
%\]
% by $\mathcal{G}$-sheaves of $\mathcal{S}$ (for a proof of the fact that the category of $\mathcal{G}$-sheaves has enough injectives, we refer to~\cite{HAEFLIGERNOTES}). Then  $H^{\bullet}(\Gamma, \mathcal{S})$ is the cohomology of the complex: 
%%\[
%% 0\to (\mathcal{I}^\bullet)^{\textrm{inv}}(\BB)
%%\]  
%\[
% 0\to \mathcal{I}^{\bullet, \textrm{inv}}(\BB)
%\]  

\medskip

We would like to emphasize that, even if one is interested only in the case $\G= \Ger(\Gamma)$ for pseudogroups $\Gamma$, the point of view of étale groupoids provides extra-insight. For instance, the usual interpretation of (standard) sheaves $\mathcal{S}$ on a space $\BB$ as étale spaces $\hat{\mathcal{S}}\to \BB$ (made of germs of sections)
has a now straightforward generalization: $\G$-shaves become étale spaces $\hat{\mathcal{S}}\to \BB$ together with a (continuous) action of $\G$ on $\hat{\mathcal{S}}$ - say from the right. In particular, any arrow of $\G$, $g: x\to y$, induces an action by $g$, $g^*: \mathcal{S}_y\to \mathcal{S}_x$. While sections of the sheaf correspond to sections of $\hat{\mathcal{S}}\to \BB$, for such a (global) section $s$ the meaning of invariance should be clear: 
\[ g^*(s(y))= s(x) \]
for any arrow $g: x\to y$ of $\G$.

\medskip

Another illustration of the use of étale groupoids is the fact that they provide bar-type complexes computing the cohomology. %, similar to the case of discrete groups. 
Ultimately, this is what allows one to relate this cohomology to the Bott-Shulman complex (and to overcome the non-Hausdorness problem). Furthermore, this also provides a framework that puts together the characteristic classes for foliations with the ones for flat bundles; note here that, when $\G$ is just a discrete group (interpreted as an étale groupoid over a point), one just recovers the usual group cohomology (Example~\ref{exam-discrete-groups}).

%
%
%Actually, this can be further understood The sheaf of sections of a representation $E\to M$ of $\mathcal{G}\rightrightarrows M$ is always a $\mathcal{G}$-sheaf. Since $\mathcal{G}$ is étale, then  $\mathcal{G}$ admits $\Lambda^qT^*M$, $q\in \mathbb{N}$, as a representation; consequently, the sheaf $\Omega^q_M$ of differential forms of degree $q$ on $M$ is a $\mathcal{G}$-sheaf. Explicitly, if $g\in \mathcal{G}$ and $\omega_{s(g)}\in (\Lambda^qT^*M)_{s(g)}$, then one sets
%\[
%g\cdot \omega_{s(g)}=(\sigma_g^{-1})^*(\omega_{s(g)})
%\]
%where $\sigma_g$ is any bisection around $g$.
%]
%

To describe the bar-type complexes we follow~\cite{HAEFLIGER}.
%Recall that, for any groupoid $\G\tto \BB$, one denotes by $\G^{(p)}$ the space of strings 
% $(g_1, \ldots, g_p)$  
As in the previous subsection, we look at the space of $p$ composable arrows
% \[ \cdot \stackrel{\g_1}{\lmap} \cdot \ldots \cdot \stackrel{\g_p}.\]
\[ \mathcal{G}^{(p)}=\mathcal{G}\tensor[_s]{\times}{_t}\dots \tensor[_s]{\times}{_t}\mathcal{G}.\] 
A {\bf $p$-cochains on $\G$} with values in a $\mathcal{G}$-sheaf $\mathcal{S}$
is any global section of the pull-back $t^*\mathcal{S}$ of the sheaf $\mathcal{S}$ to
$\mathcal{G}^{(p)}$
% =\mathcal{G}_s\times_t\dots_s\times_t\mathcal{G}$ ($p$ strings of composable arrows) via 
where $t: \mathcal{G}^{(p)}\to \BB$ takes the target of the first element of a $p$-string.
The corresponding space is denoted by $C^p(\mathcal{G}, \mathcal{S})$. Hence
\begin{equation}\label{equation:Cp(G,E)} 
C^p(\mathcal{G}, \mathcal{S})= \Gamma(\G^{(p)}, t^*\mathcal{S}).
\end{equation}
Equivalently, in terms of the stalks $\mathcal{S}_x$ (i.e. the associated étale spaces of germs $\hat{\mathcal{S}}\to \BB$), a $p$-cochain $c\in C^p(\mathcal{G}, \mathcal{S})$ is a map 
\[ \G^{(p)} \ni (g_1, \ldots, g_p)\mapsto c(g_1, \ldots, g_p)\in \mathcal{S}_{t(g_1)}\]
which is continuous.
%\footnote{With the last viewpoint in mind it may be tempting to use the term ``continuous cochains" but, as in standard sheaf theory, continuity is built in: the topology on $\hat{\mathcal{S}}$ being defined precisely so that all the sections of $\mathcal{S}$ are automatically continuous (and one never says "let $s\in \mathcal{S}(U)$ be a continuous section"). Even more, when working over manifolds, the topology on $\hat{\mathcal{S}}$ can be promoted to a smooth structure so that all the sections $\mathcal{S}$ are automatically even differentiable.}.
%Given a $\mathcal{G}$-sheaf $\mathcal{S}$, one defines {\bf $p$-cochains with values in $\mathcal{S}$} as the global sections of the pullback $t^*S\to \mathcal{G}^{(p)}$, where $t$ denotes the target of the first element of a $p$-string in $\mathcal{G}^{(p)}$. 
%The corresponding space is denoted by $C^p(\mathcal{G}, S)$;
%Hence a $p$-cochain $c\in C^p_c(\mathcal{G}, S)$ is a continuous map 
%\[ \G^{(p)} \ni (g_1, \ldots, g_p)\mapsto c(g_1, \ldots, g_p)\in \mathcal{S}_{t(g_1)}.\]
 The {\bf groupoid differential} 
 \[
\delta: C^p(\mathcal{G},\mathcal{S})\to C^{p+1}(\mathcal{G}, \mathcal{S})
\]
is defined by the adaptation of formula (\ref{delta-group}), where we use the action by $g_1$ so that the resulting formula lands in the desired space (germs at $t(g_1)$):
 \begin{align}\label{delta-group-sheaf}
\delta(c)(g_1, \dots, g_{p+1})= & g_1\cdot c(g_2, \dots, g_{p+1})+ \sum\limits^p_{i=1}(-1)^ic(g_1, \dots, g_i\cdot g_{i+1}, \dots, \dots g_{p+1})\\
& +(-1)^{p+1}c(g_1, \dots, g_p)\nonumber 
\end{align}
%\[
%\delta c(g_1, \dots, g_{p+1})=g_1\cdot c(g_2, \dots, g_{p+1})+\sum\limits^p_{i=1}(-1)^ic(g_1, \dots, g_i\cdot g_{i+1}, \dots g_{p+1})+(-1)^{p+1}c(g_1, \dots, g_p)
%\]
%% The map $\delta$ squares to zero. 

%\begin{defn} Given a pseudogroup $\Gamma$ over $\BB$, a {\bf continuous $p$-cochain} on $\Gamma$ is any $p$-cochain on the germ groupoid $\Ger(\Gamma)\tto \BB$. Furthermore, we will be using the notations $C^{\bullet}_{\textrm{cont}}(\Gamma, \mathcal{S})$ for the space $C^{\bullet}(\Ger(\Gamma), \mathcal{S})$, for any $\Gamma$-sheaf $\mathcal{S}$.
%\footnote{in contrast with the previous footnote, calling $H^{\bullet}(\Gamma)$ the "continuous cohomology" of $\Gamma$ would make more sense (and the differentiable cohomology of $\Gamma$ will come a bit later). However, in the case of Lie groupoids, and even in the case of Lie groups $G$, one has to be careful with the use of the term "cohomology of $G$", as it may mean many different things- such as discrete, continuous, differentiable or even DeRham cohomology.}
%%\[ C^{\bullet}_{\textrm{cont}}(\Gamma, \mathcal{S}):= C^{\bullet}_{\textrm{cont}}(\Ger(\Gamma), \mathcal{S}), \quad H^{\bullet}_{\textrm{cont}}(\Gamma, \mathcal{S}):= H^{\bullet}_{\textrm{cont}}(\Ger(\Gamma), \mathcal{S})\]
%%(for any $\Gamma$-sheaf $\mathcal{S}$). 
%\end{defn}

Notice that, if $\Phi:\mathcal{S}^1\to \mathcal{S}^2$ is a morphism of $\mathcal{G}$-sheaves, there is an induced morphism of complexes $C^*(\mathcal{G},S^1)\to C^{*}(\mathcal{G}, S^2)$.

\medskip
In particular, for pseudogroups one obtains the following notion of cocycle, which would not have been so natural without the viewpoint of the germ groupoid.

\begin{defn}\label{sdefinition:cont-cochains} Given a pseudogroup $\Gamma$ over $\BB$, by a {\bf continuous $p$-cochain} on $\Gamma$ we mean any $p$-cochain on the germ groupoid $\Ger(\Gamma)\tto \BB$, with resulting complex denoted
\[ C^{*}_{\textrm{cont}}(\Gamma, \mathcal{S}):= C^{*}(\Ger(\Gamma), \mathcal{S}),\]
for any $\Gamma$-sheaf $\mathcal{S}$.
%\footnote{In contrast to the previous footnote, calling $H^{*}(\Gamma)$ the "continuous cohomology" of $\Gamma$ would make more sense since the  reference is to $\Gamma$ and not to $\Ger(\Gamma)$; the terminology would also be in line with the notion of differentiable cohomology of $\Gamma$ that we will be discussing a bit later.}
% However, in the case of Lie groupoids, and even in the case of Lie groups $G$, 
% one has to be careful with the use of the term "cohomology of $G$", as it may 
% mean many different things- such as discrete, continuous, differentiable or even 
% DeRham cohomology.
%\[ C^{\bullet}_{\textrm{cont}}(\Gamma, \mathcal{S}):= C^{\bullet}_{\textrm{cont}}(\Ger(\Gamma), \mathcal{S}), \quad H^{\bullet}_{\textrm{cont}}(\Gamma, \mathcal{S}):= H^{\bullet}_{\textrm{cont}}(\Ger(\Gamma), \mathcal{S})\]
%(for any $\Gamma$-sheaf $\mathcal{S}$). 
\end{defn}

Returning to general étale groupoids $\G\tto \BB$, it is not so easy to ``resolve" an arbitrary sheaf $\mathcal{S}$ by a resolution consisting of sheaves which are injective as $\G$-sheaves (or just acyclic w.r.t. the cohomology of $\G$-sheaves). One reason is that the injectivity (or acyclicity) as a sheaf over $\BB$ does not imply the injectivity (or acyclicity) as a $\G$-sheaf. The bar-complexes arise by ``resolving" the sheaves in steps. The outcome is the following:

\begin{thm}\label{double_cpx}
Let $\mathcal{S}$ be a $\mathcal{G}$-sheaf and 
\[
0\to \mathcal{S}\to \mathcal{F}^*
\]
a resolution by $\mathcal{G}$-sheaves which are acyclic as sheaves on $\BB$ and also when pulled back to the spaces $\G^{(p)}$ via $t$. Then, for all $k$,
\begin{equation}
H^k(\mathcal{G}, \mathcal{S})\cong H^k(C^*(\mathcal{G}, \mathcal{F}^*))
\end{equation}
where the right hand side is the $k$-th cohomology group of the double complex $\left((C^p(\mathcal{G}, \mathcal{F}^q), \delta, d)\right)_{p,q}$ and $d$ is induced by the differential of the resolution $\mathcal{F}^*$.
\end{thm}

See~\cite{HAEFLIGERNOTES} for a proof; in~\cite{HAEFLIGER}, the above theorem is used to define continuous cohomology via the Godement resolution of a sheaf.  The advantage of the Godement resolution
\[
0\to \mathcal{S}\to \mathcal{C}^0\mathcal{S}\to \mathcal{C}^1\mathcal{S}\to \mathcal{C}^2\mathcal{S}\dots
\]
(defined over any space) is that it is flabby (hence also acyclic) and is preserved by taking pull-backs via étale maps - hence the previous theorem applies. 
%
%Recall that this has the form
%\[
%0\to \mathcal{S}\to \mathcal{C}^0\mathcal{S}\to \mathcal{C}^1\mathcal{S}\to \mathcal{C}^2\mathcal{S}\dots
%\]
%Here, $\mathcal{C}^0\mathcal{S}$ is the sheaf of discontinuous sections of the étale space $S$, $C^1\mathcal{S}$ is the same object for the sheafification of the quotient between $\mathcal{C}^0\mathcal{S}$ and $\mathcal{S}$ and the sheaves $\mathcal{C}^q\mathcal{S}$, $q\geq 1$, are constructed inductively in a similar way. 
%

The above theorem only needs the base $\BB$ to be a topological space. When $\BB$ is a manifold and we deal with real coefficients, we have also the de Rham resolution
\begin{equation}\label{DeRham-res}
0\to \mathbb{R}_\BB\to \Omega^0_\BB\to \Omega^1_\BB\to \dots
\end{equation}
by the $\mathcal{G}$-sheaves of differential forms. Since the maps $t: \G^{(p)}\to \BB$ are étale, they pull-back the sheaf of differential forms on $\BB$ to sheaves of differential forms on the manifolds $\G^{(p)}$. Hence one obtains the following, which gives a re-interpretation of the Bott-Shulman complex:
\[ \Omega^{q}(\G^{(p)})= \mathcal{C}^{p}(\G, \Omega_{\BB}^{q}). \]
Recall that the sheaf of differential forms of any degree is acyclic whenever we work over Hausdorff manifolds. Consequently, one obtains the relationship with DeRham cohomology mentioned above (see Definition \ref{defn:BS-DR}) ~\cite{HAEFLIGER, HAEFLIGERNOTES}:

\begin{lem}\label{res_diff_forms}
For any étale groupoid $\mathcal{G}\rightrightarrows \BB$ there is an canonical map 
\[ \iota: H^{*}_{\rm dR}(\G)\to H^{*}(\G, \R)\]
and, if $\G$ is Hausdorff, this map is an isomorphism.   
%For étale groupoids $\mathcal{G}\rightrightarrows \BB$ which are Hausdorff, the Bott-Shulman complex computes the 
%cohomology of $\G$ with coefficients in $\mathbb{R}$.
\end{lem}

In the non-Hausdorff case the situation is a bit more delicate and 
Theorem~\ref{double_cpx} may not apply to the DeRham resolution; hence in that case one obtains a map between the two cohomologies, but it may fail to be an isomorphism.

%%is a fine resolution, it is not a flabby one. In general, an analogue of Theorem~\ref{double_cpx} one can
%%%not replace the flabby resolution $\mathcal{F}^\bullet$ by a fine one. This happens 
%%where the resolution $\mathcal{F}^\bullet$ is flabby holds only if $\mathcal{G}$ is paracompact and Hausdorff. 
%Let us however state the positive result ~\cite{HAEFLIGER, HAEFLIGERNOTES}:
%
%\begin{lem}\label{res_diff_forms}
%For étale groupoids $\mathcal{G}\rightrightarrows \BB$ which are Hausdorff, the Bott-Shulman complex computes the 
%cohomology of $\G$ with coefficients in $\mathbb{R}$.
%\end{lem}
%
%% If $\G$ is not Hausdorff one has instead a canonical map from between the two cohomologies. 
%

\begin{exm}\label{exam-discrete-groups} It is clear that when $\G= G$ is a discrete group viewed as an étale groupoid over a point, one recovers the usual cohomology of $G$ with coefficients in $G$-modules.

At the other extreme one has that, for any manifold $M$, the cohomology of the unit groupoid $M\tto M$ (consisting only of units)
is the standard sheaf cohomology of $M$. 

A related example is the groupoid $M_{\mathcal{U}}$ associated to an open cover $\mathcal{U}$ of $M$ (see Remark \ref{rk:gpd-GU}). Due to the gluing condition on sheaves, there is an immediate equivalence of categories
\[ \mathsf{Ab}(M_{\mathcal{U}})\cong \mathsf{Ab}(M)\]
and it is not difficult to see that also the cohomology is preserved:
\[ H^*(M, \mathcal{S})\cong H^*(M_{\mathcal{U}}, \mathcal{S}_{\mathcal{U}}).\]
\end{exm}

\medskip

\begin{exm}[back to the abstract characteristic map (\ref{eq:abstr-k-G})] % (\ref{eq:abstr-k-G})] 
\label{exm:back:charact;map} 
The very last example gives, as promised, a more concrete model for the abstract characteristic map (\ref{eq:abstr-k-G}). To achieve that, recall that any $\G$-cocycle $\gamma= \{\gamma_{ij}\}$ over the cover $\mathcal{U}$ can be interpreted as a morphism of groupoids $\gamma: M_{\mathcal{U}}\to \G$ (cf. Remark \ref{rk:gpd-GU}). Furthermore, it is clear that the Bott-Shulman complex (and the DeRham cohomology) of groupoids is functorial with respect to groupoid morphisms. Hence, a cocycle $\gamma$ induces
\begin{align*}
\gamma^*: C^p(\mathcal{G}, \Omega^q_{\BB})&\to C^p(M_\mathcal{U}, \Omega^q_M)\\
c &\to \gamma^*(c)
\end{align*}
where, explicitly,
\begin{equation}\label{dR_gamma}
{\gamma}^*(c)_{i_0, \dots i_p}(x)={f}^*_i(c(\gamma_{i_{p-1}i_p}(x), \dots,\gamma_{i_0i_1}(x))).
\end{equation}
Here, we are making use of the fact that a groupoid cochain in $M_\mathcal{U}$ is the same thing as a \v{C}ech cochain.
Passing to cohomology, 
\begin{equation}\label{eq:abstr-k-G-dR}
 \gamma^{*}: H^{*}_{\textrm{dR}}(\G) \to H^{*}(M) .
 \end{equation}
When $\G$ is Hausdorff so that we can apply Lemma \ref{res_diff_forms} (see also Theorem \ref{thm-Haefl-conj}), this becomes precisely the promised explicit description of the characteristic map (\ref{eq:abstr-k-G}). In the general case, this is the composition of $\iota: H^{*}_{\textrm{dR}}(\G) \to H^*(\G, \R)\cong H^*(B\G)$ with the characteristic map (\ref{eq:abstr-k-G}) \[\kappa^{\mathcal{P}}_{\textrm{abs}}: H^*(B\G)\to H^*(M).\]

%
%The very last example makes the abstract characteristic map (\ref{eq:abstr-k-G}) a bit more transparent. Indeed, for an étale groupoid $\G\tto M$, while a $\G$-cocycle can be interpreted as a morphism of groupoids $M_{\mathcal{U}}\to \G$ (cf. Remark 
%\ref{rk:gpd-GU}) one has an induced map in cohomology
%\[ H^{*}(\G) \to H^{*}(M_{\mathcal{U}}) .\]
%With the identifications we discussed, this is precisely the characteristic map (\ref{eq:abstr-k-G}).
%
\end{exm}

\begin{rmk}[product structure]\label{rk:product-str} One of the advantages of the bar-complexes is that they allow one to 
% Let us also point out that the Bott-Shulman complex and bar-complexes allow one to 
exhibit explicitly the product structures present in cohomology. 
For instance, for any three abelian $\G$-sheaves $\mathcal{S}_1$, $\mathcal{S}_2$ and
$\mathcal{S}$ and any map of abelian $\G$-sheaves $w: \mathcal{S}_1\otimes \mathcal{S}_2\to \mathcal{S}$, 
one has an induced operation  
\[ -\cup_{w} -:  C^{p}(\G, \mathcal{S}_1)\times C^{p'}(\G, \mathcal{S}_2)\to C^{p+p'}(\G, \mathcal{S}), \]
\[ (c_1\cup_w c_2)(g_1, \ldots, g_{p+p'})= w\left( c(g_1, \ldots, g_p),\ g_1\cdot\ldots \cdot g_p \cdot c'(g_{p+1}, \ldots, g_{p+ p'})\right)\in \mathcal{S}_{t(g_1)}.\]
In particular:
\begin{itemize}
\item when $\mathcal{S}_1= \mathcal{S}_2= \mathcal{S}=\mathbb{R}$ and $\omega$ is the identity, one obtains the so-called cup-product 
\[  -\cup -:  C^{p}(\G)\times C^{p'}(\G)\to C^{p+p'}(\G), \]
making $C^{*}(\G)$ into a graded algebra. Together with $\delta$, it is actually a differential graded algebra, i.e. $\delta$ satisfies the derivation rule
\[ \delta(c\cup c')= \delta(c)\cup c'+ (-1)^p c\cup \delta(c');\] 
\item similarly, when $\mathcal{S}_1$ is $\mathbb{R}$ and $\mathcal{S}_2$ is an arbitrary $\mathcal{S}$ 
%and $\omega: \mathcal{S}\otimes\mathbb{R}\to \mathcal{S}$, 
one obtains that $C^{p}(\G, \mathcal{S})$  comes with a left and a right action of $C^{*}(\G)$, making it into a (differential) graded module;
%\item when $\mathcal{S}_1= \mathbb{R}$ and $\mathcal{S}_2= \mathcal{S}$ arbitrary, obtaining that  $C^{p}(\G, \mathcal{S})$ is a (differential) graded module over $C^{\bullet}(\G)$.
\item when $\mathcal{S}_1= \Omega^{q}_{\BB}$, $\mathcal{S}_2= \Omega^{q'}_{\BB}$ are sheaves of differential forms and $w$ is the usual wedge operation, one obtains an induced operation 
\begin{equation}\label{eq:cup-BS-complex}
- \cup -: \Omega^q(\G^{(p)})\times \Omega^{q'}(\G^{(p')})\to \Omega^{q+q'}(\G^{(p+p')}).
% \quad c\cup c':= \textrm{first}_{p}^*(c) \wedge \textrm{last}_{p'}^*(c'),
\end{equation}
%Or, in a more compact form that makes sense also when $\G$ is not étale:
%\[ c\cup c':= \textrm{first}_{p}^*(c) \wedge \textrm{last}_{p'}^*(c'),\]
%where $\textrm{first}_p: \G^{(p+p')}\to \G^{(p)}$ keeps the first $p$ arrows and $\textrm{last}_{p'}$ keeps the last $p'$. 
\end{itemize}
%  and $\wedge$ is the usual wedge product of forms $\Omega^q\times \Omega^{q'}\to \Omega^{q+q'}$.
It is immediate to see that this is precisely the unsigned version of the total cup-product (\ref{eq:cup-BS-complex}), a version that interacts nicer with $\delta$ and $d$ (rather than the total differential). More precisely one has:
\begin{equation}\label{eq:1_delta-deriv} 
\delta(\omega\cup \omega')= \delta(\omega)\cup \omega'+ (-1)^p \omega\cup \delta(\omega') 
\end{equation}
\begin{equation}\label{eq:1_d-deriv-cup} 
d(\omega\cup \omega')= d(\omega)\cup \omega'+ (-1)^q \omega\cup d(\omega') 
\end{equation}
for all $\omega\in \Omega^{p, q}, \omega'\in \Omega^{p', q'}$. Of course, the last equation is a consequence of the Leibniz rule for the usual wedge product:
\begin{equation}\label{eq:1_d-deriv} 
d(\omega\wedge \omega')= d(\omega)\wedge \omega'+ (-1)^q \omega\wedge d(\omega') . 
\end{equation}

%
%Note that the last operation can be written in a form that is more compact and also makes sense when $\G$ is not étale:
%\begin{equation}\label{cup-prod-str} 
%\omega\cup \omega':= \textrm{first}_{p}^*(\omega) \wedge \textrm{last}_{p'}^*(\omega'), 
%\quad \textrm{for $\omega\in \Omega^{p, q}, \omega'\in \Omega^{p', q'}$.}
%\end{equation}
%where $\textrm{first}_p: \G^{(p+p')}\to \G^{(p)}$ keeps the first $p$ arrows and $\textrm{last}_{p'}$ keeps the last $p'$. Moreover, when passing to the total complex (\ref{Tot-BS}), there is a signed-version of this operation;
%\[ \omega\cup_{\textrm{tot}}\omega':= (-1)^{qp'} \omega\cup \omega'. \]
%% \[ \omega\cup_{\textrm{tot}}\omega':= (-1)^{qp'} \omega\cup \omega',\quad \textrm{for $\omega\in \Omega^{p, q}, \omega'\in \Omega^{p', q'}$.}\]
%Of course, the signs are chosen so that $D_{\textrm{tot}}$ satisfies the Leibniz identity w.r.t. the total degree:
%\begin{equation}\label{Leibniz-total} 
%D_{\textrm{tot}}(\omega\cup_{\textrm{tot}}\omega')= D_{\textrm{tot}}(\omega)\cup_{\textrm{tot}}\omega'+ 
%(-1)^{k} \omega\cup_{\textrm{tot}} D_{\textrm{tot}}(\omega')
%\end{equation}
%where $k= p+q$ is the total degree of $\omega\in \Omega^q(\G^{(p)})$. Therefore one obtains a DGA:
%\[(\textrm{Tot}^k \Omega^{\bullet, \bullet}, D_{\textrm{tot}}, \cup_{\textrm{tot}}).\] % $ becomes a DGA.
\end{rmk}

%%%%%%%%%%%%%%%%%%%%%%%%%%%%
%%%%%%%%%%%%%%%%%%%%%%%%%%%%
%%%%%%%%%%%%%%%%%%%%%%%%%%%%
%%%%%%%%%%%%%%%%%%%%%%%%%%%%
%%%%%%%%%%%%%%%%%%%%%%%%%%%%
%%%%%%%%%%%%%%%%%%%%%%%%%%%%
%%%%%%%%%%%%%%%%%%%%%%%%%%%%
%%%%%%%%%%%%%%%%%%%%%%%%%%%%
%%%%%%%%%%%%%%%%%%%%%%%%%%%%
%%%%%%%%%%%%%%%%%%%%%%%%%%%%
\section{Haefliger's differentiable cohomology}\label{Haefliger's differentiable cohomology}
%%%%%%%%%%%%%%%%%%%%%%%%%%%%
%%%%%%%%%%%%%%%%%%%%%%%%%%%%
%%%%%%%%%%%%%%%%%%%%%%%%%%%%
%%%%%%%%%%%%%%%%%%%%%%%%%%%%
%%%%%%%%%%%%%%%%%%%%%%%%%%%%
%%%%%%%%%%%%%%%%%%%%%%%%%%%%
%%%%%%%%%%%%%%%%%%%%%%%%%%%%
%%%%%%%%%%%%%%%%%%%%%%%%%%%%
%%%%%%%%%%%%%%%%%%%%%%%%%%%%
%%%%%%%%%%%%%%%%%%%%%%%%%%%%

%%%%%%%%%%%%%%%%%%%%%%%%%%%%
%%%%%%%%%%%%%%%%%%%%%%%%%%%%
%%%%%%%%%%%%%%%%%%%%%%%%%%%%
%%%%%%%%%%%%%%%%%%%%%%%%%%%%
%%%%%%%%%%%%%%%%%%%%%%%%%%%%
\subsection{Haefliger's differentiable cohomology of pseudogroups}
\label{ssec:Motivation and definition}
%%%%%%%%%%%%%%%%%%%%%%%%%%%%
%%%%%%%%%%%%%%%%%%%%%%%%%%%%
%%%%%%%%%%%%%%%%%%%%%%%%%%%%
%%%%%%%%%%%%%%%%%%%%%%%%%%%%
%%%%%%%%%%%%%%%%%%%%%%%%%%%%

Let us return to the two characteristic maps for foliations: the geometric one $\kappa^{\F}$ from (\ref{eq:k-abstr}) and the abstract one $\kappa^{\F}_{\textrm{abs}}$ from (\ref{eq:abstr-k-Gammaq}). As for principal bundles with corresponding commutative diagram (\ref{char-diag-G-bundles}), and as for discrete bundles
with corresponding commutative diagram (\ref{char-diag-G-bundles-discrete}), one has a similar commutative diagram 
\begin{equation}\label{char-diag-Gamma-bundles}
\begin{tikzcd}
\textrm{GF}_{q}^{*}\arrow[rr]{}{\kappa^{\F}} \arrow[d, swap]{}{\kappa^{\textrm{univ}}}& & H^*(M)\\
H^*(B\Gamma^q)\arrow[rru, swap]{}{\kappa^{\F}_{\rm abs}}& & 
\end{tikzcd}
\end{equation}
Here, the existence of vertical map $\kappa^{\textrm{univ}}$ making the previous diagram commutative can be obtained using the universality of $B\Gamma^q$, or explicitly using the bar-complex (see also below for a different, simpler, description).  We have seen that, in this diagram: 
\begin{itemize}
\item the domain of the abstract characteristic map could be interpreted as the cohomology of the groupoid $\Gamma^q$.
\item the domain $\textrm{GF}_{q}^{*}$ of the geometric characteristic map has several descriptions - e.g. very explicitly using the complex $WO_q$ or, still explicit but more conceptual, using the Lie algebra $\mathfrak{a}_q$ of formal vector fields.
\end{itemize}
Given the striking analogy with the discussion of characteristic classes for flat bundles, there is an obvious question: {\it can one interpret the domain of the geometric characteristic map as a certain ``differentiable cohomology" 
% of $\Gamma_q$, 
% \[GF_{q}^{*}\cong H^{*}_{\textrm{diff}}(\Gamma^q) \]
% $H^{*}_{\textrm{diff}}(\Gamma^q)$, 
\begin{equation}\label{eq:diff-coh-gamma-q} 
H^{*}_{{\rm diff}}(\Gamma^q)
\end{equation}
so that the previous diagram appears as a variation of (\ref{char-diag-G-bundles-discrete}):
\begin{equation}\label{eq: Haefliger-diagram}
\begin{tikzcd}
H^*_{{\rm diff}}(\Gamma^q)\arrow[rr]{}{\kappa^{\F}} \arrow[d, swap]{}{\kappa^{\textrm{univ}}}& & H^*(M)\\
H^*(\Gamma^q)\arrow[rru, swap]{}{\kappa^{\F}_{\rm abs}}& & 
\end{tikzcd}
\end{equation}
combined with a ``van Est isomorphism" 
\begin{equation}\label{eq: Van-Est-Haefliger}
VE: H^*_{{\rm diff}}(\Gamma^q)\overset{\simeq}{\to} \textrm{GF}_{q}^{*} 
\end{equation}}
% The answer is yes; 
The construction of the ``differentiable cohomology (\ref{eq:diff-coh-gamma-q})" was carried out by Haefliger \cite{HAEFLIGER}. However, it seems to be forgotten. It does look a bit ad-hoc in the sense that it is described only for $\Gamma^q$ and furthermore, it is not clear what is the structure that makes the definition work. Our aim in this paper is to clarify this construction and, in particular, provide the general conceptual  framework to which it belongs.

Let us go through Haefliger's definition, also making sure that it makes sense for more general pseudogroups $\Gamma$ - hence defining $H^{*}_{\textrm{diff}}(\Gamma)$. 
Haefliger's idea is very simple: 
while ``continuous cohomology" of $\Gamma$ is, modulo the Hausdorfness issue, computed by the Bott-Shulman complex
\[ \Omega^{q}(\G^{(p)})= C^p(\G, \Omega^{q}_{\BB})\quad (\G= \Ger(\Gamma)),\]
the differentiable cohomology should be computed by a sub-complex
% \[ C^{p}_{\textrm{diff}}(\G, \Omega^{q}_{\BB})\subset C^p(\G, \Omega^{q}_{\BB})\]
consisting of ``differentiable cochains". The notion of ``differentiable cochain"
makes sense for cochains
\[ c\in C^p(\G, \mathcal{E})\]
whenever $\mathcal{E}$ is the sheaf of sections of a smooth vector bundle $E\to \BB$: it means that at any 
\[ g= (g_1, \ldots, g_p)= (\textrm{germ}_{x_1}(\phi_1), \ldots, \textrm{germ}_{x_p}(\phi_p))\in \G^{(p)},\]
the value $c(g)\in E_{t(g_1)}$ depends only on 
\[ j^{\infty}(g):= (j^{\infty}_{x_1}(\phi_1), \ldots, j^{\infty}_{x_p}(\phi_p)),\]
and it is ``smooth in $j^{\infty}(g)$". More precisely, we define
\[ C^{p}_{\textrm{diff}}(\Gamma, E):= C^p(J^{\infty}\Gamma, E), \]
the space of smooth sections of the pull-back of the vector bundle $E$ via 
$t: J^{\infty}\Gamma\to \BB$ that takes the target of the first arrow. We recall once more that the ``smooth structure" on $J^\infty\Gamma$ is the one of profinite dimensional manifold; see the appendix for details. Observe that, to have a profinite dimensional smooth structure on $J^\infty\Gamma$, we need to assume that $\Gamma$ is a Lie pseudogroup, see~\eqref{eq:the-jet-tower} and the following discussion; recall also that we take $J^\infty\Gamma\cong \plim J^l\Gamma$ as part of our definition of Lie pseudogroup.  

Of course, via the infinite jet map $j^{\infty}: \G\to J^{\infty}\Gamma$, we have
\[ C^{p}_{\textrm{diff}}(\Gamma, E)\subset C^{p}(\G, \mathcal{E})\quad (\textrm{$\mathcal{E}$ being the sheaf of sections of $E$}).\]

\begin{defn} \label{defn:diff-coh-gamma} {\bf Haefliger's differentiable cohomology} of a Lie pseudogroup $\Gamma$,
denoted $H^*_{\textrm{diff}}(\Gamma)$, 
% \[ H^\bullet_{\textrm{diff}}(\Gamma),\]
 is the cohomology of the simple complex associated to the double complex
\begin{equation}\label{eq:diff-cplx-3}
C^{p, q}_{\textrm{diff}}(\Gamma):= C^{p}_{\textrm{diff}}(\Gamma, \Lambda^{q}T^*\BB)= C^{p}(J^\infty\Gamma,  \Lambda^{q}T^*\BB),
\end{equation}
which, via the map induced by $j^{\infty}: \G\to J^\infty\Gamma$, is a subcomplex of the Bott-Shulman complex associated to $\G= \Ger(\Gamma)$. 
% , 
% \[ \Omega^q(\G^{(p)})= C^{p}(\G, \Omega^q_\BB).\]
We will denote by $j^*: C^{p, q}_{\textrm{diff}}(\Gamma)\hookrightarrow \Omega^q(\G^{(p)})$ the corresponding inclusion and, similarly the map induced in cohomology
\[ j^*: H^{*}_{\textrm{diff}}(\Gamma)\to H^{*}_{\textrm{dR}}(\G) \quad (\G= \Ger(\Gamma)).\]
Finally, for a $\Gamma$-structure $\mathcal{P}$ on a manifold $M$ (represented by some cocycle $\gamma$, c.f. Definition~\ref{Cocycle}), the composition of $j^*$ with the characteristic map $\gamma^*$ (\ref{eq:abstr-k-G-dR}) is denoted
\begin{equation}\label{eq:dif-charact-map} 
 \kappa^{\mathcal{P}}_{\textrm{diff}}: H^{*}_{\textrm{diff}}(\Gamma)\to H^*(M)
 \end{equation}
and is called \textbf{the differentiable characteristic map} associated to the $\Gamma$-structure $\mathcal{P}$ on M. 
% and the induced differentials will be denoted by $\delta$ (along $p$) and $D$ (along $q$). 
\end{defn}

The definition of the differentiable cohomology above is precisely Haefliger's definition \cite{HAEFLIGER}. Being a subcomplex means that it is preserved by the differentials $\delta$ (along $p$) and $d$ (along $q$) of the Bott-Shulman complex, therefore giving rise to similar differentials 
\begin{equation}\label{eq:diff-cplx-delta-diff} \ % label{eq:diff-cplx-D-diff} 
\delta: C^{p, q}_{\textrm{diff}}(\Gamma)\to C^{p+1, q}_{\textrm{diff}}(\Gamma), \quad d: C^{p, q}_{\textrm{diff}}(\Gamma)\to C^{p, q+1}_{\textrm{diff}}(\Gamma).
\end{equation}

%\begin{equation}\label{eq:diff-cplx-delta-diff} 
%\delta_{\textrm{diff}}: C^{p, q}_{\textrm{diff}}(\Gamma)\to C^{p+1, q}_{\textrm{diff}}(\Gamma)
%\end{equation}
%and 
%\begin{equation}\label{eq:diff-cplx-D-diff} 
%D_{\textrm{diff}}: C^{p, q}_{\textrm{diff}}(\Gamma)\to C^{p, q+1}_{\textrm{diff}}(\Gamma).
%\end{equation}
On the other hand, it is immediate to see that this subcomplex is closed also under the product structure (see Remark \ref{rk:product-str}); actually, the induced cup-product on the differentiable complex fits precisely the scheme described in Remark \ref{rk:product-str}, but for the groupoid $J^\infty\Gamma$ and with coefficients in vector bundle representations rather than in sheaves. Therefore, $C^{*,*}_{\textrm{diff}}(\Gamma)$ inherits from the Bott-Shulman complex the same type of structure/properties; this will be made more precise later on but, for now, let us mention the fact that the Leibniz-type identities  \eqref{1_Leibniz-total}, \eqref{eq:1_delta-deriv} and
\eqref{eq:1_d-deriv-cup}  
 %(\ref{eq:1_d-deriv}), 
 will be inherited. 
% 
% 
% 
%Therefore, next to the fact that $\delta$ and $D$ are commuting differentials, i.e. 
%\[ \delta^2= 0, \quad D^2= 0, \quad \delta\circ D= D\circ \delta,\]
%we will also focus on the derivation properties:

%\begin{equation}\label{eq:D-deriv} 
%D(\omega\wedge \omega')= D(\omega)\wedge \omega'+ (-1)^q \omega\wedge D(\omega') 
%\end{equation}
%whenever $\omega$ and $\omega$ are of bidegrees $(p, q)$, and $(p', q')$, respectively. 

Strictly speaking, the fact that the differentiable complex is a subcomplex is not immediate and requires a proof. Even more puzzling is to understand ``why" this happens and unravel the structure on $J^{\infty}\Gamma$ that governs this construction. More precisely, while the differentiable complex is defined using cochains on $J^{\infty}\Gamma$, it is natural
 to look at general Lie groupoids $\Sigma\tto \BB$ and:
 
 \medskip
 
%{\it {\bf Problem:} investigate the structure on $\Sigma\tto \BB$ that is needed in order to be able to form a bicomplex
%\begin{equation}\label{eq:general-bic}
%(C^{p}(\Sigma, \Lambda^qT^*\BB), \delta, d),
%\end{equation}
%with "basic properties" similar to those of $C^{\bullet, \bullet}_{\textrm{diff}}(\Gamma)$- which, in turn, should be recovered for $\Sigma= J^{\infty}\Gamma$. 
%}
%

{\it {\bf Problem:} investigate the structure on $\Sigma\tto \BB$ that is needed in order to be able to form a bicomplex
\begin{equation}\label{eq:general-bic}
(C^{p}(\Sigma, \Lambda^qT^*\BB), \delta, d),
\end{equation}
with ``basic properties" similar to those of the differentiable complex; and, of course, $C^{*,*}_{\rm diff}(\Gamma)$ should be obtained in the particular case when $\Sigma= J^{\infty}\Gamma$. Here, by basic properties we mean that (\ref{eq:general-bic}) is a double complex, the Leibniz identities \eqref{1_Leibniz-total}, \eqref{eq:1_delta-deriv} and \eqref{eq:1_d-deriv-cup} still hold and, in low degrees:
\begin{itemize}
\item for $q= 0$, $\delta$ is the usual groupoid differential on $C^{*}(\Sigma)$
\item for $p= 0$, $d$ is the DeRham differential on $\Omega^{*}(\BB)$. 
% on basic forms (coming from forms on the base), $d$ becomes the DeRham differential on $\BB$. 
\end{itemize}
}

It is remarkable that the outcome is precisely the structure that also shows up in the study of partial differential equations from a geometric point of view, i.e. the study of the {\bf Cartan distribution} on jet spaces and their submanifolds, see~\cite{KRASILSHCHIKVERBOVETSKY, GOLDSCHMIDT}. 

Below, we explain how one can slowly discover this structure when trying to construct the complex~\eqref{eq:general-bic}.

%%%%%%%%%%%%%%%%%%%%%%%%%%%%
%%%%%%%%%%%%%%%%%%%%%%%%%%%%
%%%%%%%%%%%%%%%%%%%%%%%%%%%%
%%%%%%%%%%%%%%%%%%%%%%%%%%%%
%%%%%%%%%%%%%%%%%%%%%%%%%%%%
\subsection{The  horizontal differential $\delta$ and groupoid actions}\label{subsection:action_TX}
%%%%%%%%%%%%%%%%%%%%%%%%%%%%
%%%%%%%%%%%%%%%%%%%%%%%%%%%%
%%%%%%%%%%%%%%%%%%%%%%%%%%%%
%%%%%%%%%%%%%%%%%%%%%%%%%%%%
%%%%%%%%%%%%%%%%%%%%%%%%%%%%

The fact that differentials $\delta$ can be seen as an algebraic way to encode actions is rather standard - e.g. one has the following general result. 

%\begin{lem}\label{lemma-general-delta} For any groupoid $\Sigma\tto \BB$ and any vector bundle $E\to \BB$ there is a 1-1 correspondence between actions of $\Sigma$ on $E$ (making $E$ into a representation) and operators 
%\[ \delta: C^{\bullet}(\Sigma, E)\to C^{\bullet+1}(\Sigma, E)\] 
%which make $C^{\bullet}(\Sigma, E)$ into a DG module over $(C^{\bullet}(\Sigma), \delta)$ i.e.:
%\begin{itemize}
%\item $\delta$ is a differential, i.e. $\delta^2= 0$
%\item $\delta$ is a derivation i.e. (\ref{eq:delta-deriv}) holds for $\omega\in C^p(\Sigma)$, $\omega'\in C^{p'}(\Sigma, E)$. 
%\end{itemize}
%and such that $\delta$ preserves the subcomplex of normal cochains (i.e. cochains that vanish whenever one of the entries is a unit).
%\end{lem}

\begin{lem}\label{lemma-general-delta} For any groupoid $\Sigma\tto \BB$ and any vector bundle $E\to \BB$ there is a 1-1 correspondence between actions of $\Sigma$ on $E$ (making $E$ into a representation) and differentials  
\[ \delta: C^{*}(\Sigma, E)\to C^{*+1}(\Sigma, E)\] 
which make $C^{*}(\Sigma, E)$ into a DG module over $(C^{*}(\Sigma), \delta)$ (i.e.\eqref{eq:1_delta-deriv} holds for $\omega\in C^p(\Sigma)$, $\omega'\in C^{p'}(\Sigma, E)$) and such that $\delta$ preserves the subcomplex of normalised cochains (i.e. cochains that vanish whenever one of the entries is a unit).
\end{lem}

This appears e.g. as Lemma 2.6 in \cite{CAMILOMARIUS}, but the proof is rather obvious. Explicitly, $\delta$ is given by the standard formula (\ref{delta-group-sheaf}), while the action can be recovered by what $\delta$ does on elements $\xi\in C^0(\Sigma, E)$, i.e. sections of $E$: for $g: x\to y$ arrow of $\Sigma$ and $v\in E_x$, one chooses $\xi$ with $\xi(x)= v$ and then $g\cdot v= \xi(y)- \delta(\xi)(g)$. 
\medskip

Returning to our main problem we see that, in order to have the differential $\delta$, we need actions of $\Sigma$ on $\Lambda^qT^*\BB$ (at least if we add the condition on normalised cochains); furthermore, taking advantage of the derivation identity (\ref{eq:delta-deriv}) for all $q$ and $q'$ (not only for $q= 0$), it is not difficult to see that the actions on $\Lambda^qT^*\BB$ must be the induced (diagonal) action on  $\Lambda^1T^*\BB= T^*\BB$.  Dualising, we need a right action of $\Sigma$ on $T\BB$ - so that any
arrow $g\in \Sigma$ acts as a linear map
\begin{equation}\label{eq:ind-right-action}
g^*: T_{t(g)}\BB \to T_{s(g)}\BB.
\end{equation}
Of course, using the fact that we deal with groupoids, i.e. using the presence of inverses, one can always turn such a right action into a left one by $g\cdot v:= (g^{-1})^*(v)$. The conclusion is that what we need is: {\it an action of $\Sigma$ on $T\BB$}.

%
%Related to the last step above note that the process of taking duals of representations, in its most natural form, takes a left action to a right action and the other way around; it just happens for groupoids that, due to the presence of inverses, one can turn right actions into left ones. 
%Therefore, while we need a left action of $\Sigma$ on $T^*\BB$, dually we should think we have a right action: any $g: x\to y$ arrow of $\Sigma$ acts as a map $g^*: T_y\BB\to T_x\BB$; the two are related by 
%\[ (g\cdot \theta)(v_y)= \theta_x(g^*(v))
%
%
%
%Therefore, while we need a left action of $\Sigma$ on $T^*\BB$, dually we should think we have a right action: any $g: x\to y$ arrow of $\Sigma$ acts as a map $g^*: T_y\BB\to T_x\BB$; the two are related by 
%\[ (g\cdot \theta)(v_y)= \theta_x(g^*(v))

Returning to the infinite jet groupoid, $T\BB$ carries an obvious action of $J^\infty\Gamma$ (in fact, of $J^k\Gamma$, for any $k\geq 1$): $j^\infty_x f\in J^{\infty}\Gamma$ acts on tangent vectors via $(df)_x$. In turn this induces actions on the exterior powers $\Lambda^qT^*\BB$ ($j^\infty_x f\in J^{\infty}\Gamma$ acts on a $q$-linear form at $x$ via the pullback by $f^{-1}$) and then it is not difficult to see that, indeed, $\delta$ of the differentiable complex becomes just the corresponding differential.
%\[ \delta: C^{p}(J^{\infty}\Gamma, \Lambda^{q}T^*\BB) \to C^{p+1}(J^{\infty}\Gamma, \Lambda^{q}T^*\BB)\]
%(given by (\ref{delta-group-sheaf})). 

%%%%%%%%%%%%%%%%%%%%%%%%%%%%
%%%%%%%%%%%%%%%%%%%%%%%%%%%%
%%%%%%%%%%%%%%%%%%%%%%%%%%%%
%%%%%%%%%%%%%%%%%%%%%%%%%%%%
%%%%%%%%%%%%%%%%%%%%%%%%%%%%
\subsection{The  vertical differential $d$ and connections}\label{The  vertical differential $d$ and connections}
%%%%%%%%%%%%%%%%%%%%%%%%%%%%
%%%%%%%%%%%%%%%%%%%%%%%%%%%%
%%%%%%%%%%%%%%%%%%%%%%%%%%%%
%%%%%%%%%%%%%%%%%%%%%%%%%%%%
%%%%%%%%%%%%%%%%%%%%%%%%%%%%

The differential $d$ is a bit more subtle as it reveals certain types of connections. Again, this is based on a rather standard lemma that holds in the very general context of surjective submersions $t: P\to \BB$ when one is looking for operators acting on horizontal forms 
% $\Omega^{\bullet}_{\textrm{hor}}(P)$, i.e. we will look for DeRham-like operators 
\begin{equation} \label{eq-D-conn-DR}
D: \Omega^{*}_{\textrm{hor}}(P)\to \Omega^{*+ 1}_{\textrm{hor}}(P).
\end{equation}
Recall here that a form $\omega\in \Omega^q(P)$ is said to be {\bf horizontal} if, 
$i_V(\omega)= 0$ whenever $V$ is a vector tangent to the $t$-fibers (i.e. $V\in \textrm{Ker}(dt)$).
To make the relationship with complexes of type 
(\ref{eq:general-bic}) more transparent, note that there is an 
%isomorphism/
identification 
 \[ \Omega^{*}_{\textrm{hor}}(P)\cong \Gamma(P, t^*\Lambda^{*} T^*\BB).\]
%Equivalently, pull-back by $t$ induces an inclusion 
%\[ \Gamma(P, t^*\Lambda^{\bullet} T^*\BB)\stackrel{t^*}{\longleftarrow} \Omega^{\bullet}(P)\]
%and $\Omega^{\bullet}_{\textrm{hor}}(P)$ is the image of this map. 

On the other hand, recall that:
% an {\bf Ehresmann connection} on $t: P\to \BB$ is a vector sub-bundle $H\subset TP$ that is complementary to the sub-bundle of vertical vectors $T^{\textrm{v}}P= \textrm{Ker}(dt)$. 

\begin{defn}\label{def:Ehr-conn}
An {\bf Ehresmann connection} on $t: P\to \BB$ is a vector sub-bundle $\CC\subset TP$ that is complementary to the sub-bundle of vertical vectors $T^{\textrm{v}}P= \textrm{Ker}(dt)$. 
\end{defn}
Such a $\CC$ gives rise to (and can be reinterpreted as) an operation of horizontal lifting of vector fields
\begin{equation}\label{eq:horizontal-lift}
 \textrm{hor}^{\CC}: \mathfrak{X}(\BB)\to \mathfrak{X}(P)
\end{equation}
which is actually defined pointwise: at each $x\in P$, $\textrm{hor}^{\CC}_x$ sends a vector $v\in T_{t(x)}\BB$ 
to the unique vector in $\CC_x$ that project via $t$ to $v$. 
%\[ \textrm{hor}^{H}: T_{t(x)}\BB\to T_{x}P\]
%sends a vector $v$ to the unique vector in $H_x$ that project via $t$ to $v$. 
%
Using the projection on $\CC$, $\textrm{pr}_{\CC}: TP\to \CC$, the DeRham operator on $\Omega^*(P)$ induces an operator $d_\CC$ on $\Omega^{*}_{\textrm{hor}}(P)$ by 
\[ d_{\CC}(\omega)(X^1, \ldots, X^q):= (d_{DR}\omega)(\textrm{pr}_{\CC}(X^1), \ldots, \textrm{pr}_{H}(X^q)).\]
or, interpreting horizontal forms as sections of $t^*\Lambda^{*} T^*\BB$, for $v^1, \ldots, v^q\in T_{t(x)}\BB$:
\[ d_{\CC}(\omega)(v^1, \ldots, v^q):= (d_{DR}\omega)(\textrm{hor}_{x}^{\CC}(v^1), \ldots, 
\textrm{hor}_{x}^{\CC}(v^p))
.\]
%,\]
% for all $v^1, \ldots, v^q\in T_{t(x)}\BB$. 

%
%
%Such an $H$ gives rise to an operation of horizontal lifting of vector fields
%\[ \textrm{hor}: \mathfrak{X}(\BB)\to \mathfrak{X}(P) \]
%which is actually defined pointqwise: at each $x\in P$, 
%\[ \textrm{hor}^{H}: T_{t(x)}\BB\to T_{x}P\]
%sends a vector $v$ to the unique vector in $H_x$ that project via $t$ to $v$. 

\begin{lem}\label{lemma-general-D}  For any bundle $t: P\to \BB$, the construction $\CC\mapsto d_\CC$ gives a 1-1 correspondence between Ehresmann connection and linear operators (\ref{eq-D-conn-DR}) 
% \[ D: \Omega^{\bullet}_{\textrm{hor}}(P)\to \Omega^{\bullet+ 1}_{\textrm{hor}}(P)\]
satisfying the following two properties:
\begin{itemize}
\item They satisfy the Leibniz derivation identity (\ref{eq:D-deriv}) 
%They satisfy the Leibniz derivation identity $d(\omega\wedge \omega')= d(\omega)\wedge \omega'+ (-1)^q \omega\wedge d(\omega')$ 
for all $\omega\in \Omega^q(\BB)$, $\omega\in \Omega^{q'}_{\rm hor}(P)$.
 % $D(fg)= fD(g)+ gD(f)$.
\item On basic forms, i.e. of type $t^*(\omega)$ with $\omega\in \Omega^{*}(\BB)$, $D(t^*(\omega))= t^*(d_{DR}\omega)$. 
% \[ D(t^*(\omega))= t^*(d_{DR}\omega).\]
\end{itemize}
Moreover, the following three conditions are equivalent:
\begin{itemize}
	\item[i)] $(\Omega^*_{\rm hor}(P),d_{\CC})$ is a cochain complex, i.e. $d_{\CC}^2= 0$;
	\item[ii)] $\CC$ is an involutive distribution, i.e. for all $X,Y$ vector fields tangent to $\CC$, the Lie bracket $[X,Y]$ is tangent to $\CC$;
	\item[iii)] $\CC$ is a flat connection, i.e. the induced operation 
	\[ \rm{hor}^\CC: \mathfrak{X}(\BB)\to \mathfrak{X}(P)\]
	of lifting vector fields from the base preserves the Lie brackets. 
\end{itemize}  
\end{lem}

This lemma is relevant for us when applied to the maps
\[ t: \Sigma^{(p)}\to \BB.\]
Note that having operators $d$ with the properties that we required in the Problem, hence also satisfying (\ref{eq:1_d-deriv-cup}),
%{eq:d-deriv}),
implies that the conditions of the lemma are satisfied. Indeed, for $p= 0$ and $\omega$ arbitrary (hence $\omega\in \Omega^q(\BB)$),  $q'= 0$ and $\omega'= 1$ (the constant $0$-form on $\Sigma^{(p)}$ with $p$-arbitrary), since
\[ \omega\cup 1= t^*(\omega)\]
identity (\ref{eq:1_d-deriv-cup}) becomes the second condition in the previous lemma. All together, one ends up with Ehresmann connections 
\[ \CC^{(p)}\subset T\Sigma^{(p)} .\]
Of course, they will not be independent, as one discovers using the entire derivation identities (see subsection $3.5$). Before we discuss that, we stare a bit at the case $p= 1$. While the way $\CC\subset T\Sigma$ is obtained appeals only to the submersion $t: \Sigma\to \BB$, one can slowly bring in more of the groupoid structure. First of all, using also the source map we obtain an ``action-like" operation induced by $\CC$: any arrow $g$ from $x$ to $y$ induces an operation 
\begin{equation}\label{eq:ind-right-action-2}
\lambda_{g}^{*}: T_y\BB\to T_x\BB, \quad \lambda_{g}^{*}(v)= (ds)_g(\textrm{hor}^{\CC}_{g}(v)).
\end{equation}
There is no surprise that this is basically the same action as in the previous subsection.

%%%%%%%%%%%%%%%%%%%%%%%%%%%%
%%%%%%%%%%%%%%%%%%%%%%%%%%%%
%%%%%%%%%%%%%%%%%%%%%%%%%%%%
%%%%%%%%%%%%%%%%%%%%%%%%%%%%
%%%%%%%%%%%%%%%%%%%%%%%%%%%%
\subsection{Compatibility of $\delta$ and $d$}
%%%%%%%%%%%%%%%%%%%%%%%%%%%%
%%%%%%%%%%%%%%%%%%%%%%%%%%%%
%%%%%%%%%%%%%%%%%%%%%%%%%%%%
%%%%%%%%%%%%%%%%%%%%%%%%%%%%
%%%%%%%%%%%%%%%%%%%%%%%%%%%%

The reason that the two actions coincide is precisely the compatibility of $\delta$ and $d_\CC$
\[ d_\CC\circ \delta= \delta\circ d_\CC,\]
applied in low degrees $p$ (i.e. for $p=0,1$):

\begin{lem}\label{lemma:commutativity-very-low} Assume we are given a Lie groupoid $\Sigma\tto \BB$ together with an Ehresmann connection $\CC$ along $t: \Sigma\to \BB$ and an action of $\Sigma$ on the vector bundle $T\BB$. Then the corresponding operators $\delta$ and $d_\CC$ are compatible, i.e. the following diagram commutes:
\[
\xymatrix{
\ldots \ar[r]^-{d_{DR}} & C^0(\Sigma, \Lambda^{q}T^*\BB) = \Omega^q(\BB) \ar[r]^-{d_{DR}}\ar[d]_-{\delta} & \Omega^{q+1}(\BB)= C^0(\Sigma, \Lambda^{q+1}T^*\BB)\ar[d]^-{\delta} \ar[r]^-{d_{DR}} & \ldots \\
\ldots \ar[r]^-{d_{\CC}} & C^1(\Sigma, \Lambda^{q}T^*\BB) = \Omega^q_{{\rm hor}}(\Sigma)\ar[r]^-{d_\CC} & \Omega^{q+1}_{{\rm hor}}(\Sigma)= 
C^1(\Sigma, \Lambda^{q+1}T^*\BB) \ar[r]^-{d_{\CC}} & \ldots 
}
\]
%\[
%\xymatrix{
%C^0(\Sigma, \Lambda^{p}T^*\BB) = \Omega^p(\BB) \ar[r]^-{d_{DR}}\ar[d]_-{\delta} & \Omega^{p+1}(\BB)= C^0(\Sigma, \Lambda^{p+1}T^*\BB)\ar[d]^-{\delta}\\
%C^1(\Sigma, \Lambda^{p}T^*\BB) = \Omega^p_{\textrm{hor}}(\Sigma)\ar[r]_-{d_{H}} & \Omega^{p+1}_{\textrm{hor}}(\Sigma)= 
%C^1(\Sigma, \Lambda^{p+1}T^*\BB) 
%}
%\]
if and only if action-like operation (\ref{eq:ind-right-action-2}) induced by $\CC$ coincides with the dual (\ref{eq:ind-right-action}) of the original action. 
\end{lem}

\begin{proof} Writing the compatibility condition for $q= 0$ and applying it to arbitrary functions 
$f\in C^{\infty}(\BB)$ one obtains
\[ d_\CC(s^*f- t^*f)= g\cdot s^*(d f)- t^* (df),\]
where the $t$-terms coincide by the properties of $d_\CC$. We are left with the $s$-terms which, applied to arbitrary $v\in T_{t(g)}\BB$ becomes
\[ (df)(ds)_g(\textrm{hor}^{\CC}_{g}(v))= (df)(ds)(g^{-1}\cdot v).\]

Conversely, if the two actions coincide it is clear that the diagram is commutative when applied on functions (i.e. at $q=0$ ). Since $d_\CC\circ \delta- \delta\circ d$ commutes with $d_\CC$, the same happens when applied on exact $1$-forms $df$. Using the derivation identities it follows that it happens on all $q$-forms on $\BB$. 
\end{proof}

%%%%%%%%%%%%%%%%%%%%%%%%%%%%
%%%%%%%%%%%%%%%%%%%%%%%%%%%%
%%%%%%%%%%%%%%%%%%%%%%%%%%%%
%%%%%%%%%%%%%%%%%%%%%%%%%%%%
%%%%%%%%%%%%%%%%%%%%%%%%%%%%
% \subsection{Just one connection! A multiplicative one!} 
% \subsection{All that matters is $H$; and its properties} 
\subsection{Only one (multiplicative) connection $\CC$}
\label{subsection: Only one (multiplicative) connection $C$} 
%%%%%%%%%%%%%%%%%%%%%%%%%%%%
%%%%%%%%%%%%%%%%%%%%%%%%%%%%
%%%%%%%%%%%%%%%%%%%%%%%%%%%%
%%%%%%%%%%%%%%%%%%%%%%%%%%%%
%%%%%%%%%%%%%%%%%%%%%%%%%%%%

The conclusion of subsection \ref{The  vertical differential $d$ and connections} was that we need Ehresmann connections $\CC^{(p)}\subset T\Sigma^{(p)}$, each one of them giving rise to $d= d_{\CC^{(p)}}$ acting on $C^{p}(\Sigma, \Lambda^{*}T^*\BB)$. However, these will be related if we want the derivation identity  (\ref{eq:1_d-deriv-cup}) to hold.

\begin{lem}\label{lem:one_connection} Assume that $\CC^{(p)}$ are Ehresmann connections on $t: \Sigma^{(p)}\to \BB$ so that the resulting operators  $d= d_{\CC^{(p)}}$ together satisfy the derivation identity  (\ref{eq:1_d-deriv-cup}). Then each $\CC^{(p)}$ is determined by $\CC:= \CC^{(1)}$ as
\begin{equation}\label{descr-H-p} 
\CC^{(p)}= \{(v^1, \ldots, v^p)\in T\Sigma^{(p)}: v^1, \ldots, v^p\in \CC\}.
\end{equation}
Equivalently, the induced horizontal lifting map is given by 
\[{\rm hor}^{\CC^{(p)}}_{g_1, \ldots, g_p}(v)= \left( {\rm hor}^{\CC}_{g_1}(v),  {\rm hor}^{\CC}_{g_2}(g_{1}^{*}v), \ldots, {\rm hor}^{\CC}_{g_p}(g_{p-1}^{*} \ldots g_{1}^{*}v)\right) .\]
\end{lem}

\begin{proof} We concentrate on the last formula (which implies the other part by dimension counting). Also, for notational simplicity, we assume $p= 2$. It suffices to use (\ref{eq:1_d-deriv-cup}) in a very extreme case: for $0$-forms, i.e. functions, $f_1, f_2\in C^{\infty}(\Sigma)= C^1(\Sigma, \Lambda^0T^*\BB)$:
\[ d_{\CC^{(2)}}(f_1\cup f_2)= d_{\CC}(f_1)\cup f_2+ f_1\cup d_{\CC}(f_2) \]
which, in turn, can be rewritten as
\[ d_{\CC^{(2)}}(\textrm{pr}^{*}_{1}f_1\cup \textrm{pr}^{*}_{2}f_2)= \textrm{pr}^{*}_{1}d_{\CC}(f_1)\cup\textrm{pr}^{*}_{2}f_2+ \textrm{pr}^{*}_{1}f_1\cup (1\cup d_{\CC}(f_2)) .\]
Taking $f_2= 1$ and $f_1= f\in C^{\infty}(\Sigma)$ arbitrary, and then $f_1= 1$ and $f_2= f$, we obtain:
\begin{equation}\label{eq:int:proof-pr} 
d_{\CC^{(2)}}(\textrm{pr}^{*}_{1}f)= \textrm{pr}^{*}_{1}d_{\CC}(f), \quad d_{\CC^{(2)}}(\textrm{pr}^{*}_{2}f)= 1\cup d_{\CC}(f).
\end{equation}
Writing out $d_\CC$ in terms of the horizontal lifts, the first equation translates into the fact that for $(g_1, g_2)\in \Sigma^{(2)}$, $v\in T_{t(g_1)}\BB$, the first component of $\textrm{hor}^{\CC^{(2)}}_{g_1, g_2}(v)$ is $\textrm{hor}^{\CC}_{g_1}(v)$. Similarly, the second equation implies that the second component of $\textrm{hor}^{\CC^{(2)}}_{g_1, g_2}(v)$ is $\textrm{hor}^{\CC}_{g_2}(g_1^*v)$. 
\end{proof}

Yet another property that we addressed only degrees $0$ and $1$, namely the compatibility of $\delta$ and $d$ from Lemma \ref{lemma:commutativity-very-low}, has more implications on $\CC$ if one moves one line higher. To state it
we make use of the fact that for any groupoid $\Sigma \tto \BB$, by taking the differentials of the structure maps of $\Sigma$ (source, target, multiplication), one obtains a new groupoid 
\[ T\Sigma \tto T\BB,\]
called the tangent groupoid of $\Sigma$. With this, 
% $H\subset T\Sigma$ is said to be {\bf multiplicative} if it is a sub-groupoid of $T\Sigma$. 

\begin{defn} \label{defin:multiplicativity}
A sub-bundle $\CC\subset T\Sigma$ (as vector bundles over $\Sigma$) is said to be {\bf multiplicative} if it is a full sub-groupoid of $T\Sigma$. 
\end{defn}
This boils down to the following explicit conditions: 
\begin{enumerate}
\item[m1:] $ds, dt: \CC\to T\BB$ are surjective. 
\item[m2:] for $v, w\in \CC$ composable (i.e. with $ds(v)= dt(w)$), $v\circ w:= dm(v, w)\in \CC$.
\item[m3:] at any $x\in \BB$, $\CC_{1_x}$ contains $T_x\BB$ (interpreted as a subspace of $T_{1_x}\Sigma$ 
% $T_x\BB$ interpreted as a subspace of $T_{1_x}\Sigma$ (via the unit map) is contained 
via the unit map).
\item[m4:] the differential of the inversion map $\tau: \Sigma\to \Sigma$ takes $\CC$ to itself. 
\end{enumerate}

While these conditions make sense for more general sub-bundles $\CC$, for Ehresmann connections they can be expressed in terms of the operation of horizontal lifting 
%$\textrm{hor}= \textrm{hor}^\CC$ 
as:
%\begin{equation}\label{eq:m2}
%dm\left( \textrm{hor}_{g_1}(v), \textrm{hor}_{g_2}(g_{1}^{*}v)\right)= \textrm{hor}_{g_1g_2}(v),
%\quad \textrm{for all $g_1, g_2\in\Sigma$ composable, $v\in T_{t(g_1)}$,}
%\end{equation}
%% for all $g_1, g_2\in\Sigma$ composable, $v\in T_{t(g_1)}$, 
%\begin{equation}\label{eq:m3}
%\textrm{hor}_{1_x}(v)= v,
%\quad \textrm{for all $v\in T_x\BB\subset T_{1_x}\Sigma$, } 
%\end{equation}
%% for all $v\in T_x\BB\subset T_{1_x}\Sigma$, 
%\begin{equation}\label{eq:m4}
%\textrm{hor}_{g^{-1}}(\lambda_{g}^{*}v)= (d\tau) \textrm{hor}_g(v)),
%\quad \textrm{for all $g\in \Sigma$, $v\in T_{t(g)}\BB$,}
%\end{equation}
%% for all $g\in \Sigma$, $v\in T_{t(g)}\BB$, 
%where $\lambda_{g}^{*}$ denotes the induced quasi-action (\ref{eq:ind-right-action-2}). Moreover, condition m1 can be rephrased into saying that all the $\lambda_{g}^{*}$ are isomorphisms; the rest of the conditions imply, of course, that the quasi-action is actually an action. 
%% to higher degrees:
%
%
%
\begin{eqnarray}
 &  & dm\left( \textrm{hor}^\CC_{g_1}(v), \textrm{hor}^\CC_{g_2}(g_{1}^{*}v)\right)= \textrm{hor}^\CC_{g_1g_2}(v),
\quad \textrm{for all $g_1, g_2\in\Sigma$ composable, $v\in T_{t(g_1)}\BB$,} \label{eq:m2}\\
% \end{equation}
%% for all $g_1, g_2\in\Sigma$ composable, $v\in T_{t(g_1)}$, 
% \begin{equation}\label{eq:m3}
 &  &  \textrm{hor}^\CC_{1_x}(v)= v,
\quad \textrm{for all $v\in T_x\BB\subset T_{1_x}\Sigma$, } \label{eq:m3}\\
% \end{equation}
%% for all $v\in T_x\BB\subset T_{1_x}\Sigma$, 
% \begin{equation}\label{eq:m4}
 &  & \textrm{hor}^\CC_{g^{-1}}(\lambda_{g}^{*}v)= (d\tau) \left(\textrm{hor}^\CC_g(v)\right),
\quad \textrm{for all $g\in \Sigma$, $v\in T_{t(g)}\BB$,}\label{eq:m4}
% \end{equation}
%% for all $g\in \Sigma$, $v\in T_{t(g)}\BB$, 
\end{eqnarray}
where $\lambda_{g}^{*}$ denotes the induced quasi-action (\ref{eq:ind-right-action-2}). See also the proof below. Moreover, condition m1 can be rephrased into the condition that all the $\lambda_{g}^{*}$ are isomorphisms; the rest of the conditions imply, of course, that the quasi-action is actually an action. 
% to higher degrees:

\begin{lem} 
With the assumptions from Lemma~\ref{lem:one_connection} and Lemma~\ref{lemma:commutativity-very-low}, the 
following diagram is commutative
\[
\xymatrix{
\ldots \ar[r]^-{d_{\CC}} & C^1(\Sigma, \Lambda^{q}T^*\BB)= \Omega^q_{{\rm hor}}(\Sigma)  \ar[r]^-{d_{\CC}}\ar[d]_-{\delta} &  C^1(\Sigma, \Lambda^{q+1}T^*\BB)=\Omega^{q+1}_{{\rm hor}}(\Sigma)\ar[d]^-{\delta} \ar[r]^-{d_{\CC}} & \ldots \\
\ldots \ar[r]^-{d_{\CC^{(2)}}} & C^2(\Sigma, \Lambda^{q}T^*\BB) = \Omega^q_{{\rm hor}}(\Sigma^{(2)})\ar[r]^-{d_{\CC^{(2)}}} & 
C^2(\Sigma, \Lambda^{q+1}T^*\BB) =\Omega^{q+1}_{{\rm hor}}(\Sigma^{(2)})\ar[r]^-{d_{\CC^{(2)}}}  & \ldots 
}
\]
if and only if $\CC$ is multiplicative (actually, if {\rm m2} holds). 
% closed under multiplication, i.e. for any $v, w\in H$ with the property that $(v, w)$ is tangent to $\Sigma^{(2)}$, $dm(v, w)$ is again in $H$. 
\end{lem}

\begin{proof} For simplicity, let us work just with functions $f\in C^{\infty}(\Sigma^{(2)})$ and write down the compatibility 
\[ d_{\CC^{(2)}}(\textrm{pr}^{*}_{1}f+ \textrm{pr}^{*}_{2}f- m^*f)= \delta d_{\CC}(f).\]
Using the two equations from (\ref{eq:int:proof-pr}), we are left with 
\[ d_{\CC^{(2)}}(m^*f)= m^*d_{\CC}(f).\]
Writing this out in terms of horizontal lifts one obtains
\[ dm(\textrm{hor}^{\CC^{(2)}}_{g_1, g_2}(v))= \textrm{hor}^{\CC}_{g_1g_2}(v), \quad \textrm{for}\ (g_1, g_2)\in \Sigma^{(2)}, v\in T_{t(g_1)}\BB .\]
Since the image of $\textrm{hor}^{\CC^{(2)}}$ is $\CC^{(2)}= T\Sigma^{(2)}\cap \CC\times \CC$, we see that the condition m2 for multiplicativity implies the commutativity of the diagram and viceversa if the diagram is commutative m2 follows. Note that, with the formula for $\textrm{hor}^{\CC^{(2)}}$ from the previous lemma, we obtain precisely (\ref{eq:m2}). It remains to show that m1, m3 and m4 can be derived from m2 and the assumptions.
Condition m1 follows from the fact that the quasi-action is an action, by assumption (Lemma \ref{lemma:commutativity-very-low}). For m3 one makes use of (\ref{eq:m2}) written at units to derive (\ref{eq:m3}). Alternatively, m3 can also be derived from the rest because, in the tangent groupoid, we can write units as $V\circ V^{-1}$ with $V$ chosen to be in $\CC$. For m4 we use again (\ref{eq:m2}), this time for $g_1= g$ arbitrary and $g_2= g^{-1}$. The outcome can be written in the tangent groupoid as:
\[ 1_v= \textrm{hor}^\CC_g(v)\circ \textrm{hor}^\CC_{g^{-1}}(g^*v)\]
from which we deduce that $\textrm{hor}_{g^{-1}}(g^*v)$ is the inverse of $\textrm{hor}_g(v)$ - i.e. precisely formula (\ref{eq:m4}). 
\end{proof}

\begin{rmk}
As one sees from the discussions above, a multiplicative Ehresmann connection $\mathcal{C}$ on $t:\Sigma\to \BB$ is also an Ehresmann connection on $s:\Sigma\to \BB$. Consequently, an analogue of the lifting operation~\ref{eq:horizontal-lift} is defined with respect to $s$. In the rest of the paper, we will use the same notation for these two lifting operations; which one is being considered will always be clear from the context.
%, for any $g\in \Sigma$ and $v\in T_{s(g)}(X)$, 
%we will keep the notation $hor^{\mathcal{C}}_g(v)$ for the only vector in $\mathcal{C}_g\subset T_g\Sigma$ which is $s$-projectable to $v$ and ${\rm hor}^{\mathcal{C}}:\mathfrak{X}(X)\to \mathfrak{X}(\Sigma)$ will be the corresponding lifting operation. 
\end{rmk}

The final outcome of our discussion is that, for the (analogue of the) ``Haefliger's bicomplex''~\eqref{eq:diff-cplx-3} to exist on $\Sigma\tto \BB$, one needs a multiplicative flat connecton $\CC$ on $\Sigma$. In the rest of this paper, starting from the next section, we study pairs $(\Sigma, \CC)$ of this form. In particular, we will gain a more conceptual understanding of Haefliger's work and a generalization of it.

Before going on with this program, in the next subsection we describe the flat multiplicative connection on $J^\infty\Gamma$ underlying bicomplex~\eqref{eq:diff-cplx-3}; then, we conclude this section spending some more time with the Bott-Shulman complex.

\subsection{The multiplicative flat connection on $J^\infty\Gamma$}\label{The Cartan connection}

To describe the flat multiplicative connection on $J^\infty\Gamma$ that underlies Haefliger's differentiable cohomology of a Lie pseudogroup, it is useful to go back to the tower~\eqref{eq:the-jet-tower}
\[
J^\infty\Gamma\to \dots \to J^l\Gamma\to J^{l-1}\Gamma\to \dots \to J^1\Gamma\to J^0\Gamma\tto \BB.
\]
This is a tower of groupoids over $\BB$ and actually, since we work with Lie pseudogroups $\Gamma$, a tower of \emph{Lie} groupoids: the $J^l\Gamma$'s are Lie groupoids and each map is a surjective submersion and a morphism of groupoids. The definition of Lie pseudogroup can be summed up and is naturally expressed by saying that this tower is a normal pf-atlas for $J^\infty\Gamma$; the reader is referred to the Appendix. The flat multiplicative connection on $J^\infty\Gamma$, that we will denote by $\CC^\infty$, is a profinite dimensional distribution, i.e. a ``limit'' of distributions on the groupoids $J^l\Gamma$ which are compatible with the tower projections. 

To begin with, it is useful to observe that $J^l\Gamma$ is a subspace of the jet space $J^l(\BB, \BB)$. A section $\sigma: \BB\to J^l(\BB, \BB)$ of the projection 
\[
s: J^l(\BB, \BB)\to \BB, \quad j^l_x f\to x
\]
 is called {\bf holonomic} when $\sigma=j^l f: x\to j^l_x f$ for some $f: \BB\to \BB$. It is possible to define a distribution $\CC^l$ on $J^l(\BB, \BB)$ via the request that holonomic sections are precisely its integral sections.
 %such that the image of a section of $J^l(\BB, \BB)\to \BB$ defines an integral manifold of $\CC^l$ if and only if the section is holonomic.
\begin{defn}
 The {\bf Cartan distribution} $\CC^l$ on $J^l(\BB, \BB)$ is the distribution on $J^l(\BB, \BB)$ such that a section $\sigma: \BB\to J^l(\BB, \BB)$ is holonomic if and only if its image is an integral manifold of $\CC^l$.
\end{defn}
The Cartan distribution on jet spaces is a very well known object. In particular, it plays a central role in the geometric theory of PDEs; among the extensive literature, see for example~\cite{KRASILSHCHIKVERBOVETSKY, GOLDSCHMIDT}. A more explicit definition can be given by adopting the dual point of view. One defines a one form $\omega^l$ on $J^l(\BB, \BB)$ valued in the vertical bundle $\ker (dpr)$ of the projection $pr: J^l(\BB, \BB)\to J^{l-1}(\BB, \BB)$ by 
\[
\omega^l_{j^{l}_x f}=d_{j^{l}_x f}pr-d_{j^{l}_x f} (j^{l-1}f \circ s)
\]
and then $\CC^l=\ker(\omega^l)$ can either be taken as an alternative definition (see~\cite{MARIA,ORI,FRANCESCO}) or proven as a lemma. What matters for us is the following
\begin{prp}[\hspace{1sp}\cite{MARIA, ORI}]
For any Lie pseudogroup $\Gamma$ and for any $l\geq 1$, the Cartan distribution $\CC^l$ restricts to $J^l\Gamma$ as a regular multiplicative distribution (still denoted by $\CC^l$).
\end{prp}
The multiplicativity of the Cartan distribution on the jet groupoids $J^l\Gamma$ of a Lie pseudogroup is a key property when studying the geometry of PDE. See~\cite{MARIA} for a systematic study of ``Cartan-like'' multiplicative forms on groupoids and~\cite{ORI} for a modern take on Cartan's seminal work on pseudogroups, where the role of multiplicativity is made explicit.

The above discussion stays true when the positive natural number $l$ is replaced by $\infty$. That is, $J^\infty\Gamma$ is equipped with a multiplicative distribution $\CC^\infty$ that comes as a restriction of an analogous distribution on $J^\infty(\BB, \BB)$ detecting holonomic sections. $\CC^\infty$ a smooth object in the sense of pf-manifolds, i.e. the ``limit'' of the distributions $\CC^l$. Indeed, to make these claims, as well as the following one, completely rigorous, one needs to work with profinite dimensional manifolds; see the Appendix and in particular Examples~\ref{exam-Jinfty-R},~\ref{exm:Car-distr-J},~\ref{ex:Cartan-form-R},~\ref{exam-pair-groupoid} and~\ref{exam-Lie-pseudogroups}. 
What one is able to see is that:
\begin{thm}\label{thm:Classical_cartan_dist}
%The sequence $\{\CC^l\}_{l\in \mathbb{N}}$ of Cartan distributions defines a distiribution 
The distribution $\CC^\infty$ on $J^\infty\Gamma$ is
\begin{itemize}
\item multiplicative;
\item a connection with respect to $s:J^\infty\Gamma\to \BB$;
\item involutive (i.e. a flat connection, see the second part of Lemma~\ref{lemma-general-D}).
\end{itemize}
\end{thm}
\begin{proof}
The proof is well known; we sketch it for the sake of completeness.

First of all, as explained in the Appendix, the pf-tangent space $TJ^\infty\Gamma$ can be identified with the limit $\plim TJ^l\Gamma$; as a consequence, the Cartan distribution can be identified with the sequence $\{\CC^l\}$ of the finite dimensional Cartan distributions, which is compatible with~\eqref{eq:the-jet-tower} in the sense that $dpr (\CC^l)\subset \CC^{l-1}$. With this in mind, $\CC^\infty$ straight away inherits its multiplicativity from the multiplicativity of the $\CC^l$'s, which can be checked directly (for example using the forms $\omega^l$'s introduced above~\cite{ORI}).

The fact that $\CC^\infty$ is an involutive distribution complementary to $s$, i.e. a flat connection (see the second part of Lemma~\ref{lemma-general-D}), is hidden in the interplay between the projections $pr: J^l\Gamma\to J^{l-1}\Gamma$ and the Cartan distributions. Using the terminology of Definition $6.2.4$ in~\cite{MARIA}, the projections $pr: J^l\Gamma\to J^{l-1}\Gamma$ are {\bf Lie prolongations}. In particular,
\begin{itemize}
\item $dpr(\CC^l\cap \ker(ds))=0$ for all $l\geq 1$;
\item $[dpr(\CC^{l+1}), dpr(\CC^{l+1})]\subset \CC^{l}$, for all $l\geq 1$.
\end{itemize}
The first identity implies that $\CC^\infty$ is a connection while the second one forces its involutivity.
\end{proof}
\begin{defn}
$\CC^\infty$ is called the {\bf Cartan connection} on $J^\infty\Gamma$. 
\end{defn}
%In fact, multiplicativity is straight away inherited from multiplicativity of the $\CC^l$, while the fact that $\CC^\infty$ is an involutive distribution is hidden in the interplay between the projections $pr: J^l\Gamma\to J^{l-1}\Gamma$ and the Cartan distributions (see~\cite{MARIA}, and in particular Definition $6.2.4$).
 Surprisingly enough, the multiplicativity of $\CC^\infty$ was the last property to appear explicitly in the literature; in fact, the multiplicativity of the finite dimensional Cartan distributions $\CC^l$ is explicitely addressed for the first time in~\cite{MARIA}. On the other hand, the transversality and involutivity properties of $\CC^\infty$ have been known for a long time.
%This is the object underlying Haefliger's differentiable cohomology of Lie pseudogroups in the way described in this section.

\subsection{A look back at the structure on the Bott-Shulman and related complexes}
\label{More on the structure(s) of the Bott-Shulman complex}
 It is interesting to spell out the various 
structures that are floating around on the various complexes. We concentrate first on the Bott-Shulman complex; so, let us fix an arbitrary Lie groupoid $\G\tto \BB$ and consider 
\[ \Omega^{p, q}= \Omega^q(\G^{(p)}).\]

First of all together with the wedge product of forms, the DeRham differential and the pull-backs along the face and degeneracy maps, $(\Omega^{*,*}, \wedge, d, d_{i}^*, s_{i}^*)$ is a {\bf cosimplicial DGA},
i.e. 
\begin{itemize}
\item for each $p$, $(\Omega^{p, *}, \wedge, d)$ is a DGA; hence 
\[ -\wedge -: \Omega^{p, q}\times \Omega^{p, q'}\to \Omega^{p, q+ q'}\]
and $d: \Omega^{p, q}\to \Omega^{p, q+1}$ satisfies the Leibniz identity:
\begin{equation}\label{eq:d-deriv} 
d(\omega\wedge \omega')= d(\omega)\wedge \omega'+ (-1)^q \omega\wedge d(\omega') 
\end{equation}
whenever $\omega$ and $\omega$ are of bidegrees $(p, q)$, and $(p, q')$, respectively. 
\item these DGAs are related by morphisms (of DGAs) 
\[ d_i^*: \Omega^{p-1, *}\to  \Omega^{p, *}, \quad s_i^*: \Omega^{p+1, *}\to \Omega^{p, *}\quad (\textrm{for $0\leq i\leq p$})\]
satisfying the cosimplicial identities (i.e. the dual of the simplicial identities (\ref{simplicial-identities})). 
\end{itemize}
In turn, for any such cosimplicial DGA, one can construct the second differential (along $p$), $\delta= \sum (-1)^i d_{i}^{*}$ and, furthermore, the product structure \eqref{eq:cup-BS-complex}; indeed, note that 
% that formula uses the wedge product of forms and the two maps there can be written as
the maps $\textrm{first}_{p}^{*}$ and $\textrm{last}_{p'}^{*}$ can be expressed in terms of the simplicial structure as
\[ \textrm{first}_{p}^{*}= \underbrace{d_{0}^*\ldots d_{0}^*}_{\textrm{$p'$ times}}, \quad
\textrm{last}_{p'}^{*}=  d_{p'+ p}^*d_{p'+ p-1}^*\ldots d_{p'+ 1}^*,\]
hence \eqref{eq:cup-BS-complex} makes sense for any cosimplicial DGA. 
The outcome is a {\bf double-DGA} $(\Omega^{p, q}, \cup, \delta, d)$ in the sense that it comes with:
\begin{itemize}
\item a bigraded algebra operation 
\[ - \cup -: \Omega^{p, q}\times \Omega^{p', q'}\to \Omega^{p+p', q+q'};\]
\item a differential $\delta: \Omega^{p, q}\to \Omega^{p+1, q}$ satisfying the horizontal Leibniz identity:
\begin{equation}\label{eq:delta-deriv} 
\delta(\omega\cup \omega')= \delta(\omega)\cup \omega'+ (-1)^p \omega\cup \delta(\omega') 
\end{equation}
whenever $\omega$ and $\omega$ are of bidegrees $(p, q)$, and $(p', q')$, respectively;
\item a differential $d: \Omega^{p, q}\to \Omega^{p, q+1}$ satisfying the vertical Leibniz identity:
\begin{equation}\label{eq:D-deriv} 
d(\omega\wedge \omega')= d(\omega)\wedge \omega'+ (-1)^q \omega\wedge d(\omega') 
\end{equation}
whenever $\omega$ and $\omega$ are of bidegrees $(p, q)$, and $(p', q')$, respectively. 
\end{itemize}

Finally, for any double DGA $(\Omega^{p, q}, \cup, \delta, d)$ one can also pass to the total complex as in (\ref{Tot-BS}):
\[ \textrm{Tot}^k \Omega^{*,*}:= \bigoplus_{p+ q= k} \Omega^{p, q}, \quad D_{\textrm{tot}}= \delta+ (-1)^p d \ \textrm{on}\ \Omega^{p, q}\]
endowed with the signed product 
\[ \omega\cup_{\textrm{tot}}\omega':= (-1)^{qp'} \omega\cup \omega',\quad \textrm{for $\omega\in \Omega^{p, q}, \omega'\in \Omega^{p', q'}$.}\]
Of course, the signs are chosen so that $D_{\textrm{tot}}$ satisfies the Leibniz identity w.r.t. the total degree:
\begin{equation}\label{Leibniz-total} 
D_{\textrm{tot}}(\omega\cup_{\textrm{tot}}\omega')= D_{\textrm{tot}}(\omega)\cup_{\textrm{tot}}\omega'+ 
(-1)^{k} \omega\cup_{\textrm{tot}} D_{\textrm{tot}}(\omega')
\end{equation}
where $k= p+q$ is the total degree of $\omega\in \Omega^q(\G^{(p)})$. Therefore one obtains a DGA:
\[(\textrm{Tot}\, \Omega^{*,*}, D_{\textrm{tot}}, \cup_{\textrm{tot}}).\] % $ becomes a DGA.

%
%$(\textrm{Tot}^k \Omega^{\bullet, \bullet}, D_{\textrm{tot}}, \cup_{\textrm{tot}})$
%the resulting 
%total complex becomes a differential graded algebra w.r.t. the total degree:
%\begin{equation}\label{Leibniz-total} 
%D_{\textrm{tot}}(\omega\cup_{\textrm{tot}}\omega')= D_{\textrm{tot}}(\omega)\cup_{\textrm{tot}}\omega'+ 
%(-1)^{k} \omega\cup_{\textrm{tot}} D_{\textrm{tot}}(\omega')
%\end{equation}
%where $k= p+q$ is the total degree of $\omega\in \Omega^q(\G^{(p)})$. 
%
%
%With this, $(\textrm{Tot}^k \Omega^{\bullet, \bullet}, D_{\textrm{tot}}, \cup_{\textrm{tot}})$ becomes a DGA. 
%
%

Of course, with the appropriate (and rather obvious) notion of morphisms, the previous constructions can be seen as functors
\[ \textrm{cosimplicial DGAs}\rightarrow \textrm{double DGAs}\rightarrow \textrm{DGAs}.\]
The conclusion is that the Bott-Shulman complex is a cosimplicial DGA and, therefore, can also be turned into a double DGA or, by passing to the total complex, into an ordinary DGA. The fact that the differentiable complex is a sub-complex of the Bott-Shulman one can be refined into the folowing result (which is actually less messy to prove):

\begin{prp} For any Lie pseudogroup $\Gamma$ on $\BB$, the differentiable complex of $\Gamma$ is a sub-cosimplicial DGA
of the Bott-Shulman complex of the corresponding groupoid $\G= \Ger(\Gamma)\tto \BB$ (via the inclusion 
$j^*: C^{p, q}_{\rm diff}(\Gamma)\hookrightarrow \Omega^q(\G^{(p)})$). In particular, the differentiable complex of $\Gamma$ carries an induced cosimplicial DGA structure.
\end{prp}

%%%%%%%%%%%%%%%%%%%%%%%%%%%%
%%%%%%%%%%%%%%%%%%%%%%%%%%%%
%%%%%%%%%%%%%%%%%%%%%%%%%%%%
%%%%%%%%%%%%%%%%%%%%%%%%%%%%
%%%%%%%%%%%%%%%%%%%%%%%%%%%%
%%%%%%%%%%%%%%%%%%%%%%%%%%%%
%%%%%%%%%%%%%%%%%%%%%%%%%%%%
%%%%%%%%%%%%%%%%%%%%%%%%%%%%
%%%%%%%%%%%%%%%%%%%%%%%%%%%%
%%%%%%%%%%%%%%%%%%%%%%%%%%%%
\section{Generalization: Haefliger cohomology}\label{Generalization: Haefliger cohomology}%%%%%%%%%%%%%%%%%%%%%%%%%%%%
%%%%%%%%%%%%%%%%%%%%%%%%%%%%
%%%%%%%%%%%%%%%%%%%%%%%%%%%%
%%%%%%%%%%%%%%%%%%%%%%%%%%%%
%%%%%%%%%%%%%%%%%%%%%%%%%%%%
%%%%%%%%%%%%%%%%%%%%%%%%%%%%
%%%%%%%%%%%%%%%%%%%%%%%%%%%%
%%%%%%%%%%%%%%%%%%%%%%%%%%%%
%%%%%%%%%%%%%%%%%%%%%%%%%%%%
%%%%%%%%%%%%%%%%%%%%%%%%%%%%

To sum up the discussion from the previous section, we found out that to define the analogue of the ``Haefliger's bicomplex''~\eqref{eq:diff-cplx-3} for a (profinite dimensional) Lie groupoid $\Sigma\tto\BB$ we need a multiplicative flat connection $\CC$ on $\Sigma$. The aim of the rest of this paper is to recast Haefliger's construction in this generality, while providing a conceptual understanding of the various steps involved. Indeed, in Section $4$ we will define the {\bf Haefliger cohomology} of a pair $(\Sigma, \CC)$ given by a Lie groupoid together with a flat multiplicative connection; in our terminology, a {\bf flat (Cartan) groupoid}. In Section $5$ we will discuss the corresponding infinitesimal picture. In Section $6$ we will construct a ``van Est map" generalizing the isomorphism~\eqref{eq: Van-Est-Haefliger}. There, we will also discuss a generalization of the characteristic map~\eqref{eq:dif-charact-map} that, remarkably, is not defined as the composition with an analogue of $j^*:H^*_{\rm{diff}}(\Gamma)\to H^*_{dR}(\mathcal{G})$, but coincides with it whenever such an analogue is available.

A terminology remark: from now on, we will simply write ``Lie groupoids" or ``Lie algebroids" (the latter starting from the next section) meaning either the usual, finite dimensional, objects or the (strict) profinite dimensional ones, discussed in the Appendix (see Definition~\ref{defn:strict-things}). We stress that our profinite dimensional groupoids (algebroids) will always be defined over a finite dimensional base $\BB$, except when discussing the general setting for our van Est map, in section~\ref{Van Est maps} and, in particular, in subsection~\ref{A very general setting}. 
\footnote{More precisely: in subsection~\ref{A very general setting} we need to consider the action groupoid $\Sigma\ltimes P=\Sigma\tensor[_s]{\times}{_\mu} P$ over $P$ associated to an action of a (pf) Lie groupoid $\Sigma\rightrightarrows \BB$ on a (pf) manifold $P$ along a map $\mu:P\to \BB$.}
%that is, they admit a pf-atlas (see the Appendix) such that the induced pf-atlas on the base consists of the constant tower. 

%%%%%%%%%%%%%%%%%%%%%%%%%%%%
%%%%%%%%%%%%%%%%%%%%%%%%%%%%
%%%%%%%%%%%%%%%%%%%%%%%%%%%%
%%%%%%%%%%%%%%%%%%%%%%%%%%%%
%%%%%%%%%%%%%%%%%%%%%%%%%%%%
\subsection{Definitions and first examples}
%%%%%%%%%%%%%%%%%%%%%%%%%%%%
%%%%%%%%%%%%%%%%%%%%%%%%%%%%
%%%%%%%%%%%%%%%%%%%%%%%%%%%%
%%%%%%%%%%%%%%%%%%%%%%%%%%%%
%%%%%%%%%%%%%%%%%%%%%%%%%%%%

The problem that we posed and the discussion around it slowly revealed the structure that is needed to define Haefliger cohomology - summarized in the definition below. 
%It is remarkable that it is precisely the structure the emerges also from the study of pseudogroups from the perspective of PDEs. % (and the fact that it also suffices, if not clear already, will be further clarified in the next subsection).

%\begin{defn}
%A {\bf flat groupoid} is any Lie groupoid $\Sigma\tto \BB$ together with a sub-bundle $H\subset T\Sigma$ satisfying:
%\begin{itemize}
%\item $H$ is a flat Ehresman connection on $t: \Sigma\to \BB$,
%\item multiplicativity:  $H$ is subgroupoid of $T\Sigma \tto T\BB$.
%\end{itemize}
%\end{defn}
%
%
%\begin{defn}
%A {\bf flat groupoid} is any Lie groupoid $\Sigma\tto \BB$ together with a involutive sub-bundle $H\subset T\Sigma$ which is an Ehresman connection on $t: \Sigma\to \BB$ which is also multiplicative in the sense that it is a subgroupoid of $T\Sigma \tto T\BB$.
%\end{defn}
%

%
%\begin{defn}
%A {\bf flat groupoid} is a Lie groupoid $\Sigma\tto \BB$ together with a subbundle $H\subset T\Sigma$ satisfying the following:
%\begin{itemize}
%\item is a flat Ehresman connection on $t: \Sigma\to \BB$,
%\item is also multiplicative in the sense that it is a subgroupoid of $T\Sigma \tto T\BB$.
%\end{itemize}
%\end{defn}

\begin{defn}\label{defn:Cartan-groupoid}
A {\bf Cartan groupoid} is a Lie groupoid $\Sigma\tto \BB$ together with a subbundle $\CC\subset T\Sigma$ satisfying the following:
\begin{itemize}
\item it is an Ehresman connection on $t: \Sigma\to \BB$ (Definition \ref{def:Ehr-conn}),
\item it is multiplicative (Definition \ref{defin:multiplicativity}). 
\end{itemize}
We say that $(\Sigma, \CC)$ is a {\bf flat Cartan groupoid} if $\CC\subset T\Sigma$ is involutive. 
\end{defn}

%
%\begin{defn}
%A {\bf flat groupoid} is a Lie groupoid $\Sigma\tto \BB$ together with a distribution $\CC\subset T\Sigma$ which is multiplicative (Definition \ref{defin:multiplicativity}) and a flat Ehresman connection on $t: \Sigma\to \BB$ (Definition \ref{def:Ehr-conn}).\end{defn}
%

\begin{rmk}\label{rk:Cartan-gpds-via-forms}
Typically, distributions $\CC$ as in the definition above appear as kernels of 1-forms 
\[ \omega\in \Omega^1(\Sigma, t^*E)\]
with values in representations $E$ of $\Sigma$, with the property that $\omega$ is 
%pointwise surjective and 
{\bf multiplicative} in the following sense: for each composable pair $(g, h)\in \Sigma^{(2)}$, 
\begin{equation}\label{multiplicativity}
(m^*\omega)_{(g,h)}= (pr_1^*\omega)_{(g,h)}+ g\cdot (pr_2^*\omega)_{(g,h)}
\end{equation}
where $m: \Sigma^{(2)}\to \Sigma$ is the multiplication map and $g\cdot: E_{s(g)}\to E_{t(g)}$ is the action of $g\in \Sigma^{(2)}$. Indeed, requiring also that $\omega$ is pointwise surjective and that
\[ \CC_{\omega}:= \textrm{Ker}(\omega)\subset T\Sigma\]
is complementary to the $t$-fibers, it follows that $(\Sigma, \CC_{\omega})$ is a Cartan groupoid.

Given a Cartan groupoid $(\Sigma, \CC)$, a multiplicative 1-form $\omega$ as above such that $\CC= \textrm{Ker}(\omega)$ will be called a {\bf Cartan form for $(\Sigma, \CC)$}. It is not difficult to see (and details can be found, within a more general framework, in \cite{MARIA}) that such a Cartan form always exists, and is even unique up to the obvious notion of isomorphism.
\end{rmk}

Given a flat Cartan groupoid $(\Sigma, \CC)$ we consider the action (\ref{eq:ind-right-action-2}) of $\Sigma$ on $T\BB$, the induced actions on $\Lambda^qT^*\BB$ and the corresponding operators
\[ \delta: C^p(\Sigma, \Lambda^qT^*\BB)\to C^{p+1}(\Sigma, \Lambda^qT^*\BB),\]
and then the Ehresmann connections $\CC^{(p)}$ on $t: \Sigma^{(p)}\to \BB$ given by (\ref{descr-H-p}) and the induced 
operators 
\[ d: \Omega^{q}_{\textrm{hor}}(\Sigma^{(p)})\to \Omega^{q+1}_{\textrm{hor}}(\Sigma^{(p)}).\]

%
%Given a flat groupoid $(\Sigma, \CC)$ we consider:
%\begin{itemize}
%\item the action (\ref{eq:ind-right-action-2}) of $\Sigma$ on $T\BB$, the induced actions on $\Lambda^qT^*\BB$ and the corresponding operators
%\[ \delta: C^p(\Sigma, \Lambda^qT^*\BB)\to C^{p+1}(\Sigma, \Lambda^qT^*\BB).\]
%\item the Ehresmann connections $\CC^{(p)}$ on $t: \Sigma^{(p)}\to \BB$ given by (\ref{descr-H-p}) and the induced 
%operators 
%\[ d: \Omega^{q}_{\textrm{hor}}(\Sigma^{(p)})\to \Omega^{q+1}_{\textrm{hor}}(\Sigma^{(p)}).\]
%\end{itemize}
%
\begin{defn}\label{def-Haefl-coh-gen}
The {\bf Haefliger complex} of a flat Cartan groupoid $(\Sigma, \CC)\tto \BB$ is the bicomplex 
\[ C^{p, q}_{\textrm{Haef}}(\Sigma, \CC):= C^p(\Sigma, \Lambda^qT^*\BB)= \Omega^{q}_{\textrm{hor}}(\Sigma^{(p)}) \]
endowed with the differentials $\delta$ and $d$ described above. Its cohomology is denoted
\[ H^{*}_{\textrm{Haef}}(\Sigma, \CC)\]
and is called the {\bf Haefliger cohomology of $(\Sigma, \CC)$}. 
\end{defn}
\begin{rmk}
Occasionally, when $\omega$ is a Cartan form for $(\Sigma, \CC)$, we write $H^*_{\rm Haef}(\Sigma, \omega)$ in place of $H^{*}_{\textrm{Haef}}(\Sigma, \CC)$.
\end{rmk}

The fact that we deal, indeed, with a double complex, and also that the double complex carries the same structures as the Bott-Shulman one, follows from the proof of the Proposition below, which was essentially carried out in the previous section:

\begin{prp} For any flat Cartan groupoid $(\Sigma, \CC)\tto \BB$, the Haefliger complex 
$C^{*,*}_{\rm Haef}(\Sigma, \CC)$ 
%(Definition \ref{def-Haefl-coh-gen}) 
becomes a cosimplicial DGA.
\end{prp}

\begin{exm}\label{ex-Cartan-gpd-J}
Of course, the motivating example for us is the flat Cartan groupoid $\Sigma= J^{\infty}\Gamma$ associated to any Lie pseudogroup $\Gamma$; the distribution on it is the Cartan connection $\CC^\infty$ described in subsection~\ref{The Cartan connection}. 
%See also the appendix, in particular Examples \ref{exm:Car-distr-J} and \ref{ex:Cartan-form-R} for the general constructions of the Cartan distribution and of the Cartan form. 
In this example, the Haefliger cohomology $H^*_{\rm Haef}(J^\infty\Gamma, \CC^\infty)$ from Definition~\ref{def-Haefl-coh-gen} is realized as Haefliger's differentiable cohomology $H^*_{\rm diff}(\Gamma)$ of $\Gamma$ from Definition~\ref{defn:diff-coh-gamma}. 
\end{exm}

\begin{exm}\label{ex-Cartan-gpd-G} Lie groups $G$, interpreted as Lie groupoids over a point $G\tto \star$, are automatically flat Cartan groupoids- the only possibility for $\CC$ being, of course, $\CC= 0$. Note however that even this very simple example possesses an interesting Cartan $1$-form (Remark~\ref{rk:Cartan-gpds-via-forms}): the Maurer-Cartan form $\omega_{MC}\in \Omega^1(G, \mathfrak{g})$. The resulting Haefliger cohomology is precisely the differentiable cohomology of $G$:
\[H^*_{\rm Haef}(G, \CC\equiv 0)= H^*_{\rm Haef}(G, \omega_{MC})=H^*_{\rm diff}(G).\]
In this way the differentiable cohomology of Lie groups and the differentiable cohomology of Lie pseudogroups are placed in a common framework. 
\end{exm}

\begin{exm}\label{ex-Cartan-gpd-G-action}
Another interesting example is that of action groupoids $G\ltimes M \tto M$ associated to actions of Lie groups $G$ on manifolds $M$. Recall that $G\ltimes M=G\times M$; a pair $(g, x)$ is viewed as an arrow from $x$ to $gx$ and the composition is given by 
\[ (h, gx)\cdot (g, x)= (hg, x).\]
In this case the source map $s= \textrm{pr}_2: G\times M\to M$ has an obvious flat Ehresman connection: $\CC= \CC_\textrm{can}$ consisting of the vectors tangent to the slices $G\times {x}$. It is not difficult to see that 
\[ (G\ltimes M,  \CC_\textrm{can})\]
is, indeed, a flat Cartan groupoid. Note that the associated Haefliger complex is
\[ C^{p, q}= C^p(G, \Omega^q(M)),\]
i.e. the bicomplex of differentiable cochains on $G$ with values in forms on $M$, where the two differentials are the Lie group differential (with coefficients) and the DeRham differential on $M$. Hence, in some sense, the cohomology is a combination of the DeRham cohomology of $M$ with the differentiable cohomology of $G$. In general, all these are related by a spectral sequence, and there are two extreme situations that can be made more explicit:
\begin{itemize}
\item when $M$ is contractible, one obtains the differentiable cohomology of $G$;
\item when $G$ is compact (the typical case when differentiable cohomology vanishes in positive degrees) and connected, one obtains the DeRham cohomology of $M$.
\end{itemize} 
\end{exm}

\begin{rmk} \label{rmk:only-action-gpds?} Actually, in the finite dimensional case, under simple topological assumptions, a flat Cartan groupoid $\Sigma$ must be isomorphic to an action groupoid. This is the case e.g. if the base $\BB$ is compact and $1$-connected, and the s-fibers of $\Sigma$ are $1$-connected. See e.g.~\cite{MARIA}, section $6$. This discussion will become more transparent in the next section, when moving to the infinitesimal picture.
\end{rmk}

%%%%%%%%%%%%%%%%%%%%%%%%%%%%
%%%%%%%%%%%%%%%%%%%%%%%%%%%%
%%%%%%%%%%%%%%%%%%%%%%%%%%%%
%%%%%%%%%%%%%%%%%%%%%%%%%%%%
%%%%%%%%%%%%%%%%%%%%%%%%%%%%
\subsection{Almost geometric structures}
\label{ssec:Formal geometric structures}
%%%%%%%%%%%%%%%%%%%%%%%%%%%%
%%%%%%%%%%%%%%%%%%%%%%%%%%%%
%%%%%%%%%%%%%%%%%%%%%%%%%%%%
%%%%%%%%%%%%%%%%%%%%%%%%%%%%
%%%%%%%%%%%%%%%%%%%%%%%%%%%%

The notion of flat Cartan groupoid (and related notions) also shows up when dealing with the
``almost" version of geometric structures. Here are two illustrations of ``almost" structures that
may be useful to have in mind:
\begin{itemize}
\item while symplectic structures on a manifold $M$ are non-degenerate 2-forms that are closed, almost symplectic structures are obtained when giving up the closedness condition;
\item while complex structures on $M$ may be interpreted as morphisms $J: TM\to TM$ with $J^{2}= -\textrm{Id}$ and for which the Nijenhuis tensor $\mathcal{N}_J$ vanishes, almost complex structures give up on the last condition. 
\end{itemize}

While the notion of $\Gamma$-structure encompasses many of the standard geometric structures, 
the notion of ``almost $\Gamma$-structure" appears as a very unifying general concept. When $\Gamma$ is transitive, 
%one can express everything 
the theory admits an equivalent formulation in terms of principal group bundles; in fact, under some assumptions on $\Gamma$, one ends up with the classical notions of $G$-structure (see, for example,~\cite{STERNBERG}) and (almost) integrable $G$-structure.
%This notion is rather standard when $\Gamma$ is a transitive: because of transitivity everything can be expressed in terms of principal bundles and then one talks about $G$-structures and (almost) integrable $G$-structures. 
In general one needs to appeal to groupoids and one discovers the notion of flat Cartan groupoid $(\Sigma, \CC)$ from Definition \ref{defn:Cartan-groupoid} and $(\Sigma, \CC)$-structure that we recall below. Indeed, the notion of ``flat Cartan'' is implicit in~\cite{FRANCESCO}, to which we refer for more details on ``almost structures''.

In some sense, the next definition is a generalization of Definition \ref{defn:Cartan-groupoid} in which we replace Lie groupoids $\Sigma\tto \BB$ by principal $\Sigma$-bundles $\pi: P\to M$ (over some manifold $M$):
\[
\xymatrix{
\Sigma \ar@<0.25pc>[d] \ar@<-0.25pc>[d]  & \ar@(dl, ul) &  P \ar[dll]^{\mu}\ar[dr]_{\pi} &    \\
\BB&  & &  M.}
\]
While before, in order to make sense of multiplicativity of $\CC\subset T\Sigma$, we made use of the differential of the multiplication and other structure maps of $\Sigma$, now we use the differential of the action map
\[ a: \Sigma\ltimes P\to P, \quad (g, p)\mapsto gp \quad (\textrm{defined on $\{(g, p): s(g)= \mu(p)\}$})\]
to obtain an action of $T\Sigma$ on $TP$. In particular, for $\CC\subset T\Sigma$, $\CC_P\subset TP$ we can consider
\[ \CC\cdot \CC_{P}:= \{(da)(v, w): v\in \CC, w\in \CC_P\ \textrm{with $ds(v)=d\mu(w)$}\}.\]

\begin{defn}\label{defn:Cartan-bundle}
Given a Cartan groupoid $(\Sigma, \CC)\tto \BB$, a {\bf principal $(\Sigma, \CC)$ -bundle} on a manifold $M$ is any principal $\Sigma$-bundle $\pi: P\to M$ endowed with an Ehresmann connection $\CC_{P}\subset TP$ with the property that $\CC\cdot \CC_{P}\subset \CC_{P}$. If additionally $\ker(d\mu)\cap \CC_P$ is an involutive distribution, we talk about {\bf $(\Sigma, \CC)$-structure}. A principal $(\Sigma, \CC)$ -bundle/$(\Sigma, \CC)$-structure is called {\bf flat} if $\CC_{P}$ is involutive.
\end{defn}
\begin{rmk}\label{rk:Pbundles_vs_structures}
As one sees from Example~\ref{ex:Principal_group_bundles}, principal $(\Sigma, \CC)$-bundles are more general objects than $(\Sigma, \CC)$-structures. That said, under flatness assumption, the two notions coincide - and the flat case is the one we focus on for the rest of this paper, since it encompasses all the geometric examples that we have in mind (Example~\ref{ex:formal_Gamma_structures}).
\end{rmk}

\begin{rmk}\label{rk:Cartan-bdles-via-forms}
Similar to Remark \ref{rk:Cartan-gpds-via-forms}, while $\CC$ arises as the kernel of a multiplicative 1-form $\omega\in \Omega^1(\Sigma, t^*E)$ with coefficients in a representation $E$ of $\mathcal{G}$, $\CC_P$ arises as the kernel of a 1-form 
\[ \theta\in \Omega^1(P, \mu^*E)\] such that the action $a:\Sigma\ltimes P\to P$
satisfies the multiplicativity property: 
%\begin{equation}\label{Mult_action}
%(a^*(\theta))_{(g, p)}=(pr_1^*(\theta))_{(g, p)}+(g\cdot pr_2^*(\theta))_{(g, p)}
%\end{equation}
\begin{equation}\label{Mult_action}
(a^*\theta)_{(g, p)}=(pr_1^*\omega)_{(g, p)}+(g\cdot pr_2^*\theta)_{(g, p)}
\end{equation}
for all $(g, p)\in \Sigma\ltimes_\mu P$. 

%Notice that $\mu^*E\cong \ker(d\pi)$, the vertical bundle of $\pi:P\to M$. In fact, the map given by
%\[\mu^*E\to \ker(d\pi),\quad (p, v_{\mu(p)})\mapsto da_{(1_{\mu(p)},p)}(v_{\mu(p)},0_p) \]
%is an isomorphism since the action is principal; furthermore, it descends 

For more details about multiplicative actions we refer to \cite{FRANCESCO}. 
% Notice that, for~\eqref{Mult_action} to make sense, the coefficent space of $\theta$ is forced to be $\mu^*E$. 
\end{rmk}

\begin{exm}\label{ex:Principal_group_bundles}
When $\Sigma= G$ is a Lie group as in Example \ref{ex-Cartan-gpd-G}, then principal $(G,0)$-bundles are the  same thing as principal $G$-bundles $\pi: P\to M$ endowed with a principal connection (in the standard sense). For instance, using the point of view of 1-forms as in the previous remark, since the 1-form corresponding to $G$ is the Maurer-Cartan form, for $P$ one is looking at $\omega_P\in \Omega^1(P, \mathfrak{g})$ satisfying 
\[
a^*(\theta)-g\cdot pr_2^*(\theta)=pr_1^*(\omega_{MC}).
\]
Evaluating on pairs $(0, v)$, where $v$ is any vector tangent to $P$, one obtains
\[
L_g^*(\theta)=g\cdot (\theta)
\] 
i.e. $G$-equivariance of $\theta$. Similarly, evaluating on pairs $(v,0)$, with $v\in \mathfrak{g}\cong T_eG$, one finds
\[
\theta(\hat{v}_p)=v
\]
where $\hat{v}$ is the vertical vector at $p$ induced by $v$. As for $(G, 0)$-structures, they correspond precisely to flat connections; see Remark~\ref{rk:Pbundles_vs_structures}.
\end{exm}

\begin{exm}\label{ex:Sigma-is-Cartan-bundle}
Any Cartan groupoid $(\Sigma, \CC)$ acting on itself by left multiplication along the target map is a principal $(\Sigma, \CC)$-bundle and a $(\Sigma, \CC)$-structure, since $\ker(dt)\cap \CC=\ker(ds)\cap \CC$ by multiplicativity and $\ker(dt)\cap \CC=0$ because $\CC$ is an Ehresmann connection. Moreover, $(\Sigma, \CC)$ is flat as a $(\Sigma, \CC)$-structure if and only if it is a flat Cartan groupoid. In fact, the multiplicativity of $\CC$ is exactly the multiplicativity of the action with respect to $\CC$ itself, and the action is clearly principal. 
\[
\xymatrix{
\Sigma \ar@<0.25pc>[d] \ar@<-0.25pc>[d]  & \ar@(dl, ul) &  \Sigma \ar[dll]^{t}\ar[dr]_{s} &    \\
\BB&  & &  \BB,}
\]
\end{exm}

\begin{defn}
Given a pseudogroup $\Gamma$ on $\BB$, a {\bf almost $\Gamma$-structure}  on a manifold $M$ 
is a flat $(J^\infty\Gamma, \CC^\infty)$-structure $(P, \CC_P)\to M$.
\[
\xymatrix{
(J^{\infty}\Gamma, \CC^\infty)\ar@<0.25pc>[d] \ar@<-0.25pc>[d]  & \ar@(dl, ul) &  (P, \CC_P)
\ar[dll]^{\mu}\ar[dr]_{\pi} &    \\
\BB&  & &  M},
\]
\end{defn}
\begin{rmk}
The flatness request in the definition above is motivated by the fact that the definiton of almost $\Gamma$-structure is meant to capture the idea of ``infinite jet of a $\Gamma$-structure''. In all the examples with geometric meaning, $P$ is a (a subset of) an infinite jet space and the corresponding $\CC_P$ is obtained as a limit of distributions and automatically flat, as in the proof Proposition~\ref{thm:Classical_cartan_dist}.  See the example below.
\end{rmk}

\begin{exm}\label{ex:formal_Gamma_structures} As indicated by the terminology and stressed in the previous remark, the motivating examples of almost $\Gamma$-structure arise when looking at the ``$\infty$-order data" induced by 
$\Gamma$-structures. To explain this, let us first look at the case when 
$\textrm{dim}(M)= \textrm{dim}(\BB)$, when $\Gamma$-structures are encoded in $\Gamma$-atlases $\mathcal{A}$. If $\mathcal{A}$ is a maximal $\Gamma$-atlas then the infinite jets of elements of $\mathcal{A}$, 
\[ J^{\infty}\mathcal{A}:= \{ j^{\infty}_{x}f: f\in \mathcal{A}, x\in M\}\subset J^{\infty}(M, \BB),\]
carries precisely the structure from the previous definition:
\[
\xymatrix{
(J^{\infty}\Gamma, \CC^\infty)\ar@<0.25pc>[d] \ar@<-0.25pc>[d]  & \ar@(dl, ul) &  (J^{\infty}\mathcal{A}, \CC^\infty_\mathcal{A}) \ar[dll]^{\mu}\ar[dr]_{\pi} &    \\
\BB&  & &  M,}
\]
% E.g., 
Here, $\CC^\infty_\mathcal{A}$ is the canonical Cartan distribution on the infinite jet space $J^{\infty}\mathcal{A}$, see Example~\ref{exm:Car-distr-J}.
$J^{\infty}\mathcal{A}$ fibers over $M$ via the map $\pi: j^{\infty}_{x}f\mapsto x$, over $\BB$ via the map $\mu: j^{\infty}_{x}f\mapsto f(x)$,  and it carries a left-action of $J^{\infty}\Gamma$:
\[j^{\infty}_{f(x)}\phi\cdot j^{\infty}_{x}f= j^{\infty}_{x} (\phi\circ f) \quad \textrm{(for $\phi\in \Gamma$, $f\in \mathcal{A}$)}.\]
The smooth structure on $J^{\infty}\mathcal{A}$ exploits the 
% he smooth structure is inherited from the 
open inclusion 
\[ J^{\infty}\mathcal{A}\subset J^{\infty}(M, \BB)\]
and the smooth structure on $J^{\infty}(M, \BB)$  (see Example~\ref{exam-Jinfty-R}). The Cartan distribution or form on $J^{\infty}(M, \BB)$ (see Examples \ref{exm:Car-distr-J}  and \ref{ex:Cartan-form-R} for the general constructions) induce the desired distribution (or form) on $J^{\infty}\mathcal{A}$, making it into a flat $(J^\infty\Gamma, \CC^\infty)$-structure. Actually, note that the entire diagram and all the operations involved sit inside the principal bundle given by the action of the (larger) groupoid $\Pi^{\infty}\tto \BB$ of infinite jets of diffeomorphisms (see the appendix) on the space $\Pi^{\infty}(M, \BB)$ of jets of (local) diffeomorphisms between $M$ and $\BB$:
\[
\xymatrix{
\Pi^{\infty}\ar@<0.25pc>[d] \ar@<-0.25pc>[d]  & \ar@(dl, ul) &  \Pi^{\infty}(M, \BB) \ar[dll]^{\mu}\ar[dr]_{\pi} &    \\
\BB&  & &  M}.
\]

\medskip

Let us now look at the general case (i.e. when $M$ and $\BB$ may have different dimensions). For an intrinsic approach, we represent $\Gamma$-structures by 
principal $\G$-bundles $\mathcal{P}$ over $M$ ($\G= \Ger(\Gamma)$). Then we define  
\[ \mathcal{P}^{\infty}\subset J^{\infty}(M, \BB)\]
consisting of jets of maps of type $\mu\circ \sigma$ where $\sigma$ is a section of $\pi: P\to M$ and one proceeds 
as above, obtaining all the desired structure from that of $J^{\infty}(M, \BB)$. One should be aware however that, 
if one embeds everything into a larger diagram (like the last one), one replaces $\Pi^{\infty}(M, \BB)$ by the similar space $\textrm{Subm}^{\infty}(M, \BB)$ defined using submersions (to ensure that the resulting action is free) but, even so, the action is not transitive. Hence, in general, it is only after we restrict to the $\Gamma$-version of the diagram, 
\[
\xymatrix{
J^{\infty}\Gamma\ar@<0.25pc>[d] \ar@<-0.25pc>[d]  & \ar@(dl, ul) &  \mathcal{P}^{\infty} \ar[dll]^{\mu}\ar[dr]_{\pi} &    \\
\BB&  & &  M,}
\]
 that we obtain a principal bundle (and a $(J^\infty\Gamma, \CC^\infty)$-structure).
 
 Note also that one could also proceed less intrinsically, using $\Gamma$-cocycles. The actual construction is, in principle, rather obvious: representing a $\Gamma$ structure on $M$ by a $\Gamma$-cocycle $c: M_{\mathcal{U}}\to \Ger(\Gamma)$, just compose it with the obvious morphism $\Ger(\Gamma)\to J^\infty\Gamma$ to get a $J^\infty\Gamma$-cocycle, hence a principal $J^\infty\Gamma$-bundle. Proceeding this way however, one still has to exhibit the Cartan distribution and prove the independence of the choice of the cocycle. 
 
%\[
%\xymatrix{
%\Pi^{\infty}_{\BB}\ar@<0.25pc>[d] \ar@<-0.25pc>[d]  & \ar@(dl, ul) &  \textrm{Subm}^{\infty}(M, \BB) \ar[dll]^{\mu}\ar[dr]_{\pi} &    \\
%\BB&  & &  M},
%\]
%the use of submersions ensures that the resulting action is free; however, in general it is not 
%one has to define $\Pi^{\infty}(M \BB)$ using jets of submersions (to make the action free) and, even so, the bundle over $M$ is not principal. It is just when we restrict to our objects that we obtain a principal (flat, Cartan) bundle 
%\[
%\xymatrix{
%J^{\infty}\Gamma\ar@<0.25pc>[d] \ar@<-0.25pc>[d]  & \ar@(dl, ul) &  \mathcal{P}^{\infty} \ar[dll]^{\mu}\ar[dr]_{\pi} &    \\
%\BB&  & &  M}.
%\]
\end{exm}

\begin{defn} An almost $\Gamma$-structure is said to be integrable if it is induced by a $\Gamma$-structure. 
\end{defn}

\begin{rmk} As we will discuss in subsection \ref{subsec:The case of Cartan bundles and almost structures}
and prove in Proposition \ref{prp:chr-map-formal-back}, 
% for any Lie pseudogroup $\Gamma$ on $\BB$, 
the differentiable characteristic maps 
\[ \kappa^{\mathcal{P}}: H^{*}_{\textrm{diff}}(\Gamma)\to H^*(M).\]
associated to $\Gamma$-structures $P$ on a manifold $M$ (where $\Gamma$ is a Lie pseudogroup) only depend on the underlying almost $\Gamma$-structures. In other words, one has similar characteristic maps 
\[ \kappa^{P}: H^{*}_{\textrm{diff}}(\Gamma)\to H^*(M).\]
% $k^{\mathcal{P}}$ 
associated to any almost $\Gamma$-structure $P$ on $M$. And the construction actually makes sense 
more generally, for any flat Cartan groupoid $(\Sigma, \CC)\tto \BB$ and flat $(\Sigma, \CC)$-structure $(P, \CC_P)\to M$, with associated characteristic maps defined on the Haefliger cohomology:
\[ \kappa^{P}: H^{*}_{\textrm{Haef}}(\Sigma, \CC)\to H^*(M).\]
\end{rmk}

\section{The infinitesimal  version of Haefliger cohomology}\label{The infinitesimal version of Haefliger cohomology}
%%%%%%%%%%%%%%%%%%%%%%%%%%%%
%%%%%%%%%%%%%%%%%%%%%%%%%%%%
%%%%%%%%%%%%%%%%%%%%%%%%%%%%
%%%%%%%%%%%%%%%%%%%%%%%%%%%%
%%%%%%%%%%%%%%%%%%%%%%%%%%%%
%%%%%%%%%%%%%%%%%%%%%%%%%%%%
%%%%%%%%%%%%%%%%%%%%%%%%%%%%
%%%%%%%%%%%%%%%%%%%%%%%%%%%%
%%%%%%%%%%%%%%%%%%%%%%%%%%%%
%%%%%%%%%%%%%%%%%%%%%%%%%%%%

%%%%%%%%%%%%%%%%%%
%%%%%%%%%%%%%%%%%%
%%%%%%%%%%%%%%%%%%
%%%%%%%%%%%%%%%%%%
%%%%%%%%%%%%%%%%%%
%%%%%%%%%%%%%%%%%%
%%%%%%%%%%%%%%%%%%
\subsection{Lie algebroids}
%%%%%%%%%%%%%%%%%%
%%%%%%%%%%%%%%%%%%
%%%%%%%%%%%%%%%%%%
%%%%%M%%%%%%%%%%%%%
%%%%%%%%%%%%%%%%%%
%%%%%%%%%%%%%%%%%%
%%%%%%%%%%%%%%%%%%

In this section, we investigate the infinitesimal data associated with (flat) Cartan groupoids. This can be seen as part of the ``Lie philosophy" of linearizing the global group-like structures. Of course, the starting point is the infinitesimal counterpart of Lie groupoids: Lie algebroids, i.e. vector bundles $A\to \BB$ endowed with a Lie bracket $[\cdot, \cdot]$ on the space $\Gamma(A)$ of section of $A$ satisfying the Leibiniz type identity 
\[ [\alpha, f\cdot\beta]= f\cdot [\alpha, \beta]+ L_{\rho(\alpha)}(f) \beta\quad \textrm{for all $\alpha, \beta\in \Gamma(A), f\in C^{\infty}(\BB)$}\]
for some vector bundle morphism $\rho: A\to T\BB$ (necessarily unique, and called the anchor of $A$). 
The Lie algebroid $A= \textrm{Lie}(\Sigma)$ of a Lie groupoid $\Sigma \tto \BB$ is defined as
\[
A:=\ker(ds)_{\BB}, 
\]
with the anchor 
\[ \rho:= dt|_{A}: A\to T\BB,\]
and with Lie bracket obtained by identifying the sections of $A$ with right invariant vector fields on $\Sigma$. For the last notion we consider arrows $g: x\to y$, the corresponding right multiplication $R_g:s^{-1}(x)\to s^{-1}(y)$ and the induced tangent map 
\[
dR_g:s^{-1}(x)\to s^{-1}(y).
\]
With this, a vector field $X$ on $\Sigma$ is called {\bf right invariant} if
\begin{itemize}
\item it is {\bf $s$-vertical};
\item one has $X_{hg}=dR_g(X_h)$ for all the composable pairs $(h, g)$.
\end{itemize}
The key remark is that any section $\alpha$ of $A$ gives rise to the right invariant vector field $\alpha^R$ given by:
\[
\alpha^R_g= dR_g(\alpha_{t(g)})
\]and then $[\alpha, \beta]$ is defined by the condition that 
\[
[\alpha, \beta]^R=[\alpha^R, \beta^R].
\]

\medskip

Next, various notions and constructions valid for Lie groupoids $\Sigma\tto \BB$ have an infinitesimal counterpart, defined for Lie algebroids $A\to \BB$. A basic example is the notion of representation which, on the infinitesimal side, takes us to representations of $A$, i.e. vector bundles $E\to \BB$ endowed with a bilinear operator
\[
\nabla:\Gamma(A)\times \Gamma(E)\to \Gamma(E), \quad (\alpha, \sigma)\mapsto \nabla_{\alpha}(\sigma)
\]
which is  $C^\infty(\BB)$-linear in the first argument, satisfies the Leibniz rule
\[\nabla_\alpha(f \sigma)=f\nabla_\alpha(\sigma)+L_{\rho(\alpha)}(f)\sigma\quad 
\textrm{for all $\alpha\in \Gamma(A)$, $\sigma\in \Gamma(E)$ and $f\in C^\infty(\BB)$},\]
and is flat in the sense that 
\begin{equation}\label{curv-A-conn} 
\nabla_{[\alpha, \beta]}= [\nabla_{\alpha}, \nabla_{\beta}].
\end{equation}
Giving up the last condition one talks about {\bf $A$-connections on $E$} and then the difference of the two terms from the flatness condition defines a tensor 
\[ k_{\nabla}\in \textrm{Hom}(\Lambda^2A^*, \textrm{End}(E)), \]
called the {\bf curvature} of $\nabla$.

\medskip

If $A= {\rm Lie}(\Sigma)$ and $E$ is a representation of $\Sigma$, then $E$ can be turned into a representation of $A$ by defining 
\begin{equation}\label{eq:Lie:on:repres}
\nabla_\alpha(\sigma)(x)=\left.\frac{d}{d\epsilon}\left({\phi^\epsilon_\alpha(x)^{-1}\cdot \sigma\left(\phi^\epsilon_{\rho(\alpha)}(x)\right)}\right)\right|_{\epsilon=0}
\end{equation}
%\[
%\nabla_\alpha(\sigma)(x)=\left.\frac{d}{d\epsilon}\right|_{\epsilon=0}\left({\phi^\epsilon_\alpha(x)\cdot \sigma\left(\phi^\epsilon_{\rho(\alpha)}(x)\right)}\right)
%\]
%\[
%\nabla_\alpha(\sigma)(x)=\left.\frac{d}{d\epsilon}\right|_{\epsilon=0}\phi^\epsilon_\alpha(x)\cdot \sigma\left(\phi^\epsilon_{\rho(\alpha)}(x)\right)
%\]
for all $\alpha\in \Gamma(A)$, $\sigma\in \Gamma(E)$ and $x\in \BB$; here, $\phi^\epsilon_{\alpha}$, called the {\bf flow} of $\alpha$, is a one parameter family of bisections such that 
\[
\left.\frac{d}{d\epsilon}\phi^\epsilon_{\alpha}(x)\right|_{\epsilon=0}=\alpha(x)
\]
%\[
%\left.\frac{d}{d\epsilon}\right|_{\epsilon=0}\phi^\epsilon_{\alpha}(x)=\alpha(x)
%\]
for all $x\in \BB$, while $\phi^\epsilon_{\rho(\alpha)}$ is the flow of the vector field $\rho(\alpha)$. A basic well-known result says that this construction $\textrm{Rep}(\Sigma)\to \textrm{Rep}(A)$ is injective if $\Sigma$ is $s$-connected (i.e. the $s$-fibers are connected) and 1-1 if the $s$-fibers are $1$-connected (connected and simply connected).

\medskip
\begin{exm}\label{exam-mult-correct-coeff} In subsection~\ref{subsection:action_TX}, we have seen that, for any Cartan groupoid $(\Sigma, \CC)\tto \BB$, $T\BB$ is a representation of $\Sigma$. 
Next to this, there is another natural representation of $\Sigma$: $A$ itself! Indeed,
\begin{align}\label{A_representation}
\Sigma_{s}\times_{\pi} A\to A,\ \ (g, \alpha_{s(g)})\to dR_{g^{-1}}(dm(hor^s_g(\rho(\alpha_{s(g)})),\alpha_{s(g)});
\end{align}
defines a representation of $\Sigma$ on $A$.

This also allows us to expand Remark \ref{rk:Cartan-gpds-via-forms}, where we pointed out that $\CC\subset T\Sigma$ can be described as the kernel of a multiplicative 1-form $\omega\in \Omega^1(\Sigma, t^*E)$ with coefficients in some representation $E$ of $\Sigma$. Since $\omega$ is pointwise surjective and its kernel is complementary to $\textrm{Ker}(dt)$, it follows that $E$ is actually isomorphic to $\textrm{Ker}(dt)|_\BB$ hence, using the inversion, also to 
$\textrm{Ker}(ds)|_\BB= A$. Using the multiplicativity condition, one can check that this is actually an isomorphism of representations, so that one may say (as in Corollary~\ref{cor:MC-eq-for-Cartan-form}) that 
\[ \omega\in \Omega^1(\Sigma, t^*A).\]
\end{exm}

\medskip

Another illustration of the Lie philosophy is the infinitesimal counterpart of differentiable cohomology- the cohomology of
Lie algebroid $A\to \BB$ with coefficients in representations $E= (E, \nabla)$: in degree $p$ one defines
$C^p(A, E)= \Gamma(\Lambda^p A^*\otimes E)$, and then the {\bf Koszul differential} 
%
%% Differentiable cohomology has an infinitesimal version.
%a {\bf $p$-cochain with coefficients in $E$} is a section of $\Lambda^pA^*\otimes E$; the space of such section is denoted by $C^p(A, E)$. The {\bf Koszul differential} 
%\begin{align*}
%d: C^p(A, E)&\to C^{p+1}(A, E)\\
%\omega&\to d\omega
%\end{align*}
\[ d: C^p(A, E)\to C^{p+1}(A, E)\]
is defined by
\begin{align*}
d\omega(\alpha_1, \dots \alpha_{p+1})=\sum\limits_{i<j}(-1)^{i+j}\omega([\alpha_i, \alpha_j], \alpha_1, \dots \hat{\alpha}_i, \dots, \hat{\alpha}_j, \dots, \alpha_{q+1})&\\
+\sum\limits_{i}(-1)^i\nabla_{\alpha_i}(\alpha_1, \dots, \hat{\alpha}_i, \dots \alpha_{q+1})&
\end{align*}
The cohomology of the cochain complex $(C^p(A, E), d)$ denoted by $H^*(A,E)$ is the {\bf algebroid cohomology with coefficents in $E$}. Lie algebra cohomology and DeRham cohomology are two examples, obtained respectively when $A=\mathfrak{g}$ (the Lie algebroid over a point) and $A=T\BB$ (where the anchor is the identity and the bracket is the Lie bracket of vector fields).
Finally, if $A= \textrm{Lie}(\Sigma)$ and $E$ comes from a representation of $\Sigma$, 
there is a natural cohomology map (known as the {\bf van Est map})
\[
VE: H^*_d(\Sigma, E)\to H^* (A, E).
\]
Here, the left hand side is the {\bf differentiable cohomology} of $\Sigma\tto \BB$ with coefficents in $E\to \BB$, defined as for Lie groups: the cohomology of the subcomplex of groupoid cochains with coefficents in $E$ (that is, sections of $t^*E$, where $t:\Sigma^{(p)}\to \BB$ takes the target of the first element) that are smooth.
If $\Sigma$ has cohomologically $n$-connected $s$ fibers, the map is an isomorphism in degree lower than $n$ and injective in degree $n+1$. This was proven in~\cite{VANESTPROOF} for Lie groups and generalized in~\cite{MARIUS} to Lie groupoids.

%%%%%%%%%%%%%%%%%%
%%%%%%%%%%%%%%%%%%
%%%%%%%%%%%%%%%%%%
%%%%%%%%%%%%%%%%%%
%%%%%%%%%%%%%%%%%%
%%%%%%%%%%%%%%%%%%
%%%%%%%%%%%%%%%%%%
\subsection{Cartan algebroids and (infinitesimal) Haefliger cohomology}
%%%%%%%%%%%%%%%%%%
%%%%%%%%%%%%%%%%%%
%%%%%%%%%%%%%%%%%%
%%%%%%%%%%%%%%%%%%
%%%%%%%%%%%%%%%%%%
%%%%%%%%%%%%%%%%%%
We now proceed with the discussion of the infinitesimal counterpart of Cartan groupoids and their Haefliger cohomology. 
First of all, it may not come as a surprise that, instead of Ehresman connection on groupoids, one now considers linear (vector-bundle) connections on algebroids:
\[ \nabla:  \mathfrak{X}(\BB)\times \Gamma(A)\to \Gamma(A), \quad (X, \alpha)\mapsto \nabla_X(\alpha).\]
While an Ehresman connection gave rise to the ``quasi-action" (\ref{eq:ind-right-action-2}) of $\Sigma$ on $T\BB$, $\nabla$ will give rise to a ``quasi-action" of $A$ on $T\BB$, i.e. an $A$-connection. Namely:
\begin{align}\label{TM_conn}
\nabla^{T\BB}:\Gamma(A)\times\Gamma(T\BB)&\to \Gamma(T\BB)\nonumber \\
(\alpha, X)&\to [\rho(\alpha), X]+\rho(D_X(\alpha)).
\end{align}
Incidentally, let us point out that also the action of $\Sigma$ on $A$ from Example \ref{exam-mult-correct-coeff}  has an infinitesimal counterpart, namely the $A$-connection on $A$ 
\begin{align}\label{A_conn}
\nabla^A:\Gamma(A)\times\Gamma(A)&\to \Gamma(A)\nonumber \\
(\alpha, \alpha')&\to D_{\rho(\alpha')}\alpha+[\alpha,\alpha'].
\end{align}
What is a bit more subtle is to find what the infinitesimal analogue of the multiplicativity condition is. 
For the connection $\CC$ on $\Sigma$, being closed under multiplication was encoded in the equation (\ref{eq:m2}), which can be thought as the vanishing of some kind of 2-cochain on $\Sigma$:
\[ \Sigma^{(2)}\ni (g_1, g_2)\mapsto dm\left( \textrm{hor}_{g_1}(\_), \textrm{hor}_{g_2}(g_{1}^{*}\_)\right)- \textrm{hor}_{g_1g_2}(\_)\in \textrm{Hom}(T_{t(g_1)}\BB, A_{s(g_2)}).\]
One may think of this 2-cochain as some kind of ``multiplicativity curvature" of $\CC$. For linear connection the situation is a bit similar but we end up with 2-cochains on $A$.

\begin{defn} Given a Lie algebroid $A\to \BB$ and a connection $\nabla$ on $A$, the {\bf basic curvature of $\nabla$} is defined as the tensor 
\[ k_{\nabla}^{\textrm{bas}}\in C^2(A, \textrm{Hom}(T\BB, A)),\]
\begin{equation}\label{Spencer_compatibility}
k_{\nabla}^{\textrm{bas}}(\alpha, \alpha')X:= 
\nabla_X[\alpha,\alpha']- [\nabla_X(\alpha), \alpha']- [\alpha, \nabla_X\alpha']- 
\nabla_{\nabla^{TM}_{\alpha'}(X)} \alpha+ \nabla_{\nabla^{TM}_{\alpha}(X)} \alpha',
\end{equation}
where $\nabla^{T\BB}$ is given by (\ref{TM_conn}). We say that $\nabla$ is {\bf infinitesimally multiplicative} if  
$k_{\nabla}^{\textrm{bas}}= 0$. 

%A {\bf Cartan algebroid} is an algebroid $A$ together with a connection $\nabla$ on it which is infinitesimally multiplicative. It is called {\bf flat} if $\nabla$ is a flat connection. 
\end{defn}

Note that the basic curvature is related to the $A$-curvature of $\nabla^{T\BB}$ by:
\[ k_{\nabla^{T\BB}}= \rho \circ k_{\nabla}^{\textrm{bas}};\]
% in particular, for Cartan algebroids, $TM$ becomes a representation of $A$ (of course, this is just the infinitesimal counterpart of the fact that, for Cartan groupoids, $TM$ is a representation of $\Sigma$). 
and, similarly, to the $A$-curvature of $\nabla^A$ by 
\[ k_{\nabla^{A}}= k_{\nabla}^{\textrm{bas}}\circ \rho .\]

\begin{defn} A {\bf Cartan algebroid} is an algebroid $A$ together with a connection $\nabla$ on it which is infinitesimally multiplicative (cf. the previous definition). It is called {\bf flat} if $\nabla$ is a flat connection. 
\end{defn}

\begin{rmk} The name ``Cartan algebroid'' comes from~\cite{BLAOM}, where the author considers infinitesimally multiplicative connections on Lie algebroids as a generalization of Cartan geometries~\cite{SHARPE}.
\end{rmk}
We deduce the following. 

\begin{cor} For any Cartan algebroid $(A, \nabla)$, (\ref{TM_conn}) and (\ref{A_conn}) make $T\BB$ and $A$ into representations of $A$. Moreover, $\rho: A\to T\BB$ is $A$-equivariant.
%One can prove by a computation (see~\cite{MARIA}) that $\nabla^A$ and $\nabla^{TM}$ are actually representations. In order to prove they are induced by the ones of $\mathcal{G}$, one uses the explicit formula for the infinitesimal representation induced by a groupoid representation (see, for example,~\cite{MARIUSRUI}).
\end{cor}

\begin{exm}\label{ex-Cartan-alg-J} The infinitesimal analogue of Example \ref{ex-Cartan-gpd-J} is the (pf) Lie algebroid $A^{\infty}$  of the (pf) Lie groupoid $J^\infty\Gamma$; we refer to our appendix, and in particular to Examples~\ref{exam-pair-groupoid} and~\ref{exam-Lie-pseudogroups}, for more details and terminology. A (strict) atlas for $A^\infty$ is given by the Lie algebroids $A^k\to \BB$ associated to $J^k\Gamma$, $k\in \mathbb{N}$, together with the maps $A^k\to A^{k-1}$ obtained deriving the projections $J^k\Gamma\to J^{k-1}\Gamma$. Since we assume normality, $A^\infty\cong\lim\limits_{\longleftarrow}A^k$. The flat infinitesimally multiplicative connection $\nabla$ on $A^\infty$ is the restriction of the Spencer operator \eqref{eq:nabla-Cartan} (see also~\cite{SPENCER} for a discussion of Spencer operators) .
%Here is also a direct description of the action of $A^{\infty}$  on $A^{\infty}$ etc. 
We notice here that $A^\infty$ can be defined as the space of infinite jets of (local) $\Gamma$-vector fields (see Definition~\ref{defn:Gamma-vector-fields}), that is
\[A^\infty=\{j^\infty_x X:\ x\in \dom(X),\ X\text{ is a } \Gamma\text{-vector field}\}.\]
%where $U_x$ is some open neighbourhood of $x$ in $\BB$ and $\phi^t_X$ denotes the flow of $X$. 

One can give an explicit description of the action of $J^\infty\Gamma$ on $A^\infty$ from example~\ref{exam-mult-correct-coeff} (which can be derived to provide a formula for the representation of $A^\infty$ on itself). Let $\alpha_x\in A^\infty_x$; we can write
\[
\alpha_x= j^\infty_x\left(\left.\frac{d}{dt}\right|_{t=0} \phi^t\right)
\]
with $\phi^t\in \Gamma$ and $\phi^0=id$. If $j^\infty_xf \in J^\infty\Gamma$ then one has
\[
j^\infty_x f \cdot \alpha_x= j^\infty_{f(x)}\left(\left.\frac{d}{dt}\right|_{t=0} f\circ \phi^t\circ f^{-1}\right)
\]
%With this at hand, one can apply the general formula~\eqref{eq:Lie:on:repres} to get the representation of $A^\infty$ on itself.
Finally, let us notice that $A^\infty$ is also a bundle of Lie algebras; in fact, the fiber $A^\infty_x$ is precisely the space $\mathfrak{a}_x(\Gamma)$ of almost $\Gamma$-vector fields from Definition~\ref{defn:Gamma-vector-fields} whose bracket is given by the formula
\[
[j^\infty_x X, j^\infty_x Y]=j^\infty_x[X,Y].
\]
%; in fact, there is a fiberwise Lie bracket $[\ ,\ ]$ given by the formula
%\[
%[j^\infty_x X, j^\infty_x Y]=j^\infty_x[X,Y]
%\]
%for $j^\infty_xX, j^\infty_x Y\in A^\infty$. Since $j^\infty_xX$ and $j^\infty_x Y$ are in particualr formal vector fields, they can be interpreted as vector fields in $\BB$ whose coefficients are formal power series. Using this interpretation, the bracket above is computed just as the ordinary Lie bracket of vector fields, which explains our choice of notation.  
The existence of this fiberwise Lie bracket is not a coincidence; an analogue structure exists on any Cartan algebroid, as we explain in the next example.
\end{exm}

\begin{exm}\label{exam-Cart-algbrds-actions} The infinitesimal analogue of Example \ref{ex-Cartan-gpd-G-action} is that of action algebroids $\mathfrak{g}\ltimes \BB$ associated to infinitesimal actions $\rho: \mathfrak{g}\to \mathfrak{X}(\BB)$ of Lie algebras. Recall that, as a vector bundle, $\mathfrak{g}\ltimes \BB$ is the product bundle $\mathfrak{g}\times \BB$, the anchor is given by the infinitesimal action and the bracket is the unique Lie bracket satisfying the Leibniz identity and the identity $[c_v, c_w]= c_{[v, w]}$, for all constant sections $c_v, c_w$ with $v, w\in \mathfrak{g}$. Endowed with the canonical flat connection, it becomes a flat Cartan algebroid. 

As in Remark \ref{rmk:only-action-gpds?}, these action algebroids exhaust most of the examples in the finite dimensional case. The key remark here is that, given a flat Cartan algebroid $(A, \nabla)$, each fiber of $A$ has a natural structure of Lie algebra. In fact, the bracket
% 
 %\footnote{{\color{red}  NOTE THAT I CHANGED THE SIGN HERE (FOR SEVERAL REASONS- SEE NEXT FOOTNOTES); BUT WE HAVE TO MAKE SURE WE ADDAPT TO THIS CHNAGE THE OTHER PLACES WHERE THIS BRACKET IS USED}}
% 
%\begin{eqnarray}\label{pt-bracket}
%\{\ ,\ \}_{\textrm{pt}}: \Gamma(A)\times \Gamma(A) &\to \Gamma(A)\\
%(\alpha,\beta) &\to \nabla_{\rho(\alpha)}(\beta)-\nabla_{\rho(\beta)}(\alpha)-[\alpha,\beta]\nonumber
%\end{eqnarray}
%\begin{eqnarray}\label{pt-bracket}
%\{\ ,\ \}_{\textrm{pt}}: & \Gamma(A)\times \Gamma(A) &\to \Gamma(A)\\
%& (\alpha,\beta) &\to \nabla_{\rho(\alpha)}(\beta)-\nabla_{\rho(\beta)}(\alpha)-[\alpha,\beta]\nonumber
%\end{eqnarray}
\begin{eqnarray}\label{pt-bracket}
\{\ ,\ \}_{\textrm{pt}}: & \Gamma(A)\times \Gamma(A) &\to \Gamma(A)\\
& (\alpha,\beta) &\to [\alpha,\beta]- \nabla_{\rho(\alpha)}(\beta)+\nabla_{\rho(\beta)}(\alpha)\nonumber
\end{eqnarray}
is $C^{\infty}(\BB)$-linear in its entries and satisfies Jacobi. Therefore, it makes $(A, \{\ ,\ \}_{\textrm{pt}})$ into a bundle of Lie algebras (explaining also the notation ``pt" to indicate ``pointwise bracket"). 

\begin{defn} \label{defn:extended-isotropy}
For $x\in \BB$, the {\bf extended isotropy Lie algebra} of $(A, \nabla)$ at $x\in \BB$, denoted $\extg_x(A)$, is $A_x$ endowed with the pointwise bracket (\ref{pt-bracket}) evaluated at $x$.
\end{defn}

The terminology indicates the fact that $\extg_x(A)$ contains, as a Lie subalgebra
% \footnote{{\color{red}  ONE REASON TO CHANGE THE SIGN IN THE POINTWISE BRACKET}}, 
the isotropy Lie algebra of $A$ at $x$, $\mathfrak{g}_x(A)$ (defined for any Lie algebroid as the kernel of the anchor equipped with the restriction of the bracket on $\Gamma(A)$). 

Note also that $\nabla$ acts by derivations w.r.t. $\{\ ,\ \}_{\textrm{pt}}$. In particular, the parallel transport w.r.t. 
$\nabla$ induces Lie algebra isomorphisms between the extended isotropy Lie algebras at different points. Hence, when the base is simply connected, fixing a point $x_0$ 
% and considering $\extg:= \extg_{x_0}(A)$, 
one obtains an isomorphism of $A$ with $\extg_{x_0}\times \BB$ and it is straightforward to check that, actually, this is an isomorphism of algebroids. 
\end{exm}

Finally, the precise relationship with Cartan groupoids:

\begin{thm} For any Cartan groupoid $(\Sigma, \CC)\tto \BB$ its Lie algebroid $A$ becomes a Cartan algebroid with $\nabla^{\CC}: \mathfrak{X}(\BB)\times \Gamma(A)\to \Gamma(A)$ given by:
\[
\nabla^{\CC}_X(\alpha)=[{\rm hor}^{\CC}(X),\alpha^R]|_{\BB}.
\]
Moreover, $\nabla^{\CC}$ is flat if $\CC$ is so. Furthermore, when $\Sigma$ has $1$-connected s-fibers, the construction 
\[ \CC\mapsto \nabla^{\CC}\]
is a 1-1 correspondence between Cartan connections on $\Sigma$ and infinitesimally multiplicative connections on $A$. 
%
%Furthemore, starting with $(\Sigma, \CC)\tto M$ and considering the corresponding $(A, \nabla= \nabla^{\CC})$, 
%\begin{enumerate} 
%\item[(1)] the infinitesimal actions of $A$ on $TM$ and on $A$ (i.e. (\ref{TM_conn}) and (\ref{A_conn})) are induced by the actions of $\Sigma$ on $TM$ and $A$ cf (\ref{eq:ind-right-action-2}) and Example \ref{exam-mult-correct-coeff}).
%\item[(2)] the Cartan form of $(\Sigma, \CC)$ can be realised as multiplicative form 
%\[ \omega\in \Omega^1(\G, t^*A).\]
%\item[(3)] using the DeRham differential on $\Sigma$ with coefficients in the (flat) pull-back bundle $t^*A$ (with connection $t^*\nabla$), 
%% differential induced by the pull-back connection $t^*\nabla$ on $t^*A$, 
%and the pointwise bracket (\ref{pt-bracket}), $\omega$ satisfies the Maurer-Cartan equation 
%\[ d_{t^*A}(\omega)+ \frac{1}{2}\{\omega, \omega\}_{\textrm{pt}}= 0.\]
%\end{enumerate}
\end{thm}

This theorem is a particular case of the main result of~\cite{MARIA}. 
Using Example \ref{exam-mult-correct-coeff} and the explicit formula for computing infinitesimal actions induced from groupoid actions (see above), one deduces the first two parts of the following. The last part will be discussed in greater generality in Proposition \ref{prp:MC-eq-CP} below.

\begin{cor}\label{cor:MC-eq-for-Cartan-form}
If $(A, \nabla= \nabla^{\CC})$ is the Cartan algebroid of a Cartan groupoid $(\Sigma, \CC)$, 
\begin{enumerate} 
\item[(1)] $\nabla^{T\BB}$ and $\nabla^A$ (i.e. (\ref{TM_conn}) and (\ref{A_conn})) are induced by corresponding representations of $\Sigma$ (cf. (\ref{eq:ind-right-action-2}) and Example \ref{exam-mult-correct-coeff}).
\item[(2)] The Cartan form of $(\Sigma, \CC)$ can be realized as a multiplicative form 
\[ \omega\in \Omega^1(\G, t^*A).\]
\item[(3)] If $(\Sigma, \CC)$ is flat, using the DeRham differential on $\Sigma$ with coefficients in the (flat) pull-back bundle $t^*A$ (with connection $t^*\nabla$), 
% differential induced by the pull-back connection $t^*\nabla$ on $t^*A$, 
and the pointwise bracket (\ref{pt-bracket}), $\omega$ satisfies the Maurer-Cartan equation 
%
% \footnote{{\color{red}  ANOTHER REASON TO CHANGE 
 %THE SIGN IN THE POINTWISE BRACKET}}
%
\[ d_{t^*A}(\omega)+ \frac{1}{2}\{\omega, \omega\}_{\textrm{pt}}= 0.\]
\end{enumerate}
\end{cor}

With the previous discussions in mind, the infinitesimal version of the Haefliger cohomology from Definition \ref{def-Haefl-coh-gen} is pretty clear. Given a flat Cartan algebroid $(A, \nabla)$ we first consider the infinitesimal action of $A$ on $T\BB$ given by (\ref{TM_conn}), the induced actions on $\Lambda^qT^*\BB$ and the corresponding Koszul operators
\[ d_A: C^p(A, \Lambda^qT^*\BB)\to C^{p+1}(A, \Lambda^qT^*\BB).\]
Similarly, we consider the DeRham differential on $\BB$ with coefficients in the flat bundles $\Lambda^pA^*$ (endowed with the flat connections induced by $\nabla$)
\[ d: \Omega^q(\BB, \Lambda^pA^*)\to \Omega^{q+1}(\BB, \Lambda^pA^*).\]

\begin{defn}\label{def-Haefl-coh-gen-inf}
The {\bf Haefliger complex} of a flat Cartan algebroid $(A, \nabla)\to \BB$ is the bicomplex 
\[ C^{p, q}_{\textrm{Haef}}(A, \nabla):= C^p(A, \Lambda^qT^*\BB)= \Omega^q(\BB, \Lambda^pA^*) \]
endowed with the differentials $d_A$ and $d$ described above. Its cohomology is denoted
\[ H^{*}_{\textrm{Haef}}(A, \nabla)\]
and is called the {\bf (infinitesimal) Haefliger cohomology of $(A, \nabla)$}. 
\end{defn}

We leave it to the reader to check that, indeed, this is a double complex (an argument arises as a ``bonus" of the discussion from the next subsection). As we shall see (and as expected), this cohomology will be related to $H^{*}_{\textrm{Haef}}(\Sigma, \CC)$ via a van Est map, which is an isomorphism under the usual topological conditions on the $s$-fibers of $\Sigma$.

%%%%%%%%%%%%%%%%%%%%
%%%%%%%%%%%%%%%%%%%%
%%%%%%%%%%%%%%%%%%%%
%%%%%%%%%%%%%%%%%%%%
%%%%%%%%%%%%%%%%%%%%
\subsection{Matched pairs and the double}
%%%%%%%%%%%%%%%%%%%%
%%%%%%%%%%%%%%%%%%%%
%%%%%%%%%%%%%%%%%%%%
%%%%%%%%%%%%%%%%%%%%
%%%%%%%%%%%%%%%%%%%%
Here we point out that flat Cartan algebroids and their Haefliger cohomology fits perfectly in the framework of matched pairs. Very briefly, a matched pair of Lie algebroids is a pair $(A_1, A_2)$ of Lie algebroids (over the same manifold $\BB$) together with a Lie algebroid structure on their direct sum, $(D:= A_1\oplus A_2, [-, -]_D, \rho_D)$ , such that $A_1$ and $A_2$ are sub-algebroids. This condition forces the anchor to be 
\[ \rho_D(\alpha_1, \alpha_2)= \rho_{A_1}(\alpha_1)+ \rho_{A_2}(\alpha_2).\]
However, while the bracket $[-, -]_D$ is determined by the brackets of the two algebroids when applied on elements of type $(\alpha_1, 0)$ or of type $(0, \alpha_2)$, one still have some freedom when combining the two types of elements. Actually, writing the components of $[(\alpha_1,0), (0,\alpha_2)]_D$ as
\[ [(\alpha_1,0), (0,\alpha_2)]_D= 
(-\nabla^2_{\alpha_2}\alpha_1,\nabla^1_{\alpha_1}\alpha_2)\]
one finds the explicit description of matched pairs:

\begin{defn}\label{def:matched_pair}
A matched pair of Lie algebroids consists of two Lie algebroids $A_1$ and $A_2$, a flat $A_1$-connection $\nabla^1$ on $A_2$ and a flat $A_2$-connection $\nabla^2$ on $A_1$ such that the identities
\begin{itemize}
\item[i)] 
\begin{align*}
[\rho_1(\alpha),\rho_2(\beta)]=-\rho_1(\nabla^2_\beta\alpha)+\rho_2(\nabla^1_\alpha\beta)
\end{align*}
\item[ii)]
\begin{align*}
\nabla^1_\alpha[\beta^1, \beta^2]=[\nabla^1_\alpha\beta^1, \beta^2]+[\beta^1, \nabla^1_\alpha\beta^2]+\nabla^1_{\nabla^2_{\beta^2}\alpha}\beta^1-\nabla^1_{\nabla^2_{\beta^1}\alpha}\beta^2
\end{align*}
\item[iii)]
\begin{align*}
\nabla^2_\beta[\alpha^1, \alpha^2]=[\nabla^2_\beta\alpha^1, \alpha^2]+[\alpha^1, \nabla^2_\beta\alpha^2]+\nabla^2_{\nabla^1_{\alpha^2}\beta}\alpha^1-\nabla^2_{\nabla^1_{\alpha^1}\beta}\alpha^2
\end{align*}
\end{itemize}
hold.
\end{defn}
It is well-known \cite{MACKENZIE, MOKRI}, and straightforward to check that, indeed, these conditions are equivalent to the fact that the resulting bracket and anchor on $A_1\oplus A_2$ make it into a Lie algebroid. The key remark for us is:

\begin{thm}\label{Double_algebroid}
There is a one to one correspondence between flat Cartan algebroids and matched pairs where one of the two algebroids is $T\BB$.
\end{thm}

\begin{proof}
First, let us start with a matched pair formed by $(A, \nabla^{T\BB})$ and $(T\BB, \nabla)$. We consider the three conditions from Definition~\ref{def:matched_pair}.
\begin{itemize}
\item The first condition is
\[
[\rho(\alpha), X]=\nabla^{T\BB}_\alpha X-\rho(\nabla_X\alpha)
\] 
for $\alpha\in \Gamma(A),\ X\in \mathfrak{X}(M)$. 
\item The second condition is
\[
\nabla^{T\BB}_\alpha[X,Y]=[\nabla^{T\BB}_\alpha X, Y]+[X, \nabla^{T\BB}_\alpha Y]+\nabla^{T\BB}_{\nabla_{Y}\alpha}X-\nabla^{T\BB}_{\nabla_X\alpha}Y
\]
for $\alpha\in \Gamma(A),\ X,Y\in \mathfrak{X}(\BB)$. 
\item Finally, the third condition is
\[
\nabla_X[\alpha, \beta]=[\nabla_X\alpha, \beta]+[\alpha, \nabla_X\beta]+\nabla_{\nabla^{T\BB}_{\beta}X}\alpha-\nabla_{\nabla^{T\BB}_\alpha X}\beta
\]
for $\alpha,\beta\in \Gamma(A),\ X\in \mathfrak{X}(\BB)$.
\end{itemize}
The fact that $\nabla$ is a representation of $T\BB$ on $A$ tells us precisely that it is a flat connection. The third condition above is the compatibility condition~\ref{Spencer_compatibility} between $\nabla$ and $[\ ,\ ]$. Hence, $(A, \nabla)$ is a flat Cartan algebroid. As for the second condition, one sees that it follows automatically from the first one, the fact that $\nabla$ is a representation and the Jacobi identity. The first condition is simply the formula of the representation on $T\BB$ associated with the flat Cartan algebroid $(A, \nabla)$. 

If viceversa we start with a flat Cartan algebroid $(A, \nabla)$, the first and third conditions above are satisfied chosing $\nabla^{T\BB}$ to be the representation of $A$ on $T\BB$. As for the second one, it is identically satisfied by the same reasons as before ($\nabla$ is a flat connection, the definition of $\nabla^{T\BB}$ and the Jacobi identity). 
\end{proof}

This discussion reveals the presence of a larger algebroid associated to flat Cartan algebroids:

\begin{defn}\label{Double_algebroid_def}
The {\bf double} of a flat Cartan algebroid $(A, \nabla)$, denoted $\mathcal{D}(A, \nabla)$, is the Lie algebroid corresponding to the matched pair $(A, T\BB)$. Explicitly, 
\[\mathcal{D}(A, \nabla)= A\oplus T\BB,\]
with anchor 
\[ \rho_{\mathcal{D}}(\alpha, X)= \rho_A(\alpha)+ X,\]
and with bracket 
% \[[(\alpha_1,X_1),(\alpha_2,X_2)]_{\mathcal{D}}=([\alpha_1, \alpha_2]_A+\nabla_{X_1}\alpha_2-\nabla_{X_2}\alpha_1,[X_1, X_2] +\nabla^{TM}_{\alpha_1}X_2-\nabla^{TM}_{\alpha_2}X_1).\]
\[
[(\alpha,X),(\beta, Y)]_{\mathcal{D}}=([\alpha, \beta]+\nabla_X\beta- \nabla_Y\alpha,[X, Y]+\nabla^{T\BB}_\alpha Y-\nabla^{T\BB}_\beta X)
\]
\end{defn}

With this, one obtains the following re-interpretation of the Haefliger cohomology:

\begin{prp}\label{prp:inf-Haefl-double} For any flat Cartan algebroid $(A, \nabla)$, its Haefliger cohomology is equal to the usual cohomology of the double algebroid $\mathcal{D}(A, \nabla)$:
\[ H^{*}_{\rm Haef}(A, \nabla)= H^*(\mathcal{D}(A, \nabla)).\]
\end{prp}

This holds true in full generality for matched pairs $(A_1, A_2)$; it is obtained by decomposing the algebroid complex of $A= A_1\oplus A_2$ 
\[ \Omega^k(A)= \bigoplus \Gamma(\Lambda^p A_1^*\otimes \Lambda^qA_2^*).\]
Actually, one has the following, which also shows that flat Cartan algebroids are precisely what is needed in order to define the Haefliger double complex (compare with the similar discussion for groupoids in Section $3$!).

\begin{thm}\label{Matched_pair_cohom}
Consider the quadruple $(A_1, \nabla^1, A_2, \nabla^2)$, where $\nabla^1$ is a representation of $A_1$ on $A_2$  and similarly for $\nabla^2$. Consider also the diagram
\[
\begin{tikzcd}
\Gamma(\Lambda^k (A_1)^*\otimes \Lambda^{q+1}(A_2)^*) \arrow[r]{}{d_{A_1}}& \Gamma(\Lambda^{k+1} (A_1)^*\otimes \Lambda^{q+1}(A_2)^*) \\
\Gamma(\Lambda^k (A_1)^*\otimes \Lambda^q(A_2)^*)\arrow[u]{}{d_{A_2}}\arrow[r]{}{d_{A_1}}&  \Gamma(\Lambda^{k+1} (A_1)^*\otimes \Lambda^q(A_2)^*)\arrow[u]{}{d_{A_2}}
\end{tikzcd}
\]
where $d_{A_1}$ and ${d_{A_2}}$ are the algebroid differentials associated to the induced representations on the exterior bundles.

The quadruple $(A_1, \nabla^1, A_2, \nabla^2)$ is a matched pair if and only if the diagram above is commutative for all $k, q \in \mathbb{N}$. Moreover, the resulting double complex is isomoprhic to the complex of algebroid cochains of $A_1\oplus A_2$.
\end{thm}
For a proof, see~\cite{GSX}.

\begin{rmk}\label{Isotropy_double_algebroid} 
The double algebroid $\mathcal{D}(A, \nabla)$ associated to a flat Cartan algebroid clarifies also 
the extended isotropy Lie algebras from Definition \ref{defn:extended-isotropy}: it is isomorphic to the isotropy Lie algebra (i.e. the kernel of the anchor) of $\mathcal{D}(A, \nabla)$; the isomorphism is simply $\alpha_x\to (\alpha_x, -\rho(\alpha_x))$. %\footnote{{\color{red} fix the formula; and change signs if we change the signs in the pointwise bracket}}. 
The remark from Example \ref{exam-Cart-algbrds-actions} that all the fiberwise Lie algebras are isomorphic can also be seen as a particular case of the fact that the isotropy Lie algebras of any transitive Lie algebroid are isomorphic to each other. 
\end{rmk}

%%%%%%%%%%%%%%%%%%%%%%%
%%%%%%%%%%%%%%%%%%%%%%%
%%%%%%%%%%%%%%%%%%%%%%%
%%%%%%%%%%%%%%%%%%%%%%%
%%%%%%%%%%%%%%%%%%%%%%%
\subsection{Flat $(\Sigma, \CC)$-structures}
%%%%%%%%%%%%%%%%%%%%%%%
%%%%%%%%%%%%%%%%%%%%%%%
%%%%%%%%%%%%%%%%%%%%%%%
%%%%%%%%%%%%%%%%%%%%%%%
%%%%%%%%%%%%%%%%%%%%%%%

In subsection \ref{ssec:Formal geometric structures} we discussed the notions of principal $(\Sigma, \CC)$-bundle and $(\Sigma, \CC)$-structure, where $(\Sigma, \CC)\tto \BB$ is a flat Cartan groupoid, as a general framework for ``almost geometric structures''. We also pointed out that flat principal $(\Sigma, \CC)$-bundles are actually the same as flat $(\Sigma, \CC)$-structures, and that examples coming from geometry are always flat. As we have seen, $(\Sigma, \CC)$ comes with its associated Cartan algebroid $(A, \nabla)$. Here we reformulate the definition of flat principal $(\Sigma, \CC)$-bundle/$(\Sigma, \CC)$-structure using only the infinitesimal data $(A, \nabla)$. 

\begin{prp}\label{prp:MC-eq-CP}
Let $(\Sigma, \CC)\tto \BB$ be a Cartan groupoid, with associated Cartan algebroid $(A, \nabla)$. Assume also that $\pi: P\to M$ is a principal $\Sigma$-bundle, 
\[
\xymatrix{
\Sigma \ar@<0.25pc>[d] \ar@<-0.25pc>[d]  & \ar@(dl, ul) &  P \ar[dll]^{\mu}\ar[dr]_{\pi} &    \\
\BB&  & &  M},
\]
and denote by $a: \mu^*A\to TP$ the induced infinitesimal action. 
Then any sub-bundle $\CC_P\subset TP$ making $P$ into a flat $(\Sigma, \CC)$-structure arises as the kernel of a 1-form 
\[ \theta\in \Omega^1(P, \mu^*A)\]
with the following two properties:
\begin{equation}\label{eq:thata-a-is-apha} 
\theta(a(\alpha))= \alpha \quad \textrm{for all $\alpha\in \Gamma(A)$},
\end{equation}
\[ d_{\mu^*A}(\theta)+ \frac{1}{2}\{\theta, \theta\}_{\textrm{pt}}= 0\]
where, as in Corollary \ref{cor:MC-eq-for-Cartan-form}, $\{\cdot, \cdot\}_{\textrm{pt}}$ is the pointwise bracket (\ref{pt-bracket}) 
% \footnote{{\color{red}  AND HAVING THE NICE MAURER-CARTAN WAS ANOTHER REASON TO CHANGE THE SIGN IN THE POINTWISE BRACKET}}
and $d_{\mu^*A}$ is DeRham differential on $P$ with coefficients in the flat pull-back bundle $\mu^*A$. 

Furthermore, if the s-fibers of $\Sigma$ are connected, this provides a 1-1 correspondence between such sub-bundles and 1-forms with these two properties. 
\end{prp}

\begin{proof} We denote $\widetilde{A}:= \mu^*A$, by $\widetilde{\rho}$ its anchor, by $\widetilde{\nabla}$ the pull-back of the connection $\nabla$ (making $\mu^*A$ itself a flat Cartan algebroid). 

First of all we will need an infinitesimal characterization of multiplicativity - characterization that was worked out in greater generality in Proposition 5.3.4 from \cite{FRANCESCO}. Here is the part that is of interest for us. 
The idea is that the multiplicativity property (\ref{Mult_action}) of $\theta$, rewritten as 
\begin{equation}\label{Mult_action-2}
(a^*\theta)_{(g, p)}- (g\cdot pr_2^*\theta)_{(g, p)}=(pr_1^*\theta)_{(g, p)},
\end{equation}
is an equality of multiplicative forms on the action groupoid $\Sigma\ltimes P$. Hence one can just look at the equality of the corresponding infinitesimal counterparts, 
\[ \nabla^{\textrm{left}}, \  \nabla^{\textrm{right}}: \mathfrak{X}(P)\times \Gamma(\tilde{A})\to \Gamma(\tilde{A}),\]
for the left and the right hand side of (\ref{Mult_action-2}), respectively. Computing the two operators we find
\[ \nabla^{\textrm{left}}_{X}(\beta)= \widetilde{\nabla}_{\widetilde{\rho}(\theta(X))}(\beta)+ [\beta, \theta(X)]_{\widetilde{A}}- \theta([\widetilde{\rho}(\beta), X]), \]
\[ \nabla^{\textrm{right}}_{X}(\beta)= \widetilde{\nabla}_X(\beta) \]
for all $X\in \mathfrak{X}(P)$, $\beta\in \Gamma(\widetilde{A})$ (see \cite{FRANCESCO} for more details). While the second satisfies Leibniz, the first one does if and only if~(\ref{eq:thata-a-is-apha}) is satisfied.

Next, we concentrate on the Maurer-Cartan expression 
\[ MC:= d_{\mu^*A}(\theta)+ \frac{1}{2}\{\theta, \theta\}_{\textrm{pt}}\in \Omega^2(P, \widetilde{A}),\]
and we compute it on several types of tangent vectors:
\begin{itemize}
\item for  $X, Y$ in the image of $\widetilde{\rho}$, we find immediately that $MC(X, Y)= 0$;
%
% Proceed as follows. Use that \theta\circ \widetilde{\rho}=id. Use that the differential term is C^\infty(M) linear, hence depends only on the value at a point. But then, you may assume your sections of $\widetilde{A}$ to be pullbacks of sections of A: at each point p the value depends only on \sigma(p) and \sigma'(p), which can be realized by pullbacks. On pullbacks, \widetilde{\nabla} behaves as \nabla and the bracket is the bracket of A.
%
\item if only $Y$ is in the image of $\widetilde{\rho}$ and $X\in \ker(\theta)$, writing $Y= \widetilde{\beta}$, a careful but simple computation 
gives 
\[ MC(X, \widetilde{\beta})=  \nabla^{\textrm{left}}_{X}(\beta)- \nabla^{\textrm{right}}_{X}(\beta);\]
\item $X, Y\in \textrm{Ker}(\theta)$ we find right away $MC(X, Y)= - \theta([X, Y])$.
\end{itemize}
Hence we see that $MC= 0$ encodes both infinitesimal multiplicativity as well as the involutivity of $\CC_P= \textrm{Ker}(\theta)$. 
\end{proof}

\begin{rmk} As we see in the proof, $MC= 0$ encodes more than just the infinitesimal characterization of multiplicativity of $\theta$. From that point of view, the condition $\nabla^{\textrm{left}}= \nabla^{\textrm{right}}$ is a 
better characterization. However, the minimal way to encode multiplicativity infinitesimally is to remove all the conditions that hold anyway. We see that what is left is $MC(X, \widetilde{\beta})= 0$ whenever $X$ is killed by $\theta$, i.e.:
 \[ \widetilde{\nabla}_X(\beta)= \theta([X, \widetilde{\rho}(\beta)])\quad \textrm{for all $X\in \Gamma(\CC_P)$, $\beta\in \Gamma(\widetilde{A})$}.\]
 Also this condition alone implies (\ref{eq:thata-a-is-apha}) and then the decomposition
 \[ TP % = T^{\textrm{vert}}P\oplus \CC_P
 = \textrm{Im}(a)\oplus \CC_P\]
 (where $a$ is the infinitesimal action). 
 On the other hand one can further rewrite the last equation on $\theta$ without any reference to $\theta$, by applying it to horizontal lifts and to $\beta= \mu^*\alpha$. We find that the infinitesimal counterpart of the multiplicativity of $\CC_P$ is
 \[ \left[ \textrm{hor}(V), a(\alpha)\right ]= a\left(\nabla_V(\alpha)\right)\oplus \textrm{hor}\left([V, \rho(\alpha)]\right)\quad \textrm{for all $V\in \mathfrak{X}(\BB)$, $\alpha\in \Gamma(A)$.}\]
\end{rmk}

%{\color{blue}
\begin{rmk}\label{rk:theta_tilde_MC}
 The form $\theta$ and its properties can be packed together sightly differently, using the corresponding double $\mathcal{D}= \mathcal{D}(A, \nabla)$. First of all, we re-interpret $\theta$ as a $\mathcal{D}$-valued form, where, in principle, we use $d\mu$ for the $T\BB$-component. However, to obtain compatibility with anchors, we arrange the terms to be:
\[ \tilde{\theta}:= (d\mu- \rho\circ \theta, \theta): TP\to \mathcal{D}.\]
Then the Maurer-Cartan equation for $\theta$ translates into a similar Maurer-Cartan equation for $\tilde{\theta}$ which, in turn, just encodes the fact that $\tilde{\theta}$ is a morphism of Lie algebroids. See also the proof of Proposition~\ref{prp:Van_Est_to_double}.
\end{rmk}

\begin{rmk}
The previous remark is undoubtedly related to the discussion in~\cite{ORI} concerning Cartan's seminal work on Lie pseudogroups and more precisely his ``three fundamental theorems''. A precise exposition of the material is out of the scope of this paper, but the point that we want to stress is that the Maurer-Cartan equation for $(\Sigma, \CC)$-structures mentioned in the remark above
% is a particular case of the {\bf structure equations} described in~\cite{ORI}, and 
proves, together with the flat Cartan algebroid structure on $A$, that a flat Cartan groupoid  $(\Sigma, \CC)$
%$(\Sigma, \CC)$-structure $(P, \theta)$ 
is a very special case of {\bf Cartan's realization} in the sense of~\cite{ORI}, Definition $5.2.11$. One considers the automorphism of vector bundles 
\[
\Phi_\D: \D\to \D, \quad (X, \alpha)\to (X+\rho(\alpha), \alpha)
\]
There is a unique bracket $[[\ ,\ ]]$ such that the above map is an isomorphism of Lie algebroids
\[
\Phi_\D: (D, [\ ,\ ])\to (\D, [[\ ,\ ]]).
\]
We can look at the flat $(\Sigma, \CC)$-structure given by the action of $(\Sigma, \CC)$ on itself and consider the form $\tilde{\omega}=(dt-\rho\circ \omega, \omega)$ as in the previous remark, where now $\omega$ is the form dual to $\CC$.
It follows from~\cite{ORI}, Proposition $5.3.7$, that the form $\Omega=\Phi_\D\circ \tilde{\omega}=(dt, \omega)$ satisfies {\bf Cartan's structure equation} 
\[
d\Omega+\frac{1}{2}[[\Omega, \Omega]]=0
\]
which is manifestly of Maurer-Cartan type. Getting rid of $\Phi_\D$, one finds the Maurer-Cartan equation for $\tilde{\omega}$ mentioned in the previous remark. On the other hand, when $(P, \CC_P)$ is a flat $(\Sigma, \CC)$-structure and $\theta$ is the form dual to $\CC_P$, the Maurer-Cartan equation for $\tilde{\theta}=(dt-\rho\circ \theta, \theta)$ can be read, composing with $\Phi_\D$, as a Cartan's structure equation in the sense of~\cite{ORI}, Definition $5.2.11$.
\end{rmk}

\section{Van Est maps}\label{Van Est maps}
%%%%%%%%%%%%%%%%%%%%%%%%%%%%%%%%%%%%%%%%
%%%%%%%%%%%%%%%%%%%%%%%%%%%%%%%%%%%%%%%%
%%%%%%%%%%%%%%%%%%%%%%%%%%%%%%%%%%%%%%%%
%%%%%%%%%%%%%%%%%%%%%%%%%%%%%%%%%%%%%%%%
%%%%%%%%%%%%%%%%%%%%%%%%%%%%%%%%%%%%%%%%
%%%%%%%%%%%%%%%%%%%%%%%%%%%%%%%%%%%%%%%%
%%%%%%%%%%%%%%%%%%%%%%%%%%%%%%%%%%%%%%%%
%%%%%%%%%%%%%%%%%%%%%%%%%%%%%%%%%%%%%%%%
%%%%%%%%%%%%%%%%%%%%%%%%%%%%%%%%%%%%%%%%
%%%%%%%%%%%%%%%%%%%%%%%%%%%%%%%%%%%%%%%%

%%%%%%%%%%%%%%%%%%%%%%%%%%%%%%%%%%%%%%%%
%%%%%%%%%%%%%%%%%%%%%%%%%%%%%%%%%%%%%%%%
%%%%%%%%%%%%%%%%%%%%%%%%%%%%%%%%%%%%%%%%
%%%%%%%%%%%%%%%%%%%%%%%%%%%%%%%%%%%%%%%%
\subsection{A very general setting}\label{A very general setting}
%%%%%%%%%%%%%%%%%%%%%%%%%%%%%%%%%%%%%%%%
%%%%%%%%%%%%%%%%%%%%%%%%%%%%%%%%%%%%%%%%
%%%%%%%%%%%%%%%%%%%%%%%%%%%%%%%%%%%%%%%%
%%%%%%%%%%%%%%%%%%%%%%%%%%%%%%%%%%%%%%%%

In this subsection, given a proper action (not necessarily principal) of a flat Cartan groupoid 
$(\Sigma, \CC)\tto B$ on a space $P$ (not necessarily carrying extra-structure) 
%(possibly in the profinite dimensional sense), 
\[
\xymatrix{
(\Sigma, \CC) \ar@<0.25pc>[d] \ar@<-0.25pc>[d]  & \ar@(dl, ul) &  P \ar[dll]^{\mu}    \\
\BB&  & },
\]
we will construct a natural map 
\begin{equation}\label{eq:general:VE} 
VE_{P}: H^*_{\rm Haef}(\Sigma, \CC)\to 
H^*_{\Sigma}(P)
\end{equation}
% H^\bullet_{\Sigma-\textrm{inv}}(P)\]
from the Haefliger cohomology of the Cartan groupoid to the cohomology corresponding to the subcomplex 
\begin{equation}\label{eq:incl-inv-forms-in-forms} 
\Omega^{*}(P)^{\Sigma}\subset \Omega^{*}(P)
\end{equation}
of DeRham complex of $P$ consisting of $\Sigma$-invariant differential forms on $P$. 

Here we are using that $\Sigma$ acts not only on $TM$ but also on the tangent space $TP$ of any $\Sigma$-space $P$. Indeed, the action groupoid is itself a flat Cartan groupoid with the Cartan distribution pulled back from $\Sigma$: 
\[ \Sigma_{P}:= \Sigma\ltimes P\tto P, \quad \mathcal{C}_{P}:= \textrm{pr}_{1}^{-1}\CC,\]
and we are just considering the induced action on the tangent space of the base. Working out the action one finds the explicit description: the action of $g: x\to y$ on $v_p\in T_p P$ ($p\in \mu^{-1}(x)$) is
\[ g\cdot v_p:= (da)_{g, p}({\rm hor}_g(d\mu(v_p)), v_p)\in T_{gp}P.\]
The first projection from the action groupoid gives rise to a pull-back map between the corresponding Haefliger cohomologies (defined already at the level of complexes)
\[ \textrm{pr}_{1}^{*}: H^*_{\rm Haef}(\Sigma, \CC)\to H^*_{\rm Haef}(\Sigma_P, \CC_P).\]

\begin{lem}\label{lem:proper-action-van-est} If the action of $\Sigma$ on $P$ is proper then the inclusion~\eqref{eq:incl-inv-forms-in-forms} of 
$\Omega^*(P)^{\Sigma}$ in the $p= 0$ row of the Haefliger complex of $(\Sigma_P, \CC_{P})$ induces an isomorphism:
%\[ \Omega^\bullet(P)^{\Sigma}\subset \Omega^{\bullet}(P)= C^{0, \bullet}_{\textrm{Hae}}(\Sigma_P, \CC_{P})\]
%induces an isomorphism 
\[ H^*_{\Sigma}(P)\cong H^*_{\rm Haef}(\Sigma_P, \CC_P).\]
\end{lem}

\begin{proof} Replacing $\Sigma_{P}$ by $\Sigma$, we may assume that $\Sigma\tto X$ is a proper groupoid and 
we are comparing $C^{*}_{\textrm{Haef}}(\Sigma)$ with $\Omega^{*}(X)^{\Sigma}$. The main point is that
each row $C^{*, q}_{\textrm{Haef}}(\Sigma)$ computes the differentiable cohomology of the proper groupoid $\Sigma$ with various coefficients. But that is know to vanish in all positive degree, and give the $\Sigma$-invariant part in degree zero, if $\Sigma$ is proper (see Lemma \ref{lem-App;proper} in the Appendix). Hence, by a standard double complex argument, the conclusion follows. 
\end{proof}

Combining the two maps above, i.e. $\textrm{pr}_{1}^{*}$ and the isomorphism from the Lemma, we obtain the desired map (\ref{eq:general:VE}).

\begin{prp}\label{prp-general-VE} If the action of $\Sigma$ on $P$ is proper, $\mu$ is a submersion and the $\mu$-fibers of $P$ are cohomologically $l$-connected then $VE_{P}: H^*_{\rm Haef}(\Sigma, \CC)\to H^*_{\Sigma}(P)$ is an isomorphism up to degree $l$ and injective in degree $l+1$.
\end{prp} 

\begin{proof} 
Let us call the property from the statement ($\textrm{iso}_{l}$). 
We have to check ($\textrm{iso}_{l}$) for the map induced in cohomology by map of double complexes
\[ \textrm{pr}_{1}^{*}: C^{*, *}_{\rm Haef}(\Sigma, \CC)\to C^{*, *}_{\rm Haef}(\Sigma_P, \CC_P).\]

If $A\hookrightarrow B$ is any inclusion of double complexes and $C$ denotes the quotient, from the long exact sequence in cohomology we see that ($\textrm{iso}_{l}$) for the map in cohomology is equivalent to $C$ having trivial cohomology up to degree $k$. Furthermore if such a property holds for all columns, then it also holds for the total complex. 

By this double complex argument, it suffices to show that, in each column $p$, our map
% \[ C^{p, \bullet}_{Hae}(\Sigma, \CC)\to C^{p, \bullet}_{Hae}(\Sigma_P, \CC_P)\]
\[ \textrm{pr}_{1}^{*}:  \Gamma(\Sigma^{(p)}, \Lambda^{*} \CC_{p}^{*}) \to \Gamma(\Sigma_P^{(p)}, \Lambda^{*} \widetilde{\CC}_{P}^{*})\]
satisfies ($\textrm{iso}_{l}$) in cohomology. Fixing $p$ and denoting 
\[ \iP:= \Sigma_{P}^{(p)}, \quad \iM= \Sigma^{(p)}, \quad \pi= \textrm{pr}_1\]
we find ourselves precisely in the setting of Proposition \ref{lem:App:2nd} from the Appendix.
\end{proof}

%%%%%%%%%%%%%%%%%%%%%%%%%%%%%%%%%%%%%%%%
%%%%%%%%%%%%%%%%%%%%%%%%%%%%%%%%%%%%%%%%
%%%%%%%%%%%%%%%%%%%%%%%%%%%%%%%%%%%%%%%%
%%%%%%%%%%%%%%%%%%%%%%%%%%%%%%%%%%%%%%%%
\subsection{The case of flat $(\Sigma, \CC)$-structures (and almost structures)}
\label{subsec:The case of Cartan bundles and almost structures}
%%%%%%%%%%%%%%%%%%%%%%%%%%%%%%%%%%%%%%%%
%%%%%%%%%%%%%%%%%%%%%%%%%%%%%%%%%%%%%%%%
%%%%%%%%%%%%%%%%%%%%%%%%%%%%%%%%%%%%%%%%
%%%%%%%%%%%%%%%%%%%%%%%%%%%%%%%%%%%%%%%%

The differentiable cohomology of pseudogroups serves as the domain of differentiable characteristic maps~\eqref{eq:dif-charact-map}
\[ \kappa^{\mathcal{P}}_{\textrm{diff}}: H^{*}_{\textrm{diff}}(\Gamma)\to H^*(M)\]
associated to $\Gamma$-structures $\mathcal{P}$ - see Definition \ref{defn:diff-coh-gamma}. Here we point out that these maps do not depend on $\mathcal{P}$ but just on the induced almost $\Gamma$-structure. Actually, these maps are defined in the general setting of flat Cartan groupoids $(\Sigma, \CC)\tto \BB$ and flat $(\Sigma, \CC)$-structures $(P, \CC_P)\to M$.
This exploits the general van Est map (\ref{eq:general:VE}) in the more special case when $P$ is actually a flat $(\Sigma, \CC)$-structure. In this case we have available the horizontal lift $\textrm{hor}$ with respect to the Ehresmann connection $\CC_P$, which induces a map
\[ \textrm{hor}^*:  \Omega^*(P)^{\Sigma}\to \Omega^*(M),\]
\[\textrm{hor}^*(\alpha)(v^1_x, \ldots, v^q_x):= \alpha(\textrm{hor}_p(v^1_x), \ldots, \textrm{hor}_p(v^q_x)),\]
where $p\in P$ is a/any element in the fiber above $x$; thanks to the 
%equivariance 
invariance of $\alpha$, the definition is independent of the choice of $p$. Composing the induced map in cohomology with the general van Est map (\ref{eq:general:VE}) gives rise to a map 
\begin{equation}\label{eq:Haefl-charact-map} 
\kappa^{P}_{\textrm{Haef}}: H^{*}_{\textrm{Haef}}(\Sigma, \CC)\to H^*(M),
\end{equation}
called {\bf the Haefliger characteristic map associated to the $(\Sigma, \CC)$-structure $(P, \CC_P)$}. As promised, we have:

\begin{prp}\label{prp:chr-map-formal-back} For Cartan groupoids $(\Sigma, \CC)$ and $(\Sigma, \CC)$-structures induced by Lie pseudogroups $\Gamma$ and $\Gamma$-structures $\mathcal{P}$
% Cartan bundles induced by pseudogroups $\Gamma$ and $\Gamma$-structures $\mathcal{P}$:
\[ \Sigma= J^{\infty}\Gamma, \quad P= J^{\infty}\mathcal{P},\]
the Haefliger characteristic map \eqref{eq:Haefl-charact-map} becomes the differentiable characteristic map \eqref{eq:dif-charact-map}  of the $\Gamma$-structure. 
\end{prp}

\begin{proof}
%Sticking to the notation introduced in this subsection, we will write $k^{\textrm{Haef}}_P$ for the map~\eqref{eq:Haefl-charact-map} and $k^{\mathcal{P}}_{\rm diff}$ for the map~\eqref{eq:dif-charact-map}. 
%We want to show that $k^{Haef}_P$ 
We start by recalling that, by definition, $\kappa^{\mathcal{P}}_{\rm diff}$ is constructed by composing the map
\[
j^*: H^*_{\text{diff}}(\Gamma)\to H^*_{\rm dR}(\mathcal{G})\quad (\mathcal{G}=Germ(\Gamma))
\]
with the map~\eqref{eq:abstr-k-G-dR}
\[
\gamma^*: H^*_{\rm dR}(\mathcal{G})\to H^*(M).
\]
We have the map of double complexes (we stick to the notation $\mathcal{G}=Germ(\Gamma)$)
\[
{\rm pr}_1^*:C^p(\mathcal{G}, \Omega^q_\BB)\to C^p(\mathcal{G}\ltimes \mathcal{P}, \Omega^q_\mathcal{P})
\]
Using properness of the principal action, the cohomology of the right hand side is seen to be isomorphic to $H^*_{\mathcal{G}}(\mathcal{P})$ (the cohomology of $\mathcal{G}$-invariant forms on $\mathcal{P}$; this goes as in Lemma~\ref{lem:proper-action-van-est}). The projection $\pi: \mathcal{P} \to M$ is étale, because so is $\mathcal{G}$ and $\mathcal{P}$ is a principal $\mathcal{G}$-bundle, hence we see that $H^*_{\mathcal{G}}(\mathcal{P})\cong H^*_{}(X)$. We get a map
\begin{equation}\label{dr_char_hat}
\hat{\gamma}^*: H^*_{\rm dR}(\mathcal{G})\to H^*_{}(X)
\end{equation}
fitting in the diagram
\begin{equation}\label{diagram:our-map}
\begin{tikzcd}
H^*_{\rm diff}(\Gamma)\arrow[rr]{}{\kappa_{\textrm{Haef}}^P}\arrow[rd]{}{j^*}& & H^*_{}(X) \\
&H^*_{\rm dR}(\mathcal{G})\arrow[ru]{}{\hat{\gamma}^*}&  
\end{tikzcd}
\end{equation}
which is commutative by construction. Of course, $\hat{\gamma}^*$ has to be compared with $\gamma^*$ and~\eqref{diagram:our-map} has to be compared with
\begin{equation}\label{diagram:Hae-map}
\begin{tikzcd}
H^*_{\rm diff}(\Gamma)\arrow[rr]{}{\kappa^{\mathcal{P}}_{\rm diff}}\arrow[rd]{}{j^*}& & H^*_{}(X) \\
&H^*_{\rm dR}(\mathcal{G})\arrow[ru]{}{\gamma^*}&  
\end{tikzcd}.
\end{equation}
Hence, it is clear that the equality $\gamma^*=\hat{\gamma}^*$ implies our claim. To prove it, let us fix a cover $\mathcal{U}$ of $M$ 
%and, for each $U_i\in \mathcal{U}$ a section $\sigma_i:U_i\to \mathcal{P}$ defining a trivialization of the principal $\mathcal{G}$-bundle $\mathcal{P}$ (in particular, the $\sigma_i$'s are local homeomorphisms). 
We consider the map
\begin{align*}
\bar{\gamma}^*: C^p(\mathcal{G}\ltimes \mathcal{P}, \Omega^q_\mathcal{P})\to C^p(M_\mathcal{U}, \Omega^q_M),\quad c \to \bar{\gamma}^*(c)
\end{align*}
where 
\[
\bar{\gamma}^*(c)_{i_0, \dots, i_p}(x)=\sigma_{i_p}^*c((\gamma_{i_{p-1}i_{p}}(x), [\sigma_{i_{p-1}}]_x), \dots, (\gamma_{i_0i_1}(x), [\sigma_{i_0}]_x)).
\]
Here the $\sigma_i$'s are the sections of the principal $\mathcal{G}$-bundle $\mathcal{P}$ which correspond to the chosen $\mathcal{G}$-cocycle on $M$, and the $[\sigma_i]_x$'s are the corresponding germs at $x$ (which can be identified with the points $\sigma_i(x)\in \mathcal{P}$ since $\pi$ is étale).
More explicitely, $\sigma_i:U_i\to \mathcal{P}$ is the section of $\mathcal{P}$ which is the inverse of $\pi$ around $[\gamma_{ii}]_x\in \mathcal{P}$, the germ of $\gamma_{ii}$ at $x$ (in particular, $[\gamma_{ii}]_x=\sigma_i(x)$). 
This is a map of double complexes. The claim follows thanks to the commutative diagram below (we leave to the reader to check that $\bar{\gamma}$ induces in cohomology the isomorphism $H^*_{\mathcal{G}}(\mathcal{P})\cong H^*(X)$ used to define $\hat{\gamma}^*$):
\[
\begin{tikzcd}
C^p(\mathcal{G}\ltimes \mathcal{P}, \Omega^q_\mathcal{P})\arrow[rr]{}{\bar{\gamma}^*}& & C^p(M_\mathcal{U}, \Omega^q_X) \\
&C^p(\mathcal{G}, \Omega^q_\BB)\arrow[ru]{}{\gamma^*} \arrow[lu,swap]{}{{\rm pr_1}^*}&  
\end{tikzcd}.
\]
%where $\gamma^*$ is given by formula~\eqref{dR_gamma}.
\end{proof}

\begin{rmk} \label{int} 
This new insight may make the characteristic maps useful also in detecting almost structures that are not integrable. 
The plan would be to use the commutative diagram 
\[
\begin{tikzcd}
H^*_{\rm diff}(\Gamma)\arrow[rr]{}{\kappa^{\mathcal{P}}_{\rm diff}} \arrow[rd, swap]{}{\kappa^{\textrm{univ}}}& & H^*(M)\\
& H^*(\Gamma)\arrow[ru, swap]{}{\kappa^{\mathcal{P}}_{\rm abs}} & 
\end{tikzcd}
\]
which is simply~\eqref{eq: Haefliger-diagram} for a general pseudogroup $\Gamma$ (notice that we see $\kappa^{\mathcal{P}}_{abs}$ as defined on $H^*(\Gamma)$ using Theorem~\ref{thm-Haefl-conj}).
Step 1 is to detect classes 
\[ u\in H^*_{\rm diff}(\Gamma)\]
that are killed when viewed as cocycles in $H^*(B\Gamma)$, i.e. are in the kernel of the universal map 
\[ \kappa^{\textrm{univ}}: H^*_{\rm diff}(\Gamma)\to H^*(\Gamma)\cong H^*(B\Gamma).\]
Then step 2 is to exhibit almost $\Gamma$-structures $P$ whose corresponding $u$-characteristic class
\[ u(P)= \kappa_{\rm diff}^{\mathcal{P}}(u)\in H^*(M)\]
does not vanish. From the previous commutative diagram one concludes immediately that $P$ cannot arise from an actual $\Gamma$-structure $\mathcal{P}$. 

We conjecture that this plan can be implemented for certain pseudogroups $\Gamma$. Note however that already step 1 is a delicate matter. For instance, even for $\Gamma= \Gamma^q$ (when the previous plan cannot work because of the Frobenius theorem) one may expect that the universal characteristic map is injective but, despite some early announcement~\cite{FUCHSWRONG}, that is still an open problem (this was pointed out in \cite{CONMOSC} - see e.g. the comments following Theorem $5$, and the further references therein). On the other hand, it may be easier to find elements in the kernel of $\kappa^{\textrm{univ}}$ than proving that the map is injective; of course, one has to look at pseudogroups for which formal integrability does not imply integrability. 
\end{rmk}

%%%%%%%%%%%%%%%%%%%%%%%%%%%%%%%%%%%%%%%%
%%%%%%%%%%%%%%%%%%%%%%%%%%%%%%%%%%%%%%%%
%%%%%%%%%%%%%%%%%%%%%%%%%%%%%%%%%%%%%%%%
%%%%%%%%%%%%%%%%%%%%%%%%%%%%%%%%%%%%%%%%
\subsection{The van Est map into infinitesimal Haefliger cohomology}
%%%%%%%%%%%%%%%%%%%%%%%%%%%%%%%%%%%%%%%%
%%%%%%%%%%%%%%%%%%%%%%%%%%%%%%%%%%%%%%%%
%%%%%%%%%%%%%%%%%%%%%%%%%%%%%%%%%%%%%%%%
%%%%%%%%%%%%%%%%%%%%%%%%%%%%%%%%%%%%%%%%

Next, we exploit the general van Est map (\ref{eq:general:VE}) in another special case: when $P$ is $\Sigma$ itself.

\begin{prp}\label{prp:Van_Est_to_double}
For a Cartan groupoid $(\Sigma, \CC)\tto \BB$, taking $P= \Sigma$ endowed with the left action of 
$\Sigma$, the resulting complex $\Omega^{*}(\Sigma)^{\Sigma}$ is isomorphic to the 
infinitesimal Haefliger complex $C^{*}_{\rm Haef}(A, \nabla)$ of 
the flat Cartan algebroid $(A, \nabla)\to \BB$ of $(\Sigma, \CC)$. 

Therefore, one obtains a canonical map
\[ VE: H^{*}_{\rm Haef}(\Sigma, \CC)\to H^{*}_{\rm Haef}(A, \nabla)\]
and this map is an isomorphism up to degree $l$ if the $s$-fibers of $\Sigma$ are cohomologically $l$-connected. 
\end{prp}

\begin{proof} 
%\footnote{{\color{red} TO BE ADAPTED TO THE NEW VERSION OF THE PAPER!}}
We will be using the description of the infinitesimal Haefliger complex as the algebroid complex of the double $\mathcal{D}(A, \nabla)$ (see Definition~\ref{Double_algebroid_def}), appealing to Proposition \ref{prp:inf-Haefl-double}. 
Recall that the total space of $\mathcal{D}(A, \nabla)$ is $A\oplus T\BB$.
We have (identifying $A$ with $\ker(ds)|_M$ via $\omega$)
\[
A_{t(g)}\oplus T_{t(g)}\BB=T_g\Sigma
\]
and we notice that a vector field left invariant for the action of $\Sigma\ltimes \Sigma$ on $T\Sigma$ is completely determined by its value over $u(\BB)\subset \Sigma$. Moreover, as a representation of $\Sigma\ltimes \Sigma$, $T\Sigma$ is the direct sum representation of $t^*A$ and $t^*T\BB$; if $(\alpha, X)\in \Gamma( A\oplus T\BB)$ then
\[
V_g=(dm({\rm hor}^\mathcal{C}_g(dt(\alpha)), \alpha),dm({\rm hor}^\mathcal{C}_g(dt(X)), X))
\]
defines a left invariant vector field $V\in \mathfrak{X}(\Sigma)$. This gives a bijection between sections of $A\oplus T\BB$ and such invariant vector fields. Furthermore
\[
V_g=(\hat{\alpha}_g, {\rm hor}^{\mathcal{C}}_g(X))
\]
where $\hat{\alpha}$ and ${\rm hor}^{\mathcal{C}}_g(X)$ are the left invariant sections of $t^*A$ and $t^*T\BB$ associated to $\alpha\in \Gamma(A)$ and $X\in \Gamma(T\BB)$. In particular, $\hat{\alpha}$ is $s$-vertical and ${\rm hor}^{\mathcal{C}}(X)$ is the unique $s$-projectable vector field tangent to $\mathcal{C}$ and extending $X$ at $u(\BB)$. Notice that the map 
\begin{align*}
\Phi: A\oplus T\BB&\to  A\oplus T\BB\\
(\alpha_x, v_x)&\to (-\alpha_x, v_x+\rho(\alpha_x))
\end{align*}
is an isomorphism of vector bundles. 
We will show that the composition of $\Phi$ with the map $(\alpha, X)\to (\hat{\alpha}, {\rm hor}^{\mathcal{C}}(X))$ 
gives an isomorphism of Lie algebras. 

First of all, for all $\alpha, \beta\in \Gamma(A)$ and $X, Y\in \mathfrak{X}(\BB)$ (notice that $\rho(\hat{\alpha})={\rm hor}^{\mathcal{C}}(\rho(\alpha))$, and the same holds for $\beta$),
\begin{align*}
[(-\hat{\alpha}, {\rm hor}^{\mathcal{C}}(X)+\rho(\hat{\alpha})), (-\hat{\beta}, {\rm hor}^{\mathcal{C}}(Y)+\rho(\hat{\beta}))]=[-\hat{\alpha}, -\hat{\beta}]+[-\hat{\alpha}, {\rm hor}^{\mathcal{C}}(Y)]+[-\hat{\alpha}, \rho(\hat{\beta})]+[{\rm hor}^{\mathcal{C}}(X), -\hat{\beta}]\\
+[{\rm hor}^{\mathcal{C}}(X), {\rm hor}^{\mathcal{C}}(Y)]+[{\rm hor}^{\mathcal{C}}(X), \rho(\hat{\beta})]+[\rho(\hat{\alpha}), -\hat{\beta}]+[\rho(\hat{\alpha}),{\rm hor}^{\mathcal{C}}(Y)]+[\rho(\hat{\alpha}), \rho(\hat{\beta})].
\end{align*}
Since the right hand side of the equality above is left invariant, we just have to show that it coincide with $\Phi([(\alpha, X), (\beta, Y)])$ when restricted to $\BB$. Following our argument for Lemma~\ref{Gelf_Fuc} below, one sees that
\[
[\hat{\alpha}, \hat{\beta}]|_\BB=-[\alpha, \beta]+\nabla_{\rho({\alpha})}( {\beta})-\nabla_{\rho({\beta})}(\alpha)
\]
Notice also that, for any $Z\in \mathfrak{X}(\BB)$, ${\rm hor}^{\mathcal{C}}(Z)$ is tangent to $\BB$. This implies that
\[
[-\hat{\alpha}, {\rm hor}^{\mathcal{C}}(Y)]|_\BB=\nabla_Y(\alpha),\ \ [{\rm hor}^{\mathcal{C}}(X), -\hat{\beta}]=-\nabla_X(\beta)
\]
and that 
\[
[-\hat{\alpha}, \rho(\hat{\beta})]|_\BB=\nabla_{\rho(\beta)}(\alpha),\ \ [\rho(\hat{\alpha}), -\hat{\beta}]|_\BB=-\nabla_{\rho(\alpha)}(\beta)
\]
In conclusion, 
\begin{align*}
[(-\hat{\alpha}, {\rm hor}^{\mathcal{C}}(X)+\rho(\hat{\alpha})), (-\hat{\beta}, {\rm hor}^{\mathcal{C}}(Y)+\rho(\hat{\beta}))]|_\BB=\\(-[\alpha, \beta]+\nabla_Y(\alpha)-\nabla_X(\beta), [X,Y] +[\rho(\alpha), Y]+[X, \rho(\beta)]+[\rho(\alpha), \rho(\beta)])
\end{align*}
that coincides with $\Phi([(\alpha, X), (\beta, Y)])$.

\end{proof}

%%%%%%%%%%%%%%%%%%%%%%%%%%%%%%%%%%%%%%%%
%%%%%%%%%%%%%%%%%%%%%%%%%%%%%%%%%%%%%%%%
%%%%%%%%%%%%%%%%%%%%%%%%%%%%%%%%%%%%%%%%
%%%%%%%%%%%%%%%%%%%%%%%%%%%%%%%%%%%%%%%%
\subsection{Application to the transitive case}
%%%%%%%%%%%%%%%%%%%%%%%%%%%%%%%%%%%%%%%%
%%%%%%%%%%%%%%%%%%%%%%%%%%%%%%%%%%%%%%%%
%%%%%%%%%%%%%%%%%%%%%%%%%%%%%%%%%%%%%%%%
%%%%%%%%%%%%%%%%%%%%%%%%%%%%%%%%%%%%%%%%

When $(\Sigma, \omega)$ is a transitive, there is yet another natural choice of a $\Sigma$-space $P$ to which one can apply Proposition \ref{prp-general-VE}:  the $s$-fiber $s^{-1}(x)$, for $x\in \BB$.

\begin{lem}\label{Gelf_Fuc}
There is an isomorphism
\[
\Omega^*t(s^{-1}(x))^{\Sigma}\simeq C^*(\extg_x(A))
\]
where $\extg_x(A)$ is the extended isotropy Lie algebra from Definition \ref{defn:extended-isotropy}. 
Therefore, there is a canonical map in cohomology
\[
VE_x: H^*_{\rm Haef}(\Sigma, \omega)\to H^* (\extg_x(A))
\]
and this map is an isomorphism up to degree $l$ if the isotropy group $\Sigma_x$ is cohomologically $l$-connected. 
\end{lem}

\begin{proof}
%\footnote{{\color{red} TO BE ADAPTED TO THE NEW VERSION OF THE PAPER!}}
This proof is best approached working with a multiplicative pointwise surjective form $\omega\in \Omega^1(\Sigma, t^*A)$ whose kernel is $\mathcal{C}$, see Remark~\ref{rk:Cartan-gpds-via-forms} and Corollary~\ref{cor:MC-eq-for-Cartan-form}. In particular, we will need the Maurer-Cartan equation from Corollary~\ref{cor:MC-eq-for-Cartan-form} or, more precisely, the one for the restriction of $\omega$ to $s^{-1}(x)$; see Proposition~\ref{prp:MC-eq-CP} (of course, $(s^{-1}(x), \omega|_{s^{-1}(x)})$ is a flat $(\Sigma, \omega)$-structure, compare with Example~\ref{ex:Sigma-is-Cartan-bundle}). We will keep the notation $t$ for the restriction of the target map $t:\Sigma \to \BB$ to $s^{-1}(x)$, which is surjective because $\Sigma$ is transitive.

Notice that we have a map
\[
r_{1_x}: \mathfrak{X}(s^{-1}(x))^{\Sigma} \to A_x,\quad X\to (\omega|_{s^{-1}(x)}(X))|_{1_x}
\]
One sees immediately that this is a linear isomorphism.
%from invariant vector fields on $s^{-1}(x)$ to elements of $A_x$;
%The linear isomorphism (we keep using $t$ to denote the restriction of the target map to $s^{-1}(x)$, which surjective because $\Sigma$ is transitive)
%\[
%\omega|_{s^{-1}(x)}: Ts^{-1}(x)\to t^*A
%\]
Observe that the inverse of $r_{1_x}$ is given by 
\[
i_x: \alpha_x\in A_x \to i_x(\alpha_x)\in \Gamma(Ts^{-1}(x))
\]
where, for all $g\in s^{-1}(x)$,
\[
(i_x(\alpha_x))_g=dm ({\rm hor}_g^\mathcal{C}(dt(\omega_{1_x}^{-1}(\alpha_x))),\omega_{1_x}^{-1}(\alpha_x)) \in T_gs^{-1}(x).
\]
%Again, this is clearly a bijection between elements of $A_x\cong T_xs^{-1}(x)$ and invariant sections of $Ts^{-1}(x)$.

%(one can readily check that, for any $\Sigma$-action on $P$, the pullback of the Cartan form on $\Sigma\ltimes P$ makes it Pfaffian; moreover, the Pfaffian representation of $\Sigma\ltimes P$ on $TP$ has the formula we wrote above). 
We can choose $\omega$ so that $\omega_{1_x}$ is the identity map on $s$-vertical vectors and the formula for $i_x$ simplifies to
\[
(i_x(\alpha_x))_g=dm ({\rm hor}_g^\mathcal{C}(\rho(\alpha_x)),\alpha_x) \in T_gs^{-1}(x).
\] 
%(it is not difficult to see that $\rho\circ \omega|_{M}=dt$).
%If we postcompose $\omega|_{s^{-1}(x)}: Ts^{-1}(x)\to t^*A$ with the restriction to $1_x$, we get 

%The proof of this is based on the following fact (see~\cite{Sal13}): if $(\Sigma, \omega)$ is a flat Cartan groupoid then a {\bf Maurer-Cartan equation} holds. The equation is
%\[
%d_D\omega+\{\omega, \omega\}=0
%\]
%where:
%\begin{itemize}
%\item the differential $d_D$ is defined by
%\begin{align*}
%d_D\theta(X_1, \dots, X_{q+1})=\sum\limits_i (-1)^{i+1} D^t_{X^i}\theta(X_1, \dots, \hat{X}_i, \dots, X_{q+1})\\
%-\sum\limits_{i<j}\theta([X_i, X_j], X_1, \dots, \hat{X}_i, \dots, \hat{X}_j, \dots, X_{q+1})
%\end{align*}
%for any $q$-form $\theta$ on $\Sigma$ with coefficents in $t^*A$, where $D^t$ is the pullback Spencer operator (in particular, a flat connection) induced by $D$ on the algebroid $t^*A$;
%\item $\{\omega, \omega\}(X, Y)=\{\omega(X), \omega(Y)\}$.
%\end{itemize}
We claim that $r_{1_x}$ is an anti-Lie algebra map, so that composing with the map 
\[
A_x\to A_x,\quad \alpha_x\to -\alpha_x
\]
we get the desired isomorphism. 

As anticipated above, this follows because $\omega|_{s^{-1}(x)}$ satisfies the Maurer-Cartan equation
\[
d_{t^*A}(\omega|_{s^{-1}(x)})+\frac{1}{2}\{\omega|_{s^{-1}(x)}, \omega|_{s^{-1}(x)}\}_{pt}=0.
\]
To keep the notation simple, we will use the same symbols as in Proposition~\ref{prp:MC-eq-CP}: $\widetilde{\nabla}$ will denote the infinitesimally multiplicative flat connection on $t^*A$ obtained pulling back $\nabla$ via $t$, $\widetilde{A}$ will denote $t^*A$ and we rebaptise $\omega|_{s^{-1}(x)}$ as $\theta$, so that the Maurer-Cartan equation becomes
\[
d_{\widetilde{A}}(\theta)+\frac{1}{2}\{\theta, \theta\}_{pt}=0.
\]
Let us recall that, for any $X_1, X_2\in \Gamma(Ts^{-1}(x))$,
\[
d_{\widetilde{A}}(\theta)(X_1, X_2)= \widetilde{\nabla}_{X_1}(\theta(X_2))-\widetilde{\nabla}_{X_2}(\theta(X_1))-\theta([X_1, X_2]).
\]
%On the other hand, for any $X_1, X_2\in \Gamma(\widetilde{A})$ whose value at $g\in s^{-1}(x)$ is $a_1, a_2\in A_{t(g)}$,
%\begin{align*}
%\{\theta, \theta\}_{pt}(X_1, X_2)|_{g}=\{\theta(a_1), \theta(a_2\}_{pt}
%=[\widehat{\theta(a_1)}, \widehat{\theta(a_2)}]-\nabla_{\rho(\theta(a_1)}(\widehat{\theta(a_2)})+\nabla_{\theta(a_2)}(\widehat{\theta(a_1)})
%\end{align*}
%where $\widehat{\theta(a_1)}$, $\widehat{\theta(a_2)}$ denote extensions to sections of $A$.
Consequently, from the Maurer-Cartan equation we see that for each pair of invariant vector fields $X_1, X_2$, it holds
\[
\{\theta(X_1), \theta(X_2)\}_{pt}= -\widetilde{\nabla}_{X_1}(\theta(X_2))+\widetilde{\nabla}_{X_2}(\theta(X_1))+\theta([X_1, X_2])
\]
%We are going to show that $\widetilde{\nabla}_{X_1}(\theta(X_2))-\widetilde{\nabla}_{X_2}(\theta(X_1))=2\theta([X_1, X_2])$.
As it is shown in Proposition $5.3.4$ from~\cite{FRANCESCO}, the following identity (already recalled while proving Proposition~\ref{prp:MC-eq-CP}) holds for all $X_1, X_2\in \Gamma(Ts^{-1}(x))$:
\[
\widetilde{\nabla}_{X_1}(\theta(X_2))=\widetilde{\nabla}_{\widetilde{\rho}(\theta(X_1))}(\theta(X_2))+[\theta(X_2),\theta(X_1)]_{\widetilde{A}}-\theta([\widetilde{\rho}(\theta(X_2)), X_1])
\]
%D^t_Y\omega(X)=\nabla_{\omega(X)}\omega(Y)-\omega[\rho(\omega(X)), Y]
where we observe that the first two terms of the right hand side correspond to the action of $\theta(X_2)$ on $\theta(X_1)$ using the representation of $\widetilde{A}$ on itself (see formula~\eqref{A_conn}, taking into account that $(\widetilde{A}, \widetilde{\nabla})$ is a flat Cartan algebroid).
We claim that, if $X_1$ and $X_2$ are left invariant, then $\theta(X_1)$ and $\theta(X_2)$ are left invariant for the representation of $\Sigma\ltimes s^{-1}(x)$ on its algebroid $\widetilde{A}$ (see~\eqref{A_representation}). Assuming for a moment that this claim is true, using the formula recalled above we see that
\[
\widetilde{\nabla}_{X_1}(\theta(X_2))=-\theta([\widetilde{\rho}(\theta(X_2)), X_1]),\ \ \widetilde{\nabla}_{X_2}(\theta(X_1))=-\theta([\widetilde{\rho}(\theta(X_1)), X_2])
\]
for any pair of left invariant vector fields $X_1, X_2$ over $s^{-1}(x)$. Still using Proposition $5.3.4$ of~\cite{FRANCESCO}, one has that $\theta$ is the inverse of $\widetilde{\rho}$, hence $-\widetilde{\nabla}_{X_1}(\theta(X_2))+\widetilde{\nabla}_{X_2}(\theta(X_1))=-2\theta([X_1, X_2])$ holds and we have 
\[
\{\theta(X_1), \theta(X_2)\}=-\theta[X_1, X_2]
\]
i.e. the map 
\[
\mathfrak{X}(s^{-1}(x))^{\Sigma} \to A_x,\quad X\to -\theta(X)_{1_x}
\]
is an isomorphism of Lie algebras, which was our original claim.

We are left to show that if $X$ is left invariant then $\theta(X)$ is left invariant, which follows by proving that $\theta$ is a map of $\Sigma\ltimes s^{-1}(x)$-representations. Of course, this uses multiplicativity.
In the following computation, we see $\widetilde{A}=t^*A$ as the fibered product $A\times_\BB s^{-1}(x)$ and $m$ denotes both the multiplication on $\Sigma$ and the one on the action groupoid:
%^{H^{pr_1^*(\omega)}}_{(g, 1_x)}
\begin{align*}
(g, 1_x)\cdot (\theta(X_x), 1_x)=dR_{(g^{-1}, g)} dm\left({\rm hor}^{{pr_1}^{-1}\mathcal{C}}_{(g, 1_x)}(\widetilde{\rho}(\theta(X_x))),(\theta(X_x), 1_x)\right)=\\
dR_{(g^{-1}, g)} dm\left(\left({\rm hor}^{\mathcal{C}}_g({\rho}(X_x)), \widetilde{\rho}(\theta(X_x)\right),\left(X_x, 1_x\right)\right)=dR_{(g^{-1}, g)} \left(dm({\rm hor}^\mathcal{C}_g({\rho}(X_x)), X_x), 1_x\right)=\\
(\theta(X_g), g).
\end{align*}
%Notice that the third term makes sense since $\rho(\omega(X_x))=\omega(X_x)=X_x$, using the fact that $\rho$ at $t^*A_{1_x}=A_x$ is given by the infinitesimal action map from $A$ to $Ts^{-1}(x)$. 
The last step follows from multiplicativity, which implies
\[
\theta(X_g)=g\cdot \theta(X_x)=dR_{g^{-1}}dm({\rm hor}^\mathcal{C}_g(\rho(X_x)), X_x)
\]
% Finally, $\rho(\omega(X))=X$ and the same holds for $Y$. In fact, 
%\[
%\rho(\omega(X_g), 0)=dm(\omega(X_g), 0)=dR_g(\omega(X_g))=dm(\rho(X_x)^{H^\omega_g}, X_x)=X_g
%\]
\end{proof}

We can then state the following theorem.
\begin{thm}\label{Improved_Haefliger}
Let $(\Sigma, \CC)$ be a transitive flat Cartan groupoid.
If $K\subset \Sigma_x$ is a compact Lie subgroup with the property that $\Sigma_x/K$ is contractible, then the map from the previous Lemma factors though an isomorphism
\[
H^*_{\rm Haef}(\Sigma, \CC)\overset{\sim}{\rightarrow} H^* (\extg_x(A),K).
\]
\end{thm}
\begin{rmk}
In examples coming from geometric structures, there is always a natural choice of $K$, even in the profinite dimensional case (see Example~\ref{Classical_Haefliger}).
\end{rmk}
\begin{proof}
%The first part is clear. 
The triple $(s^{-1}(x), t, \Sigma_x)$ is a principal bundle and we know that $s^{-1}(x)/K$ has contractible fibers. Then, we can apply Proposition~\ref{prp-general-VE} to the $\Sigma$-space $s^{-1}(x)/K$. 
%Notice that the $\Sigma$-space $s^{-1}(x)/K$ is still proper.
%due to the compactness of $K$.
We have to prove that $\Sigma$-invariant forms on $s^{-1}(x)/K$ are isomorphic to the Lie algebra cochains in $A_x$ relative to $K$. For this, one just uses the fact that the isomorphism constructed in Lemma~\ref{Gelf_Fuc} descends to the basic subcomplexes; that is, it respects the contraction and the Lie derivative with respect to elements of $\mathfrak{k}={\rm Lie}(K)$. 
%This is readily verified using the associativity of the multiplication on $\Sigma$. 
\end{proof}
Composing the isomorphism from the above Theorem with map~\eqref{eq:Haefl-charact-map}, one sees that, if $(\Sigma, \CC)$ is a transitive flat Cartan groupoid, a $(\Sigma, \CC)$-structure comes with a natural map
\begin{equation}\label{eq:transitive_char_map}
\kappa^P: H^*(\mathfrak{a}_x(A), K)\to H^*(M).
\end{equation}

\begin{exm}\label{Classical_Haefliger}
%\footnote{{\color{red} to be re-adapted. Also, shall we add here that the construction of the characteristic map, as it follows from the last proposition, is essentially the one of Bernstein and Rosenfeld (isn't it?).}}  
When $\Gamma$ is a transitive Lie pseudogroup and $\Sigma=J^\infty\Gamma$, the resulting isomorphism becomes 
\[
H^*_{\rm diff}(\Gamma)\cong H^*(\mathfrak{a}_x(\Gamma), K),
\]
see Definition~\ref{defn:Gamma-vector-fields} and Example~\ref{ex-Cartan-alg-J}.
For $k\geq 2$, the projections
\[
pr^{k, k-1}: J^k\Gamma\to J^{k-1}\Gamma
\]
are affine bundles projections; consequently, one can always choose a subgroup $K$ as in Theorem~\ref{Improved_Haefliger}.

In particular, for $\Gamma={\rm Diff}_{\rm loc}(\mathbb{R}^q)$ and $\Sigma=J^\infty{\rm Diff}_{\rm loc}(\mathbb{R}^q)$ 
%and $\Gamma$ is a transitive Lie pseudogroup, 
the group $K$ can be identified with the orthogonal group $O(q)$. This is a phenomenon slightly more general than ${\rm Diff}_{\rm loc}(\mathbb{R}^q)$ that we explain in Example~\ref{ex:G-struct}. As a consequence, the isomorphism from Theorem~\ref{Improved_Haefliger} reduces to~\eqref{eq: Van-Est-Haefliger}, with codomain given by the relative Gelfand-Fuchs cohomology of formal vector fields. 

Going back to a transitive $\Gamma$, the characteristic map~\eqref{eq:transitive_char_map} corresponding to an almost $\Gamma$-structure $\P^\infty$
\begin{equation}\label{eq:transitive_char_map_pseudo}
\kappa^{\P^\infty}: H^*(a_x(\Gamma), K)\to H^*(M)
\end{equation}
is seen to be defined on the relative Gelfand-Fuchs cohomology of $\Gamma$-vector fields.
%to be a maximal compact subgroup of the isotropy group of the first jet bundle.   
\end{exm}

The characteristic map~\eqref{eq:transitive_char_map_pseudo} for transitive Lie pseudogroups had been addressed in~\cite{BOTTHAEFLIGER} and in~\cite{BERNSTEINROSENFELD, BERNSTEINROSENFELD2}. We adapt the construction of~\cite{BERNSTEINROSENFELD, BERNSTEINROSENFELD2} to our setting; the outcome is a description of map~\eqref{eq:transitive_char_map} which is in the same spirit as the geometric characteristic map of flat principal bundles (see Subsection~\ref{Geometric/abstract characteristic classes for flat bundles}), i.e. of Chern-Weil nature. We use the description from Proposition~\ref{prp:MC-eq-CP} of $\CC_P$ in terms of $\theta\in \Omega^1(P, \mu^*A)$. Restricting to 
\[ Q:= \mu^{-1}(x),\]
we obtain a principal $\Sigma_x$-bundle $Q\to M$ and 
\[ \theta_Q:= \theta|_{Q}\in \Omega^1(P, \mathfrak{a}_x(A))\]
will satisfy the Maurer-Cartan equation (cf. Proposition \ref{prp:MC-eq-CP}). The form $\theta_Q$ induces a map 
\[ \theta_Q^*: C^{*}(\mathfrak{a}_x(A))\to \Omega^*(Q).\]
This works as in the standard situation of flat principal bundles. Passing to $K$-basic elements and using that the fibers of $Q/K\to M$ are contractible, one obtains an 
induced map (compare with (\ref{kappa-omega})):
\[ \kappa^{\theta_Q}: H^{*}(\mathfrak{a}_x(A),K)\to H^{*}(M).\]
%\begin{prp}\label{prp:Maurer_Cartan_characteristic_map} The map $k_{\theta}$ induced by the 1-form $\theta$ coincides with the characteristic map (\ref{eq:Haefl-charact-map}), modulo the isomorphism from Theorem \ref{Improved_Haefliger}. 
%\end{prp}
%\begin{proof}
%?
%\end{proof}

\begin{exm}\label{ex:G-struct}
To further elaborate on Example~\ref{Classical_Haefliger}, let us consider the transitive pseudogroup $\Gamma_G$ on $\mathbb{R}^q$ induced by some Lie subgroup $G\subset GL(q, \mathbb{R})$ as follows:
\[
\Gamma_G=\{\varphi\in {\rm Diff}_{\rm loc}(\mathbb{R}^q):\ d_x\varphi\in G, x\in {\rm dom}(\varphi)\}
\]
Notice that this pseudogroup is defined by a first order condition; more compactly it is {\bf of first order}, which means that $\varphi\in \Gamma_G$ if and only if $j^1_x\varphi\in J^1\Gamma_G$ for all $x\in {\rm dom}(\varphi)$.
As a consequence, a group $K$ as in Theorem~\ref{Improved_Haefliger} is induced by a maximal compact subgroup $\bar{K}$ of the isotropy group $(J^1\Gamma_G)_x \cong G$. In fact, one defines $K$ to be the group of infinite jets at $x$ of linear diffeomorphisms $\varphi$ fixing $x$ and such that $d_x\varphi\in \bar{K}$.

When $\dim(M)=q$, a $\Gamma_G$-structure on $M$ is the same thing as a {\bf flat} (or {\bf integrable}) {\bf $G$-structure}; a $(J^\infty\Gamma_G, \CC^\infty)$-structure on $M$, as in example~\ref{ex:formal_Gamma_structures}, is a {\bf formally integrable} {\bf $G$-structure} (see, for example,~\cite{STERNBERG} for $G$-structures and~\cite{FRANCESCO} for a discussion about formal integrability that is close to the formalism of this paper). 
When $\dim(M)=n>q$, a $\Gamma_G$-structure is a codimension $q$ foliation equipped with a transverse flat $G$-structure. 
\begin{comment}
Moreover, we point out that the characteristic map described in~\cite{BERNSTEINROSENFELD2} is essentially the one from Proposition~\ref{prp:Maurer_Cartan_characteristic_map}. 
\end{comment}
\end{exm}

\appendix

%%%%%%%%%%%%%%%%%%%%%%%%%%%%%%
%%%%%%%%%%%%%%%%%%%%%%%%%%%%%%
%%%%%%%%%%%%%%%%%%%%%%%%%%%%%%
%%%%%%%%%%%%%%%%%%%%%%%%%%%%%%
%%%%%%%%%%%%%%%%%%%%%%%%%%%%%%
%%%%%%%%%%%%%%%%%%%%%%%%%%%%%%
%%%%%%%%%%%%%%%%%%%%%%%%%%%%%%
%%%%%%%%%%%%%%%%%%%%%%%%%%%%%%
%%%%%%%%%%%%%%%%%%%%%%%%%%%%%%
%%%%%%%%%%%%%%%%%%%%%%%%%%%%%%
%%%%%%%%%%%%%%%%%%%%%%%%%%%%%%
%%%%%%%%%%%%%%%%%%%%%%%%%%%%%%
%%%%%%%%%%%%%%%%%%%%%%%%%%%%%%
%%%%%%%%%%%%%%%%%%%%%%%%%%%%%%
%%%%%%%%%%%%%%%%%%%%%%%%%%%%%%
%%%%%%%%%%%%%%%%%%%%%%%%%%%%%%
%%%%%%%%%%%%%%%%%%%%%%%%%%%%%%
%%%%%%%%%%%%%%%%%%%%%%%%%%%%%%
%%%%%%%%%%%%%%%%%%%%%%%%%%%%%%
%%%%%%%%%%%%%%%%%%%%%%%%%%%%%%
%%%%%%%%%%%%%%%%%%%%%%%%%%%%%%
%%%%%%%%%%%%%%%%%%%%%%%%%%%%%%
%%%%%%%%%%%%%%%%%%%%%%%%%%%%%%
%%%%%%%%%%%%%%%%%%%%%%%%%%%%%%
%\section{Pf manifolds (pf= pro-finite)}
\section{Pro-finite manifolds}

%%%%%%%%%%%%%%%%%%%%
%%%%%%%%%%%%%%%%%%%%
%%%%%%%%%%%%%%%%%%%%
%%%%%%%%%%%%%%%%%%%%
%%%%%%%%%%%%%%%%%%%%
\subsection{Pf manifolds (pf= pro-finite)}
%%%%%%%%%%%%%%%%%%%%
%%%%%%%%%%%%%%%%%%%%
%%%%%%%%%%%%%%%%%%%%
%%%%%%%%%%%%%%%%%%%%
%%%%%%%%%%%%%%%%%%%%
%%%%%%%%%%%%%%%%%%%%
Pro-finite manifolds are, in principle, infinite dimensional manifolds together with a ``tail'' of finite dimensional ones, a tail that allows one to export the usual (finite dimensional) concepts to the infinite dimensional setting. The basic example is that of infinite jets, where the ``tail'' is the sequence of jets spaces of finite orders. 

\begin{defn}\label{defn:tower-mnfds}
A tower of manifolds $\iM_{\bullet}$ is a sequence 
\[ \iM_{\bullet}: \quad \ldots \longrightarrow \iM_k\stackrel{\pi_k}{\longrightarrow} \iM_{k-1}\longrightarrow \ldots \longrightarrow \iM_{2} \stackrel{\pi_2}{\longrightarrow}
\iM_{1} \stackrel{\pi_1}{\longrightarrow} \iM_{0}\]
consisting of smooth, finite dimensional manifolds $\iM_k$ and surjective submersions between them. 
%For $k\geq l$ we will use the notation 
%\[ \pi_{k, l}: \iM_k\to \iM_l\]
%obtained by successively composing projections. 

A {\bf pf-atlas} on a set $\iM$, denoted
\[ a: \iM\to \iM_{\bullet},\] 
is a tower of manifolds $\iM_{\bullet}$ together with 
% a projection $p: \iM\to \iM_{\bullet}$, i.e. 
a collection $a= \{a_k: \iM\to \iM_k\}$ of 
surjections compatible with the tower projections ($\pi_k\circ a_k= a_{k-1}$).
\end{defn}

\begin{rmk}
Given such a pf-atlas $a: \iM\to \iM_{\bullet}$, any $x\in \iM$ has associated components w.r.t. $a$, namely:
% one may think of 
\[ x_k= a_k(x)\in \iM_k .\]
But note that, in principle, two distinct points may have the same components.

Perhaps a better name for the components $x_k= a_k(x)$ would be that of ``truncations of $x$''. They are not arbitrary - actually each one of them determines the previous ones by:
\[ x_{k-1}= \pi_k(x_k) \quad \textrm{for all $k$}.\]
However, there may be sequences $(y_k)$ satisfying this compatibility but which do not arise from an element of $x\in \iM$. 
%When that does happen, one may think that $x$ is "the inverse limit of the sequence $(y_k)$". 

All these remarks can be packed together by considering the inverse limit of the tower
\[ \plim \iM_{k}:= \{(x_k)_{k\geq 0}: x_k\in \iM_k, \pi_k(x_k)= x_{k-1}\}\] 
and interpreting $a$ as a map
\[ a: \iM\to \plim_k \iM_k .\]
And we were just saying that this map may fail to be injective (different elements may have same components) or surjective (compatible sequences may not have a limit in $\iM$). 
\end{rmk}

\begin{defn} \label{defn:normal-tower}
A pf-atlas $a: \iM\to \iM_{\bullet}$ is said to be {\bf normal} if the induced map $\iM\to \plim \iM_k$ is a bijection.  
\end{defn}

\begin{exm}\label{ex:pro-not-normal}
One can always achieve normality by replacing $\iM$ by $\plim_k \iM_k$. However, there are natural examples which are not necessarily normal. For instance, $J^{\infty}\Gamma$, for certain choices of $\Gamma$ (see Example~\ref{exam-Lie-pseudogroups}). A standard example is given by the pseudogroup of analytic diffeomorphisms of $\mathbb{R}^q$; see Example~~\ref{exam-Lie-pseudogroups} and Remark~\ref{Lie_exm}. 
\end{exm}

%
% To any profinite manifold $\iM$ one can associated a normalised one, $\iM_{\textrm{norm}}$: just take $\plim_k \iM_k$ where $a: \iM\to \iM_{\bullet}$ is an atlas (but, up to isomorphisms, independent of the choice of the atlas). But there are natural examples of pro-finite manifolds which are not normal. For instance  $J^{\infty}\Gamma \ldots$. 
%\end{exm}

\begin{defn}\label{defn:prof-map-atl} Consider $a: \iM\to \iM_{\bullet}$ and $b: \iN\to \iN_{\bullet}$ be two pf-atlases. A set theoretical function $f: \iM\to \iN$ is said to be a smooth map w.r.t. $a$ and $b$, or we say that $f: (\iM, a)\to (\iN, b)$ is a {\bf smooth map of pf-atlases} if, for each $k$, the $k$-th component
\[ b_k\circ f: \iM\to \iN_k\]
comes from a smooth map defined on some $\iM_{m_k}$ for some $m_k$ large enough:
\[ b_k\circ f= f^{m_k}_{k}\circ a_{m_k},\quad \textrm{with $f^{m_k}_{k}\in C^{\infty}(\iM_{m_k}, \iN_k)$}.\]
% Any sequence $(f^{m_{1}}_{1}, f^{m_{2}}_{2}, \ldots )$ like this, with $m_1< m_2< \ldots $ is  
\end{defn}

\begin{rmk}\label{rk:concrete-morphism} The outcome is very much related to the notion of morphism between towers of manifolds. 
Seeing towers as particular cases of projective systems of (finite dimensional) manifolds, one obtains an abstract but very standard description of the resulting hom-spaces:
\[ \textrm{Hom}(\iM_{\bullet}, \iN_{\bullet}):= \plim_{n} \ilim_{m} C^{\infty}(\iM_{m}, \iN_{n}).\]
Working this out and using that the $\pi_k$'s are surjective submersions, one obtains the following description which is related to the previous definition. 

First of all, let us call a {\bf concrete morphism} between two towers $\iM_{\bullet}$ and $\iN_{\bullet}$ any 
collection of smooth maps
\[ f_{\bullet}= \{f_k:\iM_{m_k}\to \iN_{k}\}_{k\geq 1}\quad   \textrm{with $m_1\leq m_2\leq \ldots$}\]
which is compatible in the sense that that it makes the resulting squares commutative:
\[
\xymatrix{
\ldots \ar[r] & \iM_{m_k} \ar[rr]^{\pi_{m_k, m_{k-1}}}\ar[dr]^{f_{k}}  &  &   \iM_{m_{k-1}}\ar[dr]^{f_{k-1}}\ar[r]& \ldots &\\
&  \ldots \ar[r] & \iN_k \ar[rr]^{\pi_{k, k-1}} &  & \iN_{k-1}  \ar[r]& \ldots 
}  
\] 
Note that, for the maps $f_{k}^{m_k}$ arising in the previous definition, this compatibility is automatically satisfied. Therefore, the condition from the previous definition can be reformulated by requiring that $f: \iM\to \iN$ 
is compatible with a concrete morphism of towers $f_{\bullet}$ from $\iM_{\bullet}$ to $\iN_{\bullet}$.

\medskip 
Next, given a second concrete morphism $g_{\bullet}= \{g_{k}: \iM_{l_{k}}\to  \iN_{k}\}$ as above, we say that $f_{\bullet}$ and $g_{\bullet}$  are {\bf equivalent} concrete morphisms if for each $k$, $f_k$ and $g_{k}$ become equal after composing them with the projections $\iM_K\to \iM_{m_k}$ and $\iM_{K}\to \iM_{l_k}$, respectively, for some $K$-large enough (which actually may be chosen to be $\textrm{max}\{m_k, l_k\}$). With this, a {\bf morphism of towers} is an equivalence class of concrete morphisms. 

And, again, it is useful to look back at the previous definition: while there the actual $f_k= f^{m_k}_{k}$ are far from unique (they are just required to exist), we see that any two choices are equivalent; hence one has a unique, well-defined induced morphism of towers
\[ f_{\bullet}\in \textrm{Hom}(\iM_{\bullet}, \iN_{\bullet}).\]
Similar to the previous discussion about coordinates, two different $f$'s may still induce the same morphism of towers, or there are morphism of towers that are not induced by an $f$. However, neither of these can happen when the two atlases are normal.
\end{rmk}

\begin{exm} \label{tower-sample} It is straightforward now to define composition of morphisms of towers obtaining, therefore, the category of towers (of finite dimensional smooth manifolds). With the resulting notion of isomorphism we see that, given $\iM_{\bullet}$, if we 
sample the tower via an increasing sequence of indices $m_0< m_1< m_2< \ldots$, then
\[ \ldots \longrightarrow \iM_{m_k}\stackrel{\pi_k}{\longrightarrow} \iM_{m_{k-1}}\longrightarrow \ldots \longrightarrow \iM_{m_2} \stackrel{\pi_2}{\longrightarrow}
\iM_{m_1} \stackrel{\pi_1}{\longrightarrow} \iM_{m_0}\] 
is isomorphic to $\iM_{\bullet}$ itself. 
\end{exm}

%
%\begin{rmk} 
%Towers of manifolds can be seen as particular cases of projective systems of (finite dimensional) manifolds. From that point of view, a morphism between $M_{\bullet}$ and $N_{\bullet}$ is any element in the double limit 
%\[ \plim_{m} \ilim_{n} C^{\infty}(M_{n}, N_{m}).\]
%Working this out and using that the $\pi_k$s are surjective submersions, one recovers precisely the  notion of morphism of towers described in the previous remark. 
%\end{rmk}

Given Example \ref{ex:pro-not-normal}, it is interesting to consider atlases that are not normal. However, to make our life easier (and be able to just quote existing results), from now on we will start imposing the normality condition. 

%Note however the usefulness of the generality of the previous discussion: while in the normal case one may just assume that $M= \plim M_k$, we would then have to work from the beginning with 
%

\begin{defn}
A {\bf pf-manifold} is a set $\iM$ together with an equivalence class of normal pf-atlases, where two pf-atlases 
\[ a:\iM\to \iM_{\bullet}, \quad a': \iM\to \iM^{'}_{\bullet}\] 
are said to be {\bf equivalent} if the identity map $\textrm{id}_\iM$ is a smooth map of pf-atlases (in the sense of Definition \ref{defn:prof-map-atl}) both as a map from $(\iM, a)$ to $(\iM, a')$ as well as the other way around. Equivalently, we require the existence of an isomorphism of towers 
\[ c_{a, a'}\in \textrm{Hom}(\iM_{\bullet}, \iM^{'}_{\bullet})\quad \textrm{(``change of coordinates'')}\]
that is compatible with $a$ and $a'$. 
\end{defn}

\begin{rmk}[the underlying topological space] \label{rmk:|M|} Note that any pf-manifold carries a natural topology: the smallest one which makes all the projection $a_k$ continuous (using a/any atlas $a$).  As terminology: 
\begin{itemize}
\item we will simply talk about ``the pf-manifold $\iM$'' without explicitly specifying the (equivalence class) of the atlas in the notation. 
\item when thinking about $\iM$ just as a topological space 
%and we want to emphasise that,
 we will use the notation $|\iM|$. 
\end{itemize}
\end{rmk}

\begin{defn}\label{defn:pro-map}
A {\bf smooth pf-map} $f: \iM\to \iN$ between two pf-manifolds is any set theoretical map with the property that, for a/any atlas $a$ of $\iM$ and $b$ of $\iN$, $f$ is a smooth map w.r.t.  $a$ and $b$ (in the sense of Definition \ref{defn:prof-map-atl}). We will denote by $C^{\infty}(\iM, \iN)$ the space of such smooth maps. 
\end{defn}

%
%For a smooth map $f: M\to N$ we see that, for any two atlases $a$ for $M$ and $b$ for $N$, $f$ will be represented by a morphism of towers 
%\[ f^{a, b}\in \textrm{Hom}(M_{\bullet}\to N_{\bullet})\]
%and then by a concrete morphism 
%\[ f^{a, b}_{k}: M_{m_k} \to N_{k}.\]
%
%
%
%
%The notion of {\bf morphism} $f: M\to N$ between profinite manifolds $M$ and $N$ is rather straightforward. 
%First of all, there is a continuous map
%\[ |f|: |M|\to |N| .\]
%Then, given profinite atlases $a:M\to M_{\bullet}$ and $b: N\to N_{\bullet}$, $f$ will be represented by morphisms of towers
%\[ f_{a, b}: M_{\bullet}\to N_{\bullet}\]
%which are compatible with $|f|$. Finally, if we change the profinite atlases to $a'$ and $b'$, then 
%\[ f_{a', b'}= c_{b, b'}\circ f_{a, b}\circ c_{a', a} .\]
%Strictly speaking, we described an equivalence relation on the collection of morphisms of towers representing $|f|$, and the morphism itself is $|f|$ together with 
%%   one has an equivalence relation and morphisms are maps $|f|$ together with 
%the choice of such an equivalence class. We will denote by $C^{\infty}(M, N)$ the space of such morphsims.
%

\begin{exm} When $\iN= \mathbb{R}$ we obtain the algebra $C^{\infty}(\iM)$ of smooth functions on the pf-manifold $\iM$; working out the previous definition we see that they are continuous functions 
\[ f: \iM\to \R\]
with the property that there exists $m_0$ such that 
\[ f= g\circ a_{m_0}, \quad \textrm{with $g: \iM_{m_0}\to \R$ smooth in the usual sense}.\]
Abstractly, we deal with 
\[ \ilim_m C^{\infty}(\iM_m)= \cup_m C^{\infty}(\iM_m),\]
viewed as a subspace of $C(\iM)$ via the pull-back by $a$. Concretely, we think of functions $f= f(x)$ on $\iM$ as functions that may depend on 
%infinite number of variables
infinitely many variables
\[ x_0= a_0(x), \quad x_1= a_1(x),  \quad \ldots,  \quad  x_k= a_k(x),  \quad \ldots\]
while smooth functions depend only on a finite number of variables $m_0$ in a smooth way. 

One can also require this condition only locally- and that give rise to another notion that we would call  ``smooth function on the space $|\iM|$" (for the notation see Remark \ref{rmk:|M|}); this gives rise to a sequence of (strict) inclusions
\begin{equation}\label{Aeq:inclusion-smooth} 
C^{\infty}(\iM)\subset  C^{\infty}(|\iM|) \subset Cont(|\iM|).
\end{equation}
\end{exm}
%
%\begin{exm} Any finite dimensional manifold $N$ can be interpreted pro-manifold using the identity tower $N\to N\to N\to \ldots$ and the previous remark applies to morphisms $M\to N$.  
%\end{exm}

%Finally, the notion of pf-groupoid $\G\tto M$ that we need in the main body of the paper does not need the most general concept when we would have $\G_{k}$s, $M_{k}$s and the structure maps (source/target, multiplication) are just pf-maps (hence may lower the index $k$). Actually, in almost all our examples the base $M$ will be finite dimensional. And in all our examples we are in the "strict" case in the following sense:
%
%\begin{dfn} A {\bf strict pf-groupoid} is a topological groupoid $\G\tto M$, where $\G$ and $M$ are pf-manifolds which admits pf-atlases $\tilde{a}: \G\to \G_{\bullet}$, $a: M\to M_{\bullet}$, where each $\G_k$ is a groupoid over $M_k$ with structure maps compatible with those of $\G$ (towers of Lie groupoids). 
%\end{dfn}
%

\begin{exm}\label{ex:manifds-are-pf} 
Of course, any finite dimensional manifold $M$ is a pf-manifold: just use the constant tower $\iM_{\bullet}= M$. 
\end{exm}

\begin{exm}\label{exam-Jinfty-R} The basic examples of towers and pf-manifolds are provided by jets and PDEs. Let us choose a fibration (meaning a surjective submersion)
\[ \pi: R\to M\]
between finite dimensional manifolds. Associated to such a fibration there is an entire tower of jet spaces 
\[ J^\infty R:\quad  \ldots \to J^2R\to J^1R\to J^0R= R \]
where $J^kR$ denotes the space of $k$-jets of sections of $\pi: R\to M$. Each such space with $k$ finite has a canonical smooth structure, while 
\[ J^\infty R= \plim J^k R\]
becomes a pf-manifold. A particular case of this is when $R= M\times N$ is a product and $\pi$ is the projection on the first factor. Then sections of $R$ are identified with smooth functions from $M$ to $N$ and one obtains the jet-spaces
\[ J^k(M, N):= J^k(\textrm{pr}_1: M\times N\to M).\]

On the other hand one can also proceed more generally, for PDEs on (sections of) fibrations $\pi: R\to M$; an order $k$-PDE is then encoded into a subspace $P\subset J^kR$ and the standard regularity assumption is that $P$ is a submanifold and $P\to M$ is still a fibration. Then one can consider the jet spaces of $P$ itself, $J^lP\subset J^lJ^kR$; the intersection with 
%$J^{k_l}R\subset J^kJ^lR$ 
$J^{k+l}R\subset J^lJ^kR$ are called the $l$-prolongations of $P$:
\[ P^{(l)}\subset J^{k+l}R.\]
All together these give a tower of spaces
\[ P^{(\infty)}:\quad  \ldots \to P^{(2)}\to P^{(1)}\to P^{(0)}= P\]
but the maps may fail to be surjective. The space $P^{(\infty)}$ plays the role of the space of ``formal solutions" of the PDE $P$.  
One says that $P$ is formally integrable if all these maps are surjective submersions. In that case $P^{(\infty)}$ becomes a pf-manifold and the inclusions 
\[ P^{(l)}\subset J^{(k+l)}R \]
provide a natural example of concrete morphism of towers that does not preserve the degree. Of course, they represent the inclusion 
\[ P^{(\infty)}\hookrightarrow J^{\infty}R\]
as a morphism of pf-manifolds. Note also that, interpreting $P^{(l)}$ itself as a PDE of order $k+ l$, one obtains
\[ (P^{(l)})^{(\infty)}\cong P^{(\infty)},\]
as pf-manifolds. 
\end{exm}

%%%%%%%%%%%%%%%%%%%%
%%%%%%%%%%%%%%%%%%%%
%%%%%%%%%%%%%%%%%%%%
%%%%%%%%%%%%%%%%%%%%
%%%%%%%%%%%%%%%%%%%%
\subsection{Direct pf-vector bundles (and sheaves)}
%%%%%%%%%%%%%%%%%%%%
%%%%%%%%%%%%%%%%%%%%
%%%%%%%%%%%%%%%%%%%%
%%%%%%%%%%%%%%%%%%%%
%%%%%%%%%%%%%%%%%%%%
%%%%%%%%%%%%%%%%%%%%
We now move to vector bundles and sheaves. As we shall explain, there are two interesting but different types of vector bundles to consider. In this section we look at the first type, called ``direct". Because of the examples we have in mind, we concentrate on vector bundles; for sheaves the discussion is entirely similar. 
%
%A {\bf vector bundle} over a tower of manifold $M_{\bullet}$ is a tower
%\[ E_{\bullet}: \quad \ldots \longrightarrow E_k\stackrel{\pi_k}{\longrightarrow} E_{k-1}\longrightarrow \ldots \longrightarrow E_{2} \stackrel{\pi_2}{\longrightarrow}
%E_{1} \stackrel{\pi_1}{\longrightarrow} E_{0}\]
%consisting of smooth, finite dimensional vector bundles $E_k\to M_k$, so that the tower maps $E_k\to E_{k-1}$ are vector bundle morphisms covering the tower maps of $M_{k}\to M_{k-1}$.
%
%A morphism $E_{\bullet}\to F_{\bullet}$ between two such vector bundles is a morphism of towers which are represented by strict morphisms that are vector bundle morphisms in the usual sense. 
% 

\begin{defn}\label{defn:pro-vector-bundle}
A {\bf (direct) pf-vector bundle} over a pf-manifold $\iM$ consists of a pf-manifold $\iE$ and a smooth pf-map $p: \iE\to \iM$, together with a vector bundle equivalence classes of pf-vector bundle atlases where:
\begin{itemize}
\item a {\bf pf-vector bundle atlas} for $p: \iE\to \iM$ is an atlas $\tilde{a}: \iE\to \iE_{\bullet}$ for $\iE$, an atlas $a: \iM\to \iM_{\bullet}$ for $\iM$ together with the structure of  
tower of vector bundles
\[ \xymatrix{
\iE_{\bullet}: \ar[d]_{p_{\bullet}} & \ldots \ar[r] & \iE_k \ar[r]^{\pi_k}\ar[d]_{p_k} & \iE_{k-1}\ar[r] \ar[d]_{p_{k-1}} & \ldots \ar[r] &  \iE_{2}\ar[d]_{p_2} \ar[r]^{\pi_2}& 
\iE_{1} \ar[r]^{\pi_1} \ar[d]_{p_1}&  \iE_0\ar[d]_{p_0}\\
\iM_{\bullet}: & \ldots \ar[r] & \iM_k  \ar[r]_{\pi_k} & \iM_{k-1}\ar[r]  & \ldots \ar[r] &  \iM_{2} \ar[r]_{\pi_2}& 
\iM_{1} \ar[r]_{\pi_1} &  \iM_0
} \]
(hence each $\iE_k\to \iM_k$ is a smooth, finite dimensional vector bundle and each tower map $\iE_k\to \iE_{k-1}$ is a vector bundle morphisms covering the tower map $\iM_{k}\to \iM_{k-1}$) that represents $p$ 
% ($p_{\bullet}\circ a_{\bullet}= \underline{a}_{\bullet}\circ p$). 
%
%
%consisting of smooth, finite dimensional vector bundles $\iE_k\to \iM_k$, so that the tower maps $\iE_k\to \iE_{k-1}$ are vector bundle morphisms covering the tower maps of $\iM_{k}\to \iM_{k-1}$, $a_{\bullet}: \iE\to \iE_{\bullet}$ is an atlas for $\iE$, $\underline{a}_{\bullet}: \iM\to \iM_{\bullet}$ is an atlas for $\iM$ such that 
%\[ p_{\bullet}\circ a_{\bullet}= \underline{a}_{\bullet}\circ p.\]
\item two pf-vector bundle atlases, $\tilde{a}$ and $\tilde{a}'$, are said to be equivalent (as pf-vector bundle atlases) if they are equivalent as atlases of $\iE$, with the corresponding isomorphism of towers $c_{\tilde{a}, \tilde{a}'}: \iE_{\bullet}\to \iE_{\bullet}^{'}$ being represented by morphisms of vector bundles. 
\end{itemize}
A {\bf (pf-)section} of $\iE$ is any smooth pf-map $\sigma: \iM\to \iE$ satisfying $p\circ \sigma= \textrm{id}$. We denote by $\Gamma(\iE)$ the space of such sections. 
\end{defn}
%
%\begin{rmk} In principle, one can also allow $p: \iE\to \iM$ to be represented by concrete morphisms of towers which are not strict (do not preserve the degree) but it is not difficult to see that $p$, as a pro-morphism, can be represented by a strict one. 
%\end{rmk}
%
%Here are some remarks that will hopefully
%make the notion of section more transparent. First of all, for any section $\sigma$ of $\iE$, one can always find a vector bundle atlas $a_{\bullet}: \iE\to \iE_{\bullet}$ with respect to which $\sigma$ can be represented by strict sections, i.e. by a family of compatible sections $\sigma_k\in \Gamma(\iE_k)$. 
%
%On the other hand, i

\medskip

\begin{rmk}[representing sections] Given a vector bundle atlas $a: \iE\to \iE_{\bullet}$, the best case scenario is when we deal with sections $\sigma$ that are represented by sections $\sigma_k$ of $\iE_k$. It is not difficult to see that, given any section $\sigma$, one can find an atlas $a$ which realises the best case scenario for the given $\sigma$. However, this is of very little use since $a$ usually changes with $\sigma$.

In practice, we usually have a fixed (preferred) vector bundle 
atlas $a: \iE\to \iE_{\bullet}$; % (covering $\underline{a}_{\bullet}: \iM\to \iM_{\bullet}$); 
given that, any section of $\iE$ can be represented by a concrete morphism (cf. Remark \ref{rk:concrete-morphism}) 
\[ \sigma_{\bullet}= \{\sigma_k:\iM_{m_k}\to \iE_{k}\}_{k\geq 1}\quad   \textrm{with $m_1< m_2< \ldots$}\]
and the section condition is that $p_k\circ \sigma_k$ is the projection $\iM_{m_k}\to \iM_k$; i.e. $\sigma_k$ comes from a section of the pull-back of $\iE_k$ via the tower projection $\pi: \iM_{m_k}\to \iM_k$
\begin{equation}\label{sk-sect}
s_k\in \Gamma(\iM_{m_k}, \pi^*\iE_k).
\end{equation}
% \end{rmk}

% \begin{rmk}[sections of $|\iE|$] 
On the other hand, the underlying topological space $|\iE|$ (cf. Remark \ref{rmk:|M|}) is itself a topological vector bundle over $|\iM|$: the fiber $\iE_x$ above any $x\in |\iM|$ inherits a structure of vector space (actually a pf-vector space) and the corresponding maps (addition and multiplication by scalars) are continuous. And, of course, sections of $\iE\to \iM$ are, in particular, continuous sections of $|\iE|\to |\iM|$:
\begin{equation}\label{eq:sect-inclusion} 
\Gamma(\iE)\subset \Gamma_{\textrm{cont}}(|\iE|).
\end{equation}
To recognize when a section 
\[ \sigma: |\iM|\to |\iE|\]
is a section of the pf-vector bundle $\iE$, let us fix a vector bundle atlas $\tilde{a}: \iE\to \iE_{\bullet}$. 
We look at the $\tilde{a}$-components $(\sigma^0, \sigma^1, \ldots )$ of $\sigma$ and we view them as functions depending 
on the $a$-components $(x_0, x_1, \ldots, )$ of $x$, 
\[ \sigma_k= \sigma_k(x_0, x_1, \ldots)\in \iE_k \quad (\textrm{a section over $|\iM|$ of the pull-back of $\iE_k$}).\]
Then $\sigma$ is a section of the pf-vector bundle $\iE$ if and only if $\sigma_k$ depends only on a finite number of entries: there exists $m_k\geq k$ large enough and $s_k$ as in (\ref{sk-sect}) such that 
\[ \sigma_k(x)= s_k(x_{n_k}) \quad \forall\ x\in |\iM|.\]

%
%$(\sigma_{0}(x), \sigma_{1}(x), \ldots, )$
%sending $x$ with $\underline{a}$-components $(x_0, x_1, \ldots, )$ to 
%$\sigma(x)$ with $a$-components $(\sigma_{0}(x), \sigma_{1}(x), \ldots, )$, 
%
%\mapsto \sigma(x)= (\sigma_{0}(x), \sigma_{1}(x), \ldots, )\]
%is a section of $\iE$ if and only if, for each $k$, the $k$-the component 
%\[ \sigma_k= \sigma_k(x_0, x_1, \ldots)\in \iE_k \quad (\textrm{a section over $|\iM|$ of the pull-back of $\iE_k$})\]
%depends only on a finite numbers of entries: there exists $m_k\geq k$ large enough and 
%\[ s_k\in \Gamma(\iM_{m_k}, \pi^*\iE_k)\]
%(where $\pi: \iM_{m_k}\to \iM_k$ is the tower projection) such that 
%\[ \sigma_k(x)= s_k(x_{n_k}) \quad \forall\ x\in |\iM|.\]

Note that the last condition (existence of $s_k$) can also be required only locally - and that gives rise to what we would call ``smooth sections of $|\iE|$", and resulting (strict) inclusions
\begin{equation}
\label{Aeq:inclusion-smooth-E} 
\Gamma(\iE)\subset  \Gamma_{\textrm{smooth}}(|\iE|) \subset \Gamma_{\textrm{cont}}(|\iE|).
\end{equation}
\end{rmk}

\begin{exm}\label{ex:pf-vbdle:JkE} In Example \ref{exam-Jinfty-R}, if $\pi: E\to M$ is a vector bundle (everything is finite dimensional), then each $J^kE$ is a vector bundle over $M$ and then 
\[ \pi_{\infty}: J^{\infty}E\to M\]
becomes a pf-vector bundle (for $M$ see Example \ref{ex:manifds-are-pf}). 
\end{exm}

\begin{rmk}[the structure of $\Gamma(\iE)$] The immediate structure on $\Gamma(\iE)$, the space of sections of a vector bundle $\iE\to \iM$ over a pf-finite manifold, is the algebraic one: it is a vector space, and a module over the ring $C^{\infty}(\iM)$ of smooth functions. 

On the other hand, each vector bundle atlas $a: \iE\to \iE_{\bullet}$ gives rise to a decreasing filtration of $\Gamma(\iE)$:
\[ \ldots \subset \mathfrak{f}_{2} \subset \mathfrak{f}_{1}\subset \mathfrak{f}_{0}\subset \mathfrak{f}_{-1}= \Gamma(\iE),\]
where $\mathfrak{f}_{k}$ consists of those sections whose $k$-th component vanishes. To obtain a structure that is independent of the atlas, one considers the induced, $\mathfrak{f}_{\bullet}$-adic topology on $\Gamma(\iE)$, that is: the linear topology for which $\mathfrak{f}_{k}$s form a basis of neighborhoods of the origin. It is a simple exercise to check that this topology does not depend on the choice of the atlas. 

In conclusion, $\Gamma(\iE)$ is a topological vector space and $C^{\infty}(\iM)$-module. Of course, one could also say that  $\Gamma(\iE)$ is a topological $C^{\infty}(\iM)$-module, where $C^{\infty}(\iM)$ is endowed with the discrete topology. 
\end{rmk}

%%%%%%%%%%%%%%%%%%%%
%%%%%%%%%%%%%%%%%%%%
%%%%%%%%%%%%%%%%%%%%
%%%%%%%%%%%%%%%%%%%%
%%%%%%%%%%%%%%%%%%%%
\subsection{The tangent bundle and vector fields}
%%%%%%%%%%%%%%%%%%%%
%%%%%%%%%%%%%%%%%%%%
%%%%%%%%%%%%%%%%%%%%
%%%%%%%%%%%%%%%%%%%%
%%%%%%%%%%%%%%%%%%%%
%%%%%%%%%%%%%%%%%%%%

There are two basic examples of vector bundles to have in mind. First of all the trivial line bundle which, where we recover the notion of smoothness and the sequence of inclusions (\ref{Aeq:inclusion-smooth-E}) becomes (\ref{Aeq:inclusion-smooth}).
% gives the resulting notions of smooth functions on any tower $\iM_{\bullet}$ and on the associated spaces $|\iM_{\bullet}|$, denoted 
%\[ C^{\infty}(\iM_{\bullet})\subset C^{\infty}(|\iM_{\bullet}|).\]
%The difference between the two is that, while for the first space, functions $f= f(x_{\bullet})$ are required to depend only on some finite coordinate $x_k$, the condition is required only locally for the second space.

The other basic example is the tangent bundle of a pf-manifold $\iM$. For an intrinsic approach one defines $T_x\iM$ as the space of derivations at $p$, $\delta_p: C^{\infty}(\iM)\to \R$, and then one sets $T\iM= \cup_{x} T_x\iM$. It follows that, for any pf-atlas $a: \iM\to \iM_{\bullet}$, one obtains a 
pf-atlas $T\iM\to T\iM_{\bullet}$, where the tangent tower of $\iM_{\bullet}$ is, of course, 
\[ T\iM_{\bullet}: \quad 
 \ldots \longrightarrow T\iM_k\stackrel{d\pi_k}{\longrightarrow} T\iM_{k-1}\longrightarrow \ldots \longrightarrow T\iM_{2} \stackrel{d\pi_2}{\longrightarrow}
T\iM_{1} \stackrel{d\pi_1}{\longrightarrow} T\iM_{0}.\]
The fact that $T\iM= \lim\limits_{\longleftarrow} T\iM_k$ is shown in \cite{PFLAUM} and can also serve as definition of $T\iM$ (a less intrinsic but perhaps more practical one). 

The tangent bundle gives rise to the  notion of smooth vector fields of $\iM$ and $|\iM|$:
\[ \mathfrak{X}(\iM)\subset \mathfrak{X}(|\iM|).\]
As a particular case of sections of vector bundles, $\mathfrak{X}(\iM)$ is a topological module over $C^{\infty}(\iM)$. % (but not a module over $C^{\infty}_{\textrm{smooth}}(|\iM|)$). 
As in the standard situation, vector fields act on smooth functions; and, using \cite{PFLAUM}, Thm 3.26, one finds 
\[ \mathfrak{X}(\iM)\cong \textrm{Der}(C^{\infty}(\iM)), \quad
\mathfrak{X}(|\iM|)\cong \textrm{Der}(C^{\infty}(|\iM|)).\]
In particular, vector fields still come endowed a Lie bracket satisfying the standard Leibniz identity. 
% \begin{exm} $J^{\infty}\G$ and $J^{\infty}A$ (with representations?). 
% \end{exm}
With the notion of tangent bundle and the Lie bracket of vector fields at hand, one can talk about distributions and involutivity in pf-manifolds $\iM$, Ehresmann connections, etc. 
%
%\begin{defn} Given a pf-manifold $M$, an involutive distribution on $\iM$ is any vector sub-bundle $C$ of $T\iM$ with the property that its space of sections is closed under the Lie bracket.
%\end{defn}

\begin{exm}[the Cartan distribution] \label{exm:Car-distr-J}
We continue with the notations from Example \ref{exam-Jinfty-R}, for a fibration $\pi: R\to M$. 
Any section $\sigma\in \Gamma(R)$ induces a section $j^k\sigma$ of $J^kR\to M$, 
and sections of $J^kR$ of this type are called holonomic. The Cartan distribution 
\[ \CC_k \subset T J^kR \]
is defined so that it detects which sections of $J^kR$ are holonomic: they must be tangent to $\CC_k$. 
Actually one can force this property into the definition and define $\CC_k$ 
as the space of vectors that are tangent to such holonomic sections (another description using 1-forms will be
recalled below). All together define 
\[ \CC_{\infty}\subset T J^{\infty} R\]
as a distribution on the pf manifold; unlike its finite versions $\CC_k$, $\CC_{\infty}$ is involutive.
%  (however, due to the infinite 
\end{exm}

\begin{exm} And here is an interesting example of pf-connection. Consider the pf-vector bundle $J^{\infty}E\to M$ 
associated to a vector bundle $E\to M$ as in Example \ref{ex:pf-vbdle:JkE}. It is well-known that $J^{\infty}E$ carries a canonical connection 
\begin{equation} 
\label{eq:nabla-Cartan}
\nabla: \mathfrak{X}(M)\times \Gamma(J^{\infty}E)\to \Gamma(J^{\infty}E)
\end{equation}
which can be used, again, to detect the holonomic sections as the ones that are flat w.r.t. $\nabla$. 
Actually, as with the Cartan distribution, this property can be used to force the definition of $\nabla$. It is interesting to 
remark that $\nabla$ is of a true pf-nature: at each level $k$ it does not descend to a connection on $J^kE$ but to an operator
\begin{equation} 
\label{eq:nabla-Cartan-k} 
\nabla^k:  \mathfrak{X}(M)\times \Gamma(J^{k}E)\to \Gamma(J^{k-1}E)
\end{equation}
which is a connection ``relative to $l: J^kE\to J^{k-1}E$" in the sense that the Leibniz identity takes the form
\[ \nabla^k_X(f\sigma)= f\nabla^k_X(\sigma)+ L_{X}(f) l(\sigma).\]
Similar to the involutivity of the Cartan distribution, $\nabla$ is a flat connection. In the literature, it is known as the {\bf Spencer operator} on $J^\infty E$. The list of literature concerning Spencer operators is extensive; see, for example, \cite{SPENCER}.
\end{exm}

%
%. For a finite dimensional manifold $M$ we consider 
%$E= TM$ in Example \ref{ex:pf-vbdle:JkE} and the resulting pf-vector bundle $J^{\infty}TM\to M$. It is well-known that this carries a canonical connection 
%\[ \nabla: \mathfrak{X}(M)\times \Gamma(J^{\infty}TM)\to \Gamma(J^{\infty}TM)\]
%which can be used, again, to detect the holonomic sections: as the ones that are flat w.r.t. $\nabla$. 
%And, again, this property can be used to force the definition of $\nabla$. It is interesting to 
%remark that $\nabla$ is of a true pf-nature: at each level $k$ it does not descend to a connection on $J^kTM$ but to an operator
%\[ \nabla^k:  \mathfrak{X}(M)\times \Gamma(J^{k}TM)\to \Gamma(J^{k-1}TM).\]
%
%descends to an operator 
%
%
%
%which, at the finite step $k$, becomes an operator 
%\[ \nabla^k:  \mathfrak{X}(M)\times \Gamma(J^{k}TM)\to \Gamma(J^{k-1}TM)\]
%
%  
%\end{exm}
%%%%%%%%%%%%%%%%%%%%
%%%%%%%%%%%%%%%%%%%%
%%%%%%%%%%%%%%%%%%%%
%%%%%%%%%%%%%%%%%%%%
%%%%%%%%%%%%%%%%%%%%
\subsection{Differential forms}
%%%%%%%%%%%%%%%%%%%%
%%%%%%%%%%%%%%%%%%%%
%%%%%%%%%%%%%%%%%%%%
%%%%%%%%%%%%%%%%%%%%
%%%%%%%%%%%%%%%%%%%%
%%%%%%%%%%%%%%%%%%%%
We now move to differential forms. We are interested in the general case with coefficients. 

\begin{defn} Let $\iM$ be a pf-manifold and let $\iE$ be a pf-vector bundle over $\iM$. A $q$-differential form on $\iM$ with coefficients in $\iE$  is any $q$-alternating, $C^{\infty}(\iM)$-multilinear map
\begin{equation}\label{eq:pro-form-in-E} 
\omega: \underbrace{\mathfrak{X}(\iM)\times \ldots \mathfrak{X}(\iM)}_{q\ \textrm{times}} \to \Gamma(\iE)
\end{equation}
which is continuous w.r.t. the natural ($\mathfrak{f}$-adic) topology on $\mathfrak{X}(\iM)$ and $\Gamma(\iE)$. We denote by $\Omega^q(\iM, \iE)$ the space of such. When $\iE$ is the trivial line bundle, we talk about $q$-forms on $\iM$ and we use the notation $\Omega^q(\iM)$. 
\end{defn}
\begin{rmk}
Note that, according to this definition, the continuity of $q$-forms 
\begin{equation}\label{eq:pro-form-in-R} 
\omega: \underbrace{\mathfrak{X}(\iM)\times \ldots \mathfrak{X}(\iM)}_{q\ \textrm{times}} \to C^{\infty}(\iM)
\end{equation}
appeal to the discrete topology on $C^{\infty}(\iM)$. Without thinking at general coefficients, that may seem odd, hence it may be interesting to point out that the actual topology on smooth functions does not have such a big influence: any other topology that is linear (i.e. makes the vector space operations continuous) and for which the origin has a neighborhood that does not contain any line, would produce the same result. 
\end{rmk}

%
%
%\begin{defn} A $q$-differential form on a profinite manifold $\iM$ is any 
%$q$-alternating, $C^{\infty}(\iM)$-multilinear maps
%\[ \omega: \underbrace{\mathfrak{X}(\iM)\times \ldots \mathfrak{X}(\iM)}_{q\ \textrm{times}} \to C^{\infty}(\iM),\]
%which is continuous w.r.t. the natural ($\mathfrak{f}$-adic) topology on $\mathfrak{X}(\iM)$.
%\end{defn}
%
%Here, for simplicity, we will endow $C^{\infty}(\iM)$ with the discrete topology, but one could use any of the other standard topologies that makes the vector space operations continuous (all that matters is that the origin has a neighborhood that does not contain any line). 
%% other topology as it has no influence on the discussion below.
%
%
%
%Let us look at 
%% We will be describing them as certain 
%$q$-alternating, $C^{\infty}(\iM)$-multilinear maps
%\[ \omega: \underbrace{\mathfrak{X}(\iM)\times \ldots \mathfrak{X}(\iM)}_{q\ \textrm{times}} \to C^{\infty}(\iM).\]
%Recall that $\mathfrak{X}(\iM)$ has a natural ($\mathfrak{f}$-adic) topology. Therefore, we will require $\omega$ to be continuous; here, for simplicity, we will endow $C^{\infty}(\iM)$ with the discrete topology, but one could use any other topology as it has no influence on the discussion below. The resulting space will be denoted 

Let us concentrate on forms with trivial coefficients $\mathbb{R}$. For a pointwise description, we will be looking at topological vector spaces $V$ and at $q$-alternating, $\R$-multilinear, continuous functionals on $V$
\[ \mathfrak{alt}^q(V):= \{\omega: \underbrace{V\times \ldots \times V}_{q\ \textrm{times}} \to \R: \textrm{alternating, $\R$-multilinear, continuous}\}\]
% \[ \omega: \underbrace{V\times \ldots \times V}_{q\ \textrm{times}} \to \R \]
(where, again, $\R$ will be endowed with the discrete topology). 
% We will denote by $\mathfrak{alt}^q(V)$ the space of such maps. 
Of course, we will be applying this to the tangent spaces $T_x\iM$ of a pf-manifold $\iM$ with their natural $\mathfrak{f}$-topology. 
\begin{rmk}
When $V$ is finite dimensional endowed with the discrete topology, $\mathfrak{alt}^q(V)$ is simply $\Lambda^qV^*$ - but this is a notation that may be confusing if used in the infinite dimensional case (note that we are not even interested in making sense of $\Lambda^qW$ for a general pf-space $W$, and then apply it to some topological dual $V^*$ of $V$). 
\end{rmk}

\begin{rmk} When looking at topological vector spaces $V$ whose topology can be described as the $\mathfrak{f}$-adic topology induced by a filtration 
 \[ \ldots \subset \mathfrak{f}_{2} \subset \mathfrak{f}_{1}\subset \mathfrak{f}_{0}\subset \mathfrak{f}_{-1}= V\]
consisting of vector subspaces of finite codimension, then the continuity of a form $\omega$ on $V$ is equivalent to the fact that $\omega$ factors through one of the (finite dimensional) quotients:
\[ V_k:= V/\mathfrak{f}_k, \]
In other words, considering the sequence of inclusions induced by the canonical projections, one has:
\[ \mathfrak{alt}^q(V_0)\hookrightarrow \mathfrak{alt}^q(V_1)\hookrightarrow \ldots \hookrightarrow 
\cup_{k} \mathfrak{alt}^q(V_k)= \mathfrak{alt}^q(V)\]
where each $V_k$ is finite dimensional hence $\mathfrak{alt}^q(V_k)$ is the standard $\Lambda^qV_k^*$. 
%
% Considering the corresponding quotients 
%\[ V_k:= V/\mathfrak{f}_k, \]
%the induced topology is the discrete one and the continuity of a form $\omega$ on $V$ is equivalent to the fact that $\omega$ factors through some $V_k$. In other words, considering the sequence of inclusions induced by the canonical projections, one has:
%\[ \mathfrak{alt}^q(V_0)\hookrightarrow \mathfrak{alt}^q(V_1)\hookrightarrow \ldots \hookrightarrow 
%\cup_{k} \mathfrak{alt}^q(V_k)= \mathfrak{alt}^q(V).\]
\end{rmk}
%
%\begin{prp} Let $\iM$ be a profinite manifold. Then $q$-alternating, $C^{\infty}(\iM)$-multilinear maps
%\[ \omega: \underbrace{\mathfrak{X}(\iM)\times \ldots \mathfrak{X}(\iM)}_{q\ \textrm{times}} \to C^{\infty}(\iM)\]
%gives rise to a family of alternating functionals 
%\[ \omega_x\in \mathfrak{alt}^q(T_x\iM),\]
%uniquely characterised by the fact that, for any $X^1, \ldots, X^q\in \mathfrak{X}(\iM)$, 
%\[ \omega(X^1, \ldots, X^q): \iM\to \R, \quad \textrm{is precisely}\ x\mapsto \omega_x(X^1_x, \ldots, X^q_x).\]
%%
%%$\omega(X^1, \ldots, X^q)$ coincides with the 
%%\[ \omega(X^1, \ldots, X^q)(x)= \omega_x(X^1_x, \ldots, X^q_x)
%%for all $X^1, \ldots, X^q\in \mathfrak{X}(\iM)$.
%%
%%\[ \omega(X^1, \ldots, X^q): \iM\to \R, \quad x\mapsto \omega_x(X^1_x, \ldots, X^q_x)
%\end{prp}

\begin{prp} Let $\iM$ be a pf-manifold. Then:
\begin{enumerate}
\item[(a)] any $\omega\in \Omega^q(\iM)$ is pointwise, i.e. there 
exists a family 
\[ \{\omega_x\in \mathfrak{alt}^q(T_x\iM):  x\in \iM\},\]
inducing it in the sense that 
\[ \omega(X^1, \ldots, X^q)(x)= \omega_x(X^1_x, \ldots, X^q_x)\]
for all $X^1, \ldots, X^q\in \mathfrak{X}(\iM)$, $x\in \iM$. 
\item[(b)] if $a: \iM\to \iM_{\bullet}$ is a pf-atlas for $\iM$ then 
any usual differential form at some level $k$, $\eta\in \Omega^q(\iM_k)$, 
can be promoted to a form on $\iM$, $\tilde{\eta}\in \Omega^q(\iM)$, 
% $a_{k}^{*}(\eta)\in \Omega^q(\iM)$, 
uniquely characterised by:
\[ \tilde{\eta}_x= \eta_{a_{k}(x)}\circ \left((d a_k)_x, \ldots, (d a_k)_x\right)\]
\end{enumerate}
Moreover, the resulting forms exhaust all the differential forms on $\iM$:
\[ \Omega^q(\iM_0)\hookrightarrow \Omega^q(\iM_1)\hookrightarrow \ldots \hookrightarrow 
\bigcup_{k} \Omega^q(\iM_k)= \Omega^q(\iM).\]
\end{prp}

In this statement, the ``inclusions" $\Omega^q(\iM_{k-1})\hookrightarrow \Omega^q(\iM_{k})$ are, of course, induced by the projection maps $\pi: \iM_k\to \iM_{k-1}$, but we will omit $\pi^*$ from the notation. Also, in the previous construction, $\tilde{\eta}$ deserves the name $a_{k}^{*}(\eta)$ but, again, in what follows we will be omitting $a_{k}^{*}$ (as well as the tilde) from the notation. 

\begin{proof} First note that any $v\in T_x\iM$ can be written as $v= X_x$ for some $X\in \mathfrak{X}(\iM)$. To see this, at each level $k$ we look at $v_k= (da_k)_x(v)\in T_{x_k}\iM_k$ and we consider $X_k\in \mathfrak{X}(\iM_k)$ extending $v_k$. However, one can proceed inductively and, at each stage, arrange that $X_k$ is projectable to $X_{k-1}$. Then the $X_k$'s together define a vector field $X$ with the desired properties (see also \cite{PFLAUM}, Lemma 3.15).

This implies the uniqueness of each $\omega_{x}$. For the existence we are left with checking that $\omega(X, -, \ldots, -)(x)= 0$ whenever $X$ vanishes at $x$. Using an atlas $a: \iM\to \iM_{\bullet}$, since $\omega$ is continuous, it follows that it vanishes on some $\mathfrak{f}_k$, for some $k$.
Using $C^{\infty}(\iM)$-linearity it suffices to show that 
\[ X\equiv  f\cdot Y \ \textrm{mod $\mathfrak{f}_k$, for some $f\in C^{\infty}(\iM)$ vanishing at $x$, $Y\in \mathfrak{X}(\iM)$}.\]
 For that: look at the $k$-th component $X_k\in \Gamma(\iM_{m_k}, \pi^*T\iM_k)$ and write it as $f\cdot Y_k$ for some section $Y_k$ of the same bundle and $f\in C^{\infty}(\iM_k)$ vanishing at $a_k(x)$; then extend $Y_k$ to a $Y\in \mathfrak{X}(\iM)$ and also interpret $f$ as a smooth function on $\iM$. 
% In this way we managed to write 
%% \[ X\cong f\cdot Y \ \textrm{mod $\mathfrak{f}_k$, with $f\in C^{\infty}(\iM)$ vanishing at $x_0$}.\]
%Using $C^{\infty}(\iM)$-linearity we obtained the desired conclusion:  $\omega(X, -, \ldots, -)(x^0)= 0$.
%
For (b) we have to check that the resulting $\tilde{\eta}$ is continuous, i.e. that it vanishes on some of the members of the filtration; but it clearly vanishes on $\mathfrak{f}_k$. 

Finally, to see that any $\omega\in \Omega^q(\iM)$ arises as some $\tilde{\eta}$ (living at some level $k$) choose, of course, the $k$ so that $\omega$ vanishes on $\mathfrak{f}_k$. And, by the same arguments as above, note that any $X_k\in \mathfrak{X}(\iM_k)$ can be realised as the $k$-th component of some $X\in \mathfrak{X}(\iM)$ (of course unique modulo $\mathfrak{f}_k$). 
\end{proof}

Note that all the usual operations and formulas valid for differential forms now make sense in the pf-setting. The ones that are relevant for us in the main body are: wedge products, interior products, Cartan magic formula, DeRham differential and the Koszul formula. 

Returning now to forms with general coefficients $\iE$, it is pretty straightforward to extend the previous discussion. The bottom line will be that the continuity condition for a form (\ref{eq:pro-form-in-E}) translates into the fact that, for any pf-vector bundle atlas $\tilde{a}: \iE\to \iE_{\bullet}$, and for any $k$, 
the $k$-component $\omega_k$ (taking values in $\Gamma(\iE_k)$), comes from a form at a finite step $m_k$
\[ \omega_k\in \Omega^q(\iM_{m_k}, \iE_{k}).\]
More formally, 
\[ \Omega^q(\iM_{\bullet}, \iE_{\bullet})= \plim_{k} \ilim_{m\geq k} \Omega^q(\iM_{m}, 
\pi^*\iE_k)\]
%\pi_{n, m}^*\iE_m)\]
where, of course, $\pi= \pi_{m, k}: \iM_m\to \iM_k$ is the tower projection.

\begin{exm}
As in the finite dimensional setting, connections $\nabla$ on $\iE$ can be interpreted as operators 
\[ d_{\nabla}: \Gamma(\iE)\to \Omega^1(\iM, \iE),  \]
and then they are in 1-1 correspondence with operators on 
% operators $D$ on 
$\Omega^*(\iM, \iE)$ satisfying the Leibniz identity. Here all the objects are in the pf-sense hence, in principle, also the connections do not have to come from connections on vector bundles $\iE_k$. That is a phenomenon that we have seen already in the case of the standard connection (\ref{eq:nabla-Cartan}) on the infinite jet space $J^{\infty}E$. Note that, in that example, $d_{\nabla}$ is reflected at level $k$ as
\[ d_{\nabla^k}: \Gamma(J^kE)\to \Omega^1(M, J^{k-1}E).\]
\end{exm}

\begin{exm} [Cartan forms]\label{ex:Cartan-form-R}
Let us return to a fibration $\pi: R\to M$ as in Example \ref{exam-Jinfty-R} and we recall the description of the Cartan distribution (see Example \ref{exm:Car-distr-J}) as the kernel of a 1-form, the so-called Cartan form
\[ \omega\in \Omega^1(J^\infty R; \mathcal{V}).\]
First of all, $\mathcal{V}$ is the vector bundle over $J^\infty R$ which, at level $k$, is the 
%vector bundle over $J^kR$ that is the vertical part of $J^kR\to M$:
vertical bundle of $J^kR\to M$:
\[ \mathcal{V}^k:= T^{\textrm{vert}} J^kR.\]
With this, the components of $\omega$ will be of type
\[ \omega^k\in \Omega^1(J^{k}R; \pi^*\mathcal{V}^{k-1})\]
(hence again of pf-nature!), where 
%the last 
$\pi$ is the projection $\pi_{k, k-1}$ from $J^{k}R$ to $J^{k-1}R$. 
%{\color{magenta}
%Explicitly, $\omega^k$ is given by ...
%\[ \]
%}
%Note that, 
In particular, for the jet spaces of type $J^k(M, N)$ (see Example \ref{exam-Jinfty-R})
$\omega^k$ is given by 
\begin{equation}\label{Cartan_form_def-Luca}
\omega^k_{j^k_xf}=d_{j^k_xf}(p^{k,k-1})-d_x(j^{k-1}f)\circ d_{j^k_xf}(s)
\end{equation}
\end{exm}

Finally, note that the differential forms in the previous sense are not part of the sheaf on $\iM$ viewed as a topological space (which we denote by $|\iM|$) but rather of a pf-sheaf on $\iM$ viewed as a pf-manifold. Of course, 
\begin{equation}\label{eq:omegaq-presheaf} 
\mathcal{O}p(|\iM|)\ni U\mapsto \Omega^q(U) 
\end{equation}
is a preseheaf. 

\begin{defn} \label{def-omegaq-|M|} The sheaf of $q$-forms on $|\iM|$, denoted $\Omega^q_{|\iM|}$, is the sheaf on $|\iM|$ associated to the sheaf (\ref{eq:omegaq-presheaf}). The space of its global sections is denoted
$\Omega^q(|\iM|)$.  Similarly for forms with coefficients. 
\end{defn}

While the elements of $\Omega^q(\iM)$ are $q$-alternating, $C^{\infty}(\iM)$-multilinear maps
(\ref{eq:pro-form-in-R}) which are required to come from a $q$-form on some $\iM_k$, the elements of $\Omega^q(|\iM|)$ can be described similarly, but imposing the same condition only locally. 

%{\color{magenta}
%%%%%%%%%%%%%%%%%%%%
%%%%%%%%%%%%%%%%%%%%
%%%%%%%%%%%%%%%%%%%%
%%%%%%%%%%%%%%%%%%%%
%%%%%%%%%%%%%%%%%%%%
\subsection{Inverse pf-sheaves (and vector bundles)}
%\footnote{I am not sure whether we should keep this or not. I wrote it for those ``specialists" that think they know what sheaves are ...  This originates in what you (Luca) wrote about inverse vector bundles and explains what sheaves on pf-manifolds really are; they induce sheaves on the underlying topological space ...  but they are something else.}
%%%%%%%%%%%%%%%%%%%%
%%%%%%%%%%%%%%%%%%%%
%%%%%%%%%%%%%%%%%%%%
%%%%%%%%%%%%%%%%%%%%
%%%%%%%%%%%%%%%%%%%%
%%%%%%%%%%%%%%%%%%%%
We now move on and 
discuss the other type of vector bundles and sheaves on pf-manifolds: the inverse ones. Here, differently from the case of direct vector bundles/sheaves, the discussion for sheaves is simpler. Because of this and given the examples that we have in mind, we concentrate on sheaves first. In Example~\ref{exm:relation_direct_inverse} we will clarify the relation between direct and inverse vector bundles/sheaves.

\begin{defn} \label{defn:sheaves-towers}
A sheaf $\mathcal{S}_{\bullet}$ on a tower $\iM_{\bullet}$ is a sequence of sheaves $\mathcal{S}_{k}$ on $\iM_k$ together
with sheaf morphisms $\phi_k: \pi_{k}^{*} \mathcal{S}_{k-1}\to \mathcal{S}_{k}$ for each $k$.

Given a second sheaf 
% vector bundle $\mathcal{S}_{\bullet}$ on a tower $\iM_{\bullet}$ and a sheaf 
$\mathcal{S}_{\bullet}^{'}$ on another tower $\iN_{\bullet}$, a concrete morphism of sheaves  $F_{\bullet}: \mathcal{S}_{\bullet}\to \mathcal{S}_{\bullet}^{'}$ consists of a strict morphism of towers
\[ f_{\bullet}= \{f_k:\iM_{m_k}\to \iN_{k}\}_{k\geq 1}\quad   \textrm{with $m_1< m_2< \ldots$}\]
together with morphisms of sheaves over $\iM_{m_k}$ 
\[ F_{k}: f_{k}^{*}\mathcal{S}_{k}^{'} \to \mathcal{S}_{m_k}\]
satisfying the coherence condition (for the notations, see the explanation below).  
%\[\xymatrix{
%\mathcal{S}_{m_k} & & & \pi^*\mathcal{S}_{m_{k-1}}\ar[lll]_{\phi}\\
%f_{k}^*\mathcal{S}_{k}^{'} \ar[u]^{F_k} & & f_{k}^{*}\pi^{*}\mathcal{S}_{k-1}^{'} \ar[ll]_{f_{k}^{*}\phi'}\ar[r]^{=} & \pi^*f_{k-1}^{*}\mathcal{S}_{k-1}^{'} \ar[u]_{\pi^*F_{k-1}}
%}\]
\[\xymatrix{
\mathcal{S}_{m_k} & & & \pi^*\mathcal{S}_{m_{k-1}}\ar[lll]_{\phi}\\
f_{k}^*\mathcal{S}_{k}^{'} \ar[u]^{F_k} & & f_{k}^{*}\pi^{*}\mathcal{S}_{k-1}^{'} \ar[ll]_{f_{k}^{*}\phi'}\ar@{}[r]|*=0[@]{=} & \pi^*f_{k-1}^{*}\mathcal{S}_{k-1}^{'} \ar[u]_{\pi^*F_{k-1}}
}\]

%(where the missing indices for $\phi$, $\phi'$ and the tower projections $\pi$ should be clear by looking at the domains of the arrows). 

Given another concrete morphism of sheaves $G_{\bullet}: \mathcal{S}_{\bullet}\to \mathcal{S}_{\bullet}^{'}$, with 
% corresponding 
concrete morphism of towers 
\[ g_{\bullet}= \{g_k:\iM_{l_k}\to \iN_{k}\}_{k\geq 1}\quad   \textrm{with $l_1< l_2< \ldots$},\]
we say that $F_{\bullet}$ and $G_{\bullet}$ are equivalent if for each $k$ there exists $K$ larger than $n_k$ and $l_k$ such that, as in the definition of equivalence of strict morphisms of towers, one has a commutative diagram
\[ \xymatrix{
& \iM_K \ar[ld]_{\pi} \ar[rd]^{\pi} & \\
\iM_{m_k}  \ar[rd]^{f} & & \iM_{l_k}\ar[ld]_{g} \\
& \iN_k & 
}\]
as well as a commutative diagram of sheaves over $\iM_{K}$:
\[ \xymatrix{
& \mathcal{S}_{K}  & \\
\pi^* \mathcal{S}_{m_k}\ar[ru]^{\phi}  & & \pi^*\mathcal{S}_{l_k}\ar[lu]_{\phi}  \\
& \pi^*f^*\mathcal{S}_{k}^{'}= \pi^*g^*\mathcal{S}_{k}^{'}\ar[lu]_{\pi^*F}\ar[ru]^{\pi^*G} & 
}\]

Finally, a morphism between two sheaves (over towers) is an equivalence class of concrete morphisms. 
\end{defn}

As a convention for our diagrams: 
\begin{itemize}
\item for notational simplicity we omit the missing indices for $\phi$, $\phi'$ and the tower projections $\pi$; on the diagram they should be clear by looking at the domains of the arrows.
\item we also denote by $\pi$ and $\phi$ the composition of consecutive $\pi_k$'s and $\phi_k$, respectively. 
\end{itemize}

And now we move to pf-manifolds. 

\begin{defn} \label{defn:sheaves-pro-manifolds} Let $\iM$ be a pf-manifold and let 
% $\mathcal{S}$
$\mathcal{S}_{|\iM|}$
 be a sheaf on the topological space 
 % (underlying) $\iM$.
 $|\iM|$ (underlying  $\iM$). 
 A pf-atlas for $\mathcal{S}_{|\iM|}$ is a pf-atlas $a: \iM\to \iM_{\bullet}$ of $\iM$, together with a sheaf $(\mathcal{S}_{\bullet}, \phi_{\bullet})$ on the tower $\iM_{\bullet}$ and an isomorphism of sheaves over $|\iM|$:
\begin{equation}\label{ind-limit-sheaf} 
\tilde{a}: \ilim a_{k}^{*} \mathcal{S}_k\cong \mathcal{S}_{|\iM|}.
\end{equation}
Two such pf-atlases $\tilde{a}$ and $\tilde{a}'$ are said to be {\bf equivalent} (as pf-atlases for $\mathcal{S}_{|\iM|}$) if they are equivalent as sheaves of towers, by isomorphisms that are compatible with $\tilde{a}$ and $\tilde{a}'$. 

Finally, a {\bf pf-sheaf} $\mathcal{S}$ on $\iM$, is a sheaf $\mathcal{S}_{|\iM|}$ over $|\iM|$ together with an equivalence class of pf-atlases. % for $\mathcal{S}_{|\iM|}$. 
\end{defn}

Note that the isomorphism (\ref{ind-limit-sheaf}), hence also the inductive limit there, is within the category of sheaves; and the connecting maps are induced by the $\phi_k$'s. This means that we have morphisms of sheaves 
\[ \tilde{a}_k: a_{k}^{*} \mathcal{S}_k\to \mathcal{S}_{|\iM|}\]
which are compatible with the $\phi_k$s:
\[ \xymatrix{
 a_{k}^{*} \mathcal{S}_k \ar[dr]^{\tilde{a}_k}  & \\
  &  \mathcal{S}_{|\iM|}  \\
 a_{k}^{*}\pi_{k}^{*} \mathcal{S}_{k-1}= 
a_{k-1}^{*} \mathcal{S}_{k-1} \ar[uu]^{a_{k}^{*}\phi_k}\ar[ur]_{\tilde{a}_{k-1}} & 
}  .\]
Furthermore, $\mathcal{S}_{|\iM|}$ is universal among sheaves equipped with such diagrams. The fact that $\tilde{a}= \ilim \tilde{a}_k$ is an isomorphism of sheaves boils down to a similar stalk-wise statement.  
% \[ \tilde{a}_x:  % \ilim a_{k}^{*} \mathcal{S}_{k, x}= 
% \ilim \mathcal{S}_{k, a_k(x)}\cong \mathcal{S}_x.\]
However, such isomorphisms do not hold at the level of sections; what happens there is the following:
\begin{itemize}
\item the spaces $\Gamma(\iM_k, \mathcal{S}_k)$ form an inductive system
%  with connecting maps induced by $\pi$s and $\phi$s:
\[ \phi: \Gamma(\iM_k, \mathcal{S}_k)\to \Gamma(\iM_{k+1}, \mathcal{S}_{k+1}),\]
\item each $\Gamma(\iM_k, \mathcal{S}_k)$ injects into $\Gamma(|\iM|, \mathcal{S}_{|\iM|})$ via the map $\tilde{a}_k$ (and the base map $a_k$)
% via $\tilde{a}_k$ (and $a_k$)
\[ \tilde{a}_k: \Gamma(\iM_k, \mathcal{S}_k)\hookrightarrow \Gamma(|\iM|, \mathcal{S}_{|\iM|})\]
\end{itemize}
and, together, one obtains 
%\[ \ilim_k \Gamma(\iM_k, \mathcal{S}_k)\cong \bigcup_{k} \tilde{a}_k\Gamma(\iM_k, \mathcal{S}_k)\subset \Gamma(\iM, \mathcal{S}).\]
\[ \ilim_k \Gamma(\iM_k, \mathcal{S}_k)\cong \bigcup_{k} \tilde{a}_k\left( \Gamma(\iM_k, \mathcal{S}_k)\right) \subset \Gamma(|\iM|, \mathcal{S}_{|\iM|}).\]
It is not difficult that the resulting subspace of $\Gamma(\iM, \mathcal{S})$ does not depend on the atlas  
$\tilde{a}$ of $\mathcal{S}$.

\begin{defn} For any pf-sheaf $\mathcal{S}$ on a pf-manifold $\iM$, the resulting subspace of $\Gamma(|\iM|, \mathcal{S}_{|\iM|})$ 
% from the previous lemma 
is called the space of (pf-sections) of $\mathcal{S}$, denoted
\[ \Gamma(\iM, \mathcal{S})\subset \Gamma(|\iM|, \mathcal{S}_{|\iM|}).\]
Hence it consists of sections $s$ of $\mathcal{S}_{|\iM|}$ with the property that for a/any pf-atlas of $\mathcal{S}$, there exists a $k$ and a section $s_k\in \Gamma(\iM_k, \mathcal{S}_k)$ such that 
% \[ s= \tilde{a}_k(a_{k}^{*}(s_k)). \]
\[ s(x)= \tilde{a}_k(s_k(a_k(x))). \]
\end{defn}

%
%\begin{lem} The resulting subspace of $\Gamma(\iM, \mathcal{S})$ does not depend on the choice of the atlas $\mathcal{a}$ of $\mathcal{S}$.
%\end{lem}
%
%\begin{defn} For any pro-sheaf $\mathcal{S}$ on a pro-finite manifold $\iM$, the subspace of $\Gamma(\iM, \mathcal{S})$ from the previous lemma is called the space of pro-sections of $\mathcal{S}$, denoted
%\[ \Gamma_{\textrm{pro}}(\iM, \mathcal{S})\subset \Gamma(\iM, \mathcal{S}).\]
%\end{defn}
% 
%
%
% hold is that $\Gamma(\iM_k, \mathcal{S}_k)$ forms an inductive system with connecting maps induced by $\pi$s and $\phi$s:
%\[ \Gamma(\iM_k, \mathcal{S}_k)\to \Gamma(\iM_{k+1}, \mathcal{S}_{k+1})\]
%
%
%We consider the induced maps $\Gamma(\iM_k, \mathcal{S}_{k})\leftarrow \Gamma(\iM_{k-1}, \mathcal{S}_{k-1})$ and then the space of sections of $\mathcal{S}_{\bullet}$ as
%\[ \Gamma(\mathcal{S}_{\bullet}):= \ilim \Gamma(\iM_k, \mathcal{S}_{k}).\]
%

Note that the entire discussion about sheaves, starting with Definition \ref{defn:sheaves-towers} and \ref{defn:sheaves-pro-manifolds}, makes sense if we replace the word ``sheaf" by ``vector bundles". The resulting notions will be coined as {\bf inverse vector bundles} on towers and on pf-manifolds. 
%For more on the terminology, see Remark \ref{rk:curiosities-pro}. 

\begin{exm}
The basic example is, of course, the pf-sheaf of $q$-forms, denoted $\Omega^{q}_{\iM}$.
The underlying sheaf on $|\iM|$ is $\Omega^{q}_{|\iM|}$ from Definition \ref{def-omegaq-|M|} and pf-atlases are obtained from atlases $a: \iM\to \iM_{\bullet}$ of $\iM$, using $\Omega^q_{\iM_{k}}$ over each $\iM_k$ (and the required sheaf morphisms are given by pull-back of forms).
%
%
%
%On towers $\iM_{\bullet}$, for the corresponding $\Omega^{q}_{\iM_{\bullet}}$:  one uses 
%$\Omega^q_{\iM_{k}}$ over each $\iM_k$, 
%and the required sheaf morphisms are given by pull-back of forms. On a pro-manifold $\iM$, for the corresponding pro-sheaf $\Omega^{q}_{\iM}$: the 
%underlying sheaf on $|\iM|$ is $\Omega^{q}_{|\iM|}$ from Definition \ref{def-omegaq-|\iM|} and, for 
%any pro-atlas $a: \iM\to \iM_{\bullet}$ of $\iM$, one uses $\Omega^{q}_{\iM_{\bullet}}$ as an atlas for 
%$\Omega^{q}_{\iM}$. 
\end{exm}

\begin{exm} Any (inverse) pf-vector bundle $\iE$ on $\iM$
gives rise to a (inverse) pf-sheaf $\Gamma_{\iE}$ of sections.  
\end{exm}

\begin{exm}\label{exm:relation_direct_inverse} Any of the usual operations with vector bundles which are covariant in all arguments (e.g. direct sums, tensor products, tensor powers, etc) can be lifted right away to both (direct/inverse) types of pf-vector bundles. The situation is different for constructions that are contravariant in one of the terms- the basic example being that of taking duals. That is precisely the operation that changes direct vector bundles into inverse ones, and the other way around. 

In particular, we see that for any pf-manifold $\iM$, we can talk about $T^*\iM$, as well as $\Lambda^q T^*\iM$, as inverse vector bundles over $\iM$. Passing to sections (cf. the previous example), one recovers the sheaf of $q$-forms. 
\end{exm}

\subsection{Oids in the pro-finite context}
%%%%%%%%%%%%%%%%%%%%
%%%%%%%%%%%%%%%%%%%%
%%%%%%%%%%%%%%%%%%%%
%%%%%%%%%%%%%%%%%%%%
%%%%%%%%%%%%%%%%%%%%
%%%%%%%%%%%%%%%%%%%%
Most of the basic concepts on Lie groupoids and Lie algebroids can be adapted to the pro-finite setting in a manner that is rather straightforward, though occasionally tedious. Furthermore, given the examples we have in mind, and to avoid further complications, in the main body of the paper we will not even use some of the most general concepts. For instance, the most general concept of pf-groupoid $\Sigma\tto \iM$ comes with $\Sigma_{k}$s, $\iM_{k}$s and the structure maps (source/target, multiplication) are just pf-maps (hence may lower the index $k$). However, the examples that we encounter in the main body are most of the time over a finite dimensional base $\iM$ and they are always ``strict" in the following sense.

\begin{defn}\label{defn:strict-things} A {\bf strict Lie pf-groupoid} is a topological groupoid $\Sigma\tto \iM$, where $\Sigma$ and $\iM$ are pf-manifolds which admit strict pf-atlases, i.e. atlases $\tilde{a}: \Sigma\to \Sigma_{\bullet}$, $a: \iM\to \iM_{\bullet}$, where each $\Sigma_k$ is a Lie groupoid over $\iM_k$ with structure maps compatible with those of $\Sigma$ (towers of Lie groupoids). We say that $\Sigma$
\begin{itemize}
\item is {\bf proper} if all the $\Sigma_k\tto \iM_k$ may be chosen to be proper. 
\item has (cohomologically) {\bf $l$-connected fibers}, if all the $\Sigma_k\tto \iM_k$ may be chosen to have this property.
\end{itemize}

A Lie pf-{\bf subgroupoid} of $\Sigma$ is any subgroupoid $\Upsilon\subset \Sigma$ which is itself a (strict) Lie pf-groupoid, such that
$\Sigma$ admits a strict pf-atlas $\tilde{a}: \Sigma\to \Sigma_{\bullet}$ with the property that $\tilde{a}|_{\Upsilon}: \Upsilon\to \Upsilon_{\bullet}$ 
%
% := \tilde{a}_{\bullet}(H)$
 is a strict pf-atlas of $\Upsilon$, where each $\Upsilon_{k}:= \tilde{a}_k(\Upsilon)$ is a Lie subgroupoid of $\Sigma_k$.

Similarly, a {\bf strict action} of such a groupoid on a pf-manifold $P$ is a topological action of $\Sigma$ on $P$ along some pf-map $\mu: P\to \iM$, such that there exists atlases for $\Sigma$ and $\iM$ as above, and for $P$, such that the action factors through actions of $\Sigma_k$ on $P_k$. 
%Similarly we talk about {\bf strict principal $\G$-bundles}. And also about {\bf strict (direct) representations} of $\G$ on pf-vector bundles. 
We define {\bf strict principal $\Sigma$-bundles} and {\bf strict (direct/inverse) representations} of $\Sigma$ on pf-vector bundles in an analogous way. 
\end{defn}

%And, when it comes to c
Cochains on $\Sigma$ with values in a representation $\iE$ are defined as in the finite dimensional case, see (\ref{equation:Cp(G,E)}), except that we use sections in the pf-setting; hence $C^{*}(\Sigma, \iE)$ is a subcomplex of $C^{*}(|\Sigma|, |\iE|)$. The same is true when defining multiplicative forms.

Concerning pf-algebroids $A\to \iM$, the discussion is again quite similar, although here we will not restrict entirely to strict ones. Hence the Lie bracket on $\Gamma(A)$ may come from ``brackets" $\Gamma(A_k)\times \Gamma(A_k)\to \Gamma(A_{k-1})$ (hence the Jacobi identity makes sense only on the entire $\Gamma(A)$). Nevertheless, if $\Sigma\tto \iM$ is a strict pf-groupoid, one defined its Lie algebroid $A= \textrm{Lie}(\Sigma)$ precisely as in the finite dimensional case, and the outcome will be a strict pf-agebroid. Actually, it is not difficult to see that, using strict pf-atlases as in Definition \ref{defn:strict-things}, the Lie algebroids $A_k= \textrm{Lie}(\Sigma_k)\to \iM_k$ can be used as strict atlases for $A$ (in particular $A= \plim A_k$). 

On the other hand one has to be a bit careful with the Lie functor at the level of representations, $\textrm{Rep}(\Sigma)\to \textrm{Rep}(A)$, since the defining formula (\ref{eq:Lie:on:repres}) uses flows. For that one just remarks that that formula defining $\nabla_{\alpha}(\sigma)(x)$ can be described as the differential at $1_x$ of the map 
\[ s^{-1}(x)\to \iE_x, \quad g\mapsto g\cdot \sigma(t(g))\in \iE_x,\]
applied to $\alpha(x)$.

Here are two results that we use in the main body of the paper.

\begin{lem}\label{lem-App;proper} If $\Sigma\tto \iM$ is a proper strict pf-groupoid and $\iE$ is any pf-representation, then $H^{p}(\Sigma, \iE)= 0$ when $p\neq 0$, while when $p= 0$ it is the space of invariant sections $\Gamma(\iE)^{\Sigma}$. 
\end{lem}

\begin{proof} 
When the coefficients are trivial (or constant) we can just appeal to the similar finite dimensional result from~\cite{MARIUS} since, in that case, the cochains are the union of cochains defined at each step of the tower. This is also enough if the coefficent space is an inverse sheaf. For coefficients given by a direct vector bundle $E$ we have to be slightly more careful. We fix an atlas $\Sigma\to \Sigma_{\bullet}$ consisting of proper groupoid, with $E_k\to \iM_k$ representation of $\Sigma_k\tto \iM_k$. Let $c$ be a $p$-cochain with $p\neq 0$ such that $\delta(c)= 0$ and let $n_0< n_1< \ldots$ so that the $k$-th component of $c$ comes from $\Sigma_{n_k}$:
\[ c_k\in C^p(\Sigma_{n_k}, \iE_k).\]
Using the properness of $\Sigma_{n_k}$ we can write 
\[ c_k= \delta(u_k)\quad \textrm{with $u_k\in C^{p-1}(\Sigma_{n_k}, \iE_k)$}.\]
We just have to make sure that one can arrange that the two types of resulting $E_k$-valued cochains over $\Sigma_{n_{k+1}}$, 
\[ \pi^*(u_k)= u_k\circ (\pi, \ldots, \pi), \quad \pi\circ u_{k+1} \in C^{p-1}(\Sigma_{n_{k+1}}, \iE_k)\]
coincide (where the first $\pi$ is the one of the tower $\Sigma_{\bullet}$, while the second is of the tower $\iE_{\bullet}$). We proceed inductively. Once $u_k$ was chosen, the previous condition determines part of $u_{k+1}$. Actually, let us decompose $\iE_{k+1}= \tilde{E}_k\oplus V_k$, where $V_k$ is the kernel of the map $\iE_{k+1}\to \iE_k$, and 
$\tilde{E}_k$ is a complement which is $\Sigma_{n_{k+1}}$-invariant (the existence of such a complement is another consequence of properness- see \cite{RUIMATIAS}). The elements $v\in \iE_{k+1}$ will be decomposed as $v= v'+ v''$ according to this direct sum. Then, for the components of the element $u_{k+1}$ that we are looking for: 
$u_{k+1}'$ is uniquely determined by the condition $\pi(u_{k+1}')= \pi^*(u_k)$, while when choosing $u_{k+1}''$ the only condition comes from the requirement $\delta(u_{k+1})= c_{k+1}$, i.e. 
\[ \delta(u_{k+1}'')= c_{k+1}- \delta(u_{k+1}').\]
Note that the right hand side is a cocycle with values in $V_k$ (since $\delta(u_{k})= c_{k}$). Hence using again the result from the finite dimensional case (for the representation $V_k$ of $\Sigma_{n_{k+1}}$), one does find the desired $V_k$-valued cochain $u_{k+1}''$.
\end{proof}

The following is a particular case of Theorem 2 from \cite{MARIUS} (applied to foliations) extended to the pf-setting. 

\begin{prp}\label{lem:App:2nd}  Let $\pi: \iP\to \iM$ be a submersion between pf-manifolds. Let $\iC\subset T\iM$ be an involutive distribution, consider its inverse
\[ \widetilde{\iC}:= \pi^{!}(\iC):= (d\pi)^{-1}(\iC)\subset T\iP.\]
Assume that the fibers of $\pi$ are cohomologically $l$-connected in the pf-sense i.e. after passing to pf-atlases, 
$\pi$ can be represented by maps whose fibers are cohomologically $l$-connected in the usual sense. Then 
the map induced by pull-back in the $\iC$-cohomologies (the Lie algebroid cohomologies of the subalgebroids $\iC\hookrightarrow T\iM$ and $\widetilde{\iC}\hookrightarrow T\iP$)
\[ \pi^*: H^*(\iC) \to H^*(\widetilde{\iC}),\]
is an isomorphism up to degree $l$ and is injective in degree $l+1$.
\end{prp}

\begin{proof} 
In principle, we just follow the proof of Theorem 2 from \cite{MARIUS} in the particular case of foliations (second part of the proof there), taking care of the pf-version. As there, using the same construction of the spectral sequence, one reduces everything to proving that 
\[ H^i(\mathcal{F}_{\pi}, \iE)= 0\quad \textrm{for all $i\in \{1, \ldots, l\}$},\]
where 
\[ \mathcal{F}_{\pi}:= \textrm{Ker}(d\pi)\]
(another involutive distribution) and $\iE$ is any pf-vector bundle over $\iM$. Here the coefficients $\pi^*\iE$ are endowed with the canonical $\mathcal{F}_{\pi}$-connection, uniquely determined by the condition that pull-back sections are flat (and of relevance for us are the cases $\iE= \Lambda^q\mathcal{C}$).

As in the proof of the previous result, this follows in principle from its finite dimensional version. To see this, start with any cocycle $\omega\in C^i(\mathcal{F}_{\pi}, \pi^*E)$, look at its $k$-components $\omega^k$, each one of them comes from the complex over some $\iP_{n_k}$, and apply the finite dimensional result there to write $\omega_k= d(\eta_k)$ (for that we need the fibers of $\iP_{n_k}\to \iM_{n_k}$ to be cohomologically $l$-connected). Again, one has to make sure that the consecutive $\eta_k$s are chosen compatible, but that is done precisely as in the previous proof, just that the situation is a bit simpler now, because we just need vector bundle complements inside each $E_{k+1}$. 
\end{proof}

 \begin{exm}\label{exam-pair-groupoid}
The main examples of groupoids that we encounter in this paper there are groupoids of germs (finite dimensional but possibly huge and non-Hausdorff) and groupoids of jets. The second type sit inside the tower of full jet groupoids associated to (finite dimensional) manifolds $\BB$
\[ \Pi^k= \Pi^k(\BB):=  \{j^k_xf:\ f\in \textrm{Diff}_{\textrm{loc}}(\BB), x\in \textrm{domain}(f)\},\]
which are groupoids over $\BB$ with source $s(j^k_xf)=x$, target $t(j^k_xf)=f(x)$ and multiplication 
\[
j^k_{f(x)}g\cdot j^k_xf=j^k_x(f\circ g).
\]
For $k$ finite $\Pi^k$ inherits a smooth structure from the fact that it sits as an open in the $k$-jet space $J^k(\BB, \BB)$. All-together, one has a tower of Lie groupoids
\[ \Pi^\infty : \ldots \to \Pi^2 \to \Pi^1 \to \Pi^0 = \BB\times \BB\]
and $\Pi^{\infty}\tto B$ becomes a strict pf-groupoid. 

Corresponding to this there is a pf-vector bundle over $\BB$, 
\[ \TT^\infty : \ldots \to \TT^2 \to \TT^1 \to \TT^0 = T\BB,\]
where 
\[ \TT^k:= J^kT\BB\to \BB\]
consists of $k$-jets of vector fields. Each $\TT^k$ with $k$ finite carries a canonical structure of Lie algebroid over $\BB$- the unique one with anchor the projection $\pi= \pi_{k, 0}: \TT^k\to \TT^0= T\BB$ and with bracket satisfying
\[ [j^k X, j^k Y]= j^k[X, Y].\]
Of course, when $k =\infty$, the previous discussion makes $\TT^\infty$ into a strict pf-algebroid over $\BB$. Actually, this is precisely the Lie algebroid corresponding to the jet groupoids: the identification of $j^k_xX$ with a tangent vector to the $s$-fiber of $\Pi^k$, at the unit, is
\[ j^k_xX\mapsto \left. \frac{d}{dt}\right|_{t= 0} j^{k}_{x}(\phi^X_t).\]

The general discussion of Cartan distributions and forms (see Example \ref{exm:Car-distr-J} and \ref{ex:Cartan-form-R}), applied to the first projection $R= \BB\times \BB\to \BB$ 
gives rise to similar objects on $\Pi^{\infty}$. Note that the vertical bundle from Example \ref{exm:Car-distr-J} becomes the pull-back via $t$ of $\TT^{\infty}$; hence using our convention for denoting forms with coefficients in representations, the Cartan form becomes
\[ \omega_{\infty}\in \Omega^1(\Pi^{\infty}, \TT^{\infty})\]
(and, over $\Pi^k$, it becomes a form with values in $\TT^{k-1}$). 

The structure that emerges is related to the general discussion of flat Cartan groupoids from the main body of the paper, as $\Pi^{\infty}\tto \BB$ becomes a flat Cartan groupoid. The corresponding flat Cartan algebroid is $(\TT^{\infty}, \nabla)$ where $\nabla$ is the Spencer operator (\ref{eq:nabla-Cartan}) in the case $E= T\BB$. 
\end{exm}

\begin{exm}[Lie pseudogroups]\label{exam-Lie-pseudogroups} 
A similar discussion applies to general (Lie) pseudogroups on $\BB$ (the previous example corresponds to the full pseudogroup $\textrm{Diff}_{\textrm{loc}}(\BB)$). First of all, for any pseudogroup $\Gamma$ on $\BB$, restricting our attention at jets of diffeomorphisms from $\Gamma$, one obtains the tower of groupoids
\begin{equation}\label{eq:tower-JinftyGamma} 
J^{\infty}\Gamma: \ldots \to J^2\Gamma\to J^1\Gamma\to J^0\Gamma,
\end{equation}
where each $J^k\Gamma$ sits inside $\Pi^k$. 
The only ``problem" is that the groupoids $J^k\Gamma$ may fail to be smooth manifolds and the maps in the tower may fail to be submersions.
%smoothness is not guaranteed.

\begin{defn} A Lie pseudogroup on $\BB$ is any pseudogroup with the property that $J^{\infty}\Gamma$ is a smooth pf-subgroupoid of $\Pi^{\infty}$. 
\end{defn}

It follows that the corresponding algebroids
\[ A^k:= \textrm{Lie}(J^k\Gamma) \]
are Lie sub-algebroids of $\TT^k$, that $A^{\infty}$ becomes a pf-sub-algebroid of $\TT^{\infty}$ and that
the action of $\Pi^k$ on $\TT^{k-1}$ restricts to an action of $J^k\Gamma$ on $A^{k-1}$. Furthermore, the Cartan forms restrict to similar forms
\[ \omega_k\in \Omega^1(J^k\Gamma, t^*A^{k-1}), \quad \omega_{\infty}\in \Omega^1(J^\infty\Gamma, t^*A^{\infty}),\]
and the same happens for the Spencer operator on $\TT^\infty$.
\end{exm}

\begin{rmk}\label{Lie_exm} Given our convention that al the pf-manifolds are normal, the condition that $\Gamma$ is a Lie subgroupoid includes the condition $J^{\infty}\Gamma= \plim J^k\Gamma$. For this reason, it would be desirable to remove normality: this condition is often satisfied, but there are interesting examples when it is not - e.g. for the pseudogroup $\Gamma_{\rm an}^n$ of analytic diffeomorphisms on $\R^n$. In fact, $J^k\Gamma^n_{\rm an}=J^k\Gamma^n$, but $J^\infty\Gamma^n_{\rm an}$ is strictly contained into $J^\infty\Gamma^n$: there are formal power series that are not convergent. 

On the other hand, for Lie pseudogroups, $J^\infty\Gamma$ is not necessarily infinite dimensional; pseudogroups for which it has finite dimension are called {\bf of finite type}. For instance, consider the pseudogroup in $\R$ consisting of (locally defined) diffeomorphisms of type
\[ \phi: x\mapsto \frac{ax+ b}{cx+ d} ,\quad \textrm{with $a, b, c, d\in \R$ with $ad- bc\neq 0$}.\]
One has
\[ J^1\Gamma= \Pi^1,\quad J^2\Gamma= \Pi^2\]
and, for all $k\geq 2$, the projection $J^{k+1}\Gamma\to J^k\Gamma$ is an isomorphism. 
This 
follows from the fact that any diffeomorphism $\phi$ of the previous type satisfies 
\[ \phi^{'''}= \frac{3 (\phi^{''})^2}{2 \phi'},\]
hence all the higher derivatives can be expressed in terms of the first two. 
Notice that from this it also follows that
%the Cartan distribution on $J^2\Gamma$ is not involutive; however, 
the Cartan distribution on $J^3\Gamma$ is involutive. Moreover, the isomorphisms $J^{k+1}\Gamma\to J^k\Gamma$, for $k\geq 3$, induce isomorphisms of the Cartan distributions. This means that the projection 
\[
(J^\infty\Gamma, \omega^{\infty}) \to (J^3\Gamma, \omega^3)
\]
is an isomorphism of flat Cartan groupoids.
\end{rmk}

%%%%%%%%%%%%%%%%%%%%%%%%%%%%%%%%%%%%%%%%%%%%%%%%%%%%%%%%
%%%%%%%%%%%%%%%%%%%%%%%%%%%%%%%%%%%%%%%%%%%%%%%%%%%%%%%%
%%%%%%%%%%%%%%%%%%%%%%%%%%%%%%%%%%%%%%%%%%%%%%%%%%%%%%%%
%%%%%%%%%%%%%%%%%%%%%%%%%%%%%%%%%%%%%%%%%%%%%%%%%%%%%%%%

\bibliography{References}{}
\bibliographystyle{plain}

\end{document}